\documentclass[12pt,reqno,english]{amsart}
\usepackage[T1]{fontenc}
\usepackage[latin9]{inputenc}
\usepackage{geometry}
\usepackage{array}
\usepackage{float}
\usepackage{mathrsfs}
\usepackage{mathtools}
\usepackage{bm}
\usepackage{amsbsy}
\usepackage{amstext}
\usepackage{amsthm}
\usepackage{amssymb}
\usepackage{graphicx}
\usepackage{setspace}
\usepackage{esint}
\makeatletter

\providecommand{\tabularnewline}{\\}

\numberwithin{equation}{section}
\numberwithin{figure}{section}
\theoremstyle{plain}
\newtheorem{thm}{\protect\theoremname}[section]
\theoremstyle{plain}
\newtheorem{prop}[thm]{\protect\propositionname}
\theoremstyle{remark}
\newtheorem{notation}[thm]{\protect\notationname}
\theoremstyle{remark}
\newtheorem{rem}[thm]{\protect\remarkname}
\theoremstyle{plain}
\newtheorem{cor}[thm]{\protect\corollaryname}
\theoremstyle{definition}
\newtheorem{defn}[thm]{\protect\definitionname}
\theoremstyle{definition}
\newtheorem*{defn*}{\protect\definitionname}
\theoremstyle{plain}
\newtheorem{lem}[thm]{\protect\lemmaname}
\theoremstyle{remark}
\newtheorem*{rem*}{\protect\remarkname}
\theoremstyle{plain}
\newtheorem{lyxalgorithm}[thm]{\protect\algorithmname}
\theoremstyle{remark}
\newtheorem*{acknowledgement*}{\protect\acknowledgementname}

\usepackage{cite}

\usepackage{babel}
\usepackage{stmaryrd}
\usepackage{bbm}

\usepackage{enumitem}

\setlist[enumerate,1]{label={(\arabic*)}}

\@ifundefined{showcaptionsetup}{}{%
 \PassOptionsToPackage{caption=false}{subfig}}
\usepackage{subfig}
\makeatother

\usepackage{babel}
\providecommand{\acknowledgementname}{Acknowledgement}
\providecommand{\algorithmname}{Algorithm}
\providecommand{\corollaryname}{Corollary}
\providecommand{\definitionname}{Definition}
\providecommand{\lemmaname}{Lemma}
\providecommand{\notationname}{Notation}
\providecommand{\propositionname}{Proposition}
\providecommand{\remarkname}{Remark}
\providecommand{\theoremname}{Theorem}

\begin{document}
\title{Energy Landscape and Metastability of Curie--Weiss--Potts Model}
\author{Jungkyoung Lee}
\address{Department of Mathematical Sciences, Seoul National University, Seoul,
Republic of Korea}
\email{ljk9316@snu.ac.kr}
\begin{abstract}
In this paper, we thoroughly analyze the energy landscape of the Curie--Weiss--Potts
model, which is a ferromagnetic spin system consisting of $q\ge3$
spins defined on complete graphs. In particular, for the Curie--Weiss--Potts
model with $q\ge3$ spins and zero external field, we completely characterize
all critical temperatures and phase transitions in view of the global
structure of the energy landscape. We observe that there are three
critical temperatures and four different regimes for $q<5$, whereas
there are four critical temperatures and five different regimes for
$q\ge5$. Our analysis extends the investigations performed in {[}M.
Costeniuc, R. S. Ellis, H. Touchette: J. Math. Phys (2005){]}; they
provide the precise characterization of the second critical temperatures
for all $q\ge3$ and in {[}Landim and Seo: J. Stat. Phys. (2016){]},
which provides a complete analysis of the energy landscape for $q=3$.
Based on our precise analysis of the energy landscape, we also perform
a quantitative investigation of the metastable behavior of the heat-bath
Glauber dynamics associated with the Curie--Weiss--Potts model.
\end{abstract}

\maketitle

\section{Introduction}

The Potts model is a well-known mathematical model suitable for studying
ferromagnetic spin system consisting of $q\ge3$ spins. We refer to
\cite{Wu-The Potts model} a comprehensive review on the Potts model.
In the present work, we focus on the Potts model defined on large
complete graphs without an external field to understand the associated
energy landscape as well as the metastable behavior of the heat-bath
Glauber dynamics to the highly precise level. This special case of
the Potts model defined on complete graphs is called a \textit{Curie--Weiss--Potts
model }and investigated in various studies; e.g., \cite{Binachi Bovier,Costeniuc Ellis T,Cuff DLLPS,Eichelsbacher M,Ellis Wang,Kulske Meibner,Landim Seo-3spin,Wang,Bovier EGK}
and references therein. We note that the rigorous mathematical definition
of the Curie--Weiss--Potts model is presented in the next section.

\subsubsection*{The Curie--Weiss model}

The Ising case of the Curie--Weiss--Potts model, i.e., the corresponding
spin system consisting only of $q=2$ spins, is the famous Curie--Weiss
model. It is well-known that the Curie--Weiss model without an external
field exhibits a phase transition at the critical (inverse) temperature
$\beta_{c}>0$. It is mainly because the number of global minima of
the potential function associated with the empirical magnetization
is one for the high temperature regime $\beta\le\beta_{c}$ while
it becomes two for the low temperature regime $\beta>\beta_{c}$,
where $\beta>0$ represents the inverse temperature (cf. \cite[Chapter 9]{Rassoul Seppalainen}
for more detail). It is also well-known that such a phase transition
for the structure of the energy landscape is closely related to the
mixing property of the associated heat-bath Glauber dynamics. In \cite{Levin LP},
it has been shown that the Glauber dynamics exhibits the so-called
cut-off phenomenon which is a signature of the fast mixing for the
high-temperature regime (i.e., $\beta<\beta_{c}$) and the metastability
for the low-temperature regime (i.e., $\beta>\beta_{c}$). The metastability
for the low-temperature regime has been more deeply investigated in
\cite{Cassandro GOV}.

\subsubsection*{The Curie--Weiss--Potts model with $q=3$}

The picture for the Curie-Weiss model explained above has been fully
extended to the Curie--Weiss--Potts model consisting of $q=3$ spins.
The complete description of the energy landscape has been obtained
recently in \cite{Kulske Meibner,Landim Seo-3spin}, where three critical
temperatures
\[
0<\beta_{1}<\beta_{2}<\beta_{3}=3
\]
are characterized. More precisely, it has been shown that the potential
function associated with the empirical magnetization (which will be
explained in detail in section \ref{subsec: Magnetization}) has
\begin{itemize}
\item the unique global minimum for $\beta\in(0,\beta_{1})$,
\item one global minimum and three local minima for $\beta\in(\beta_{1},\beta_{2})$,
\item three global minima and one local minimum for $\beta\in(\beta_{2},\beta_{3})$,
and
\item three global minima for $\beta\in(\beta_{3},\infty)$.
\end{itemize}
The articles \cite{Kulske Meibner,Landim Seo-3spin} also analyzed
the associated saddle structure. Based on this analysis, \cite{Landim Seo-3spin}
discussed the quantitative feature of the metastable behavior of the
heat-bath Glauber dynamics in view of the Eyring--Kramers formula
and Markov chain model reduction (cf. \cite{Beltran Landim,Beltran Landim2,Landim MMC})
for all the low-temperature regime $\beta>\beta_{1}$. Because of
the abrupt change in the structure of the potential function at $\beta=\beta_{2}$
and $\beta=\beta_{3}$, the metastable behaviors of the Glauber dynamics
in three low-temperature regimes $(\beta_{1},\beta_{2})$, $(\beta_{2},\beta_{3})$,
and $(\beta_{3},\infty)$ turned out to be both quantitatively and
qualitatively different. For the high-temperature regime $(0,\,\beta_{1})$,
the cut-off phenomenon has been verified in \cite{Cuff DLLPS} for
all $q\ge3$. Adjoining all these works completes the picture for
the Curie--Weiss--Potts model with $q=3$ spins. 

\subsubsection*{The Curie--Weiss--Potts model with $q\ge4$}

Compared to the Curie--Weiss--Potts model with $q=2$ or $3$ spins,
the analysis of the case with $q\ge4$ spins is not completed so far.
In many literature, two critical temperatures $\beta_{1}(q)<\beta_{2}(q)$
for the Curie--Weiss--Potts model with $q\ge4$ spins are observed
and the phase transitions near these critical temperatures have been
analyzed. For instance, in \cite{Cuff DLLPS}, the phase transition
from the fast mixing (the cut-off phenomenon) to the slow mixing (due
to the appearance of new local minima) at $\beta=\beta_{1}(q)$ has
been confirmed. In \cite{Ellis Wang}, it has been observed that the
limiting distributions of the empirical magnetization exhibits the
abrupt change at $\beta=\beta_{2}(q)$. In \cite{Costeniuc Ellis T},
the phase transition around $\beta_{2}(q)$ also has been studied
in view of the equivalence and non-equivalence of ensembles.

These studies focus on the phase transitions involved with the local
and the global minima of the potential function. However, in order
to investigate the metastable behavior whose main objective is to
analyze the transitions between neighborhoods of local minima (i.e.,
the metastable states), the precise understanding of the \textit{saddle
structure} is also required. To the best of our knowledge, the analysis
of the saddle structure as well as the metastable behavior of the
heat-bath Glauber dynamics for $q\ge4$ has not been analyzed yet.

\subsubsection*{Main contribution of the article}

The main result of the present work is to provide the complete description
of the energy landscape including the saddle structure and to analyze
dynamical features of the Glauber dynamics based on it for the Curie--Weiss--Potts
models with $q\ge4$ spins.

First, we observe that for $q=4$, as in the case of $q=3$, the potential
function has three critical temperatures
\[
0<\beta_{1}(4)<\beta_{2}(4)<\beta_{3}(4)=4\ ,
\]
and moreover the associated metastable behavior is quite similar to
that of the case $q=3$. On the other hand, for $q\ge5$, we will
deduce that there are four critical temperatures
\[
0<\beta_{1}(q)<\beta_{2}(q)<\beta_{3}(q)<\beta_{4}(q)=q\ ,
\]
where two critical temperatures $\beta_{1}(q)$ and $\beta_{2}(q)$
play essentially the same role with $\beta_{1}(3)$ and $\beta_{2}(3)$
(and hence $\beta_{1}(4)$ and $\beta_{2}(4)$), respectively. Surprisingly,
our work reveals that the role of the third critical temperature $\beta_{3}(q)$
for $q\le4$ is divided into the third and fourth critical temperatures
$\beta_{3}(q)$ and $\beta_{4}(q)$ for $q\ge5$. More precisely,
for $q\le4$, the change in the saddle gates between global minima
and the disappearance of the local minimum representing the chaotic
configuration happen simultaneously at $\beta=\beta_{3}(q)=q$; however,
for $q\ge5$, the change of saddle gates happens at $\beta=\beta_{3}(q)<q$
and the disappearance of the chaotic local minimum occurs at $\beta=\beta_{4}(q)=q$.
Hence, for $q\ge5$, we observe another type of metastable behavior
at $\beta\in[\beta_{3}(q),\beta_{4}(q))$ compared to the case $q\le4$. 

\subsubsection*{Other studies on the Potts model}

Although the present work focuses on the Potts model on complete graphs,
we also note that the Ising and Potts models on the lattice are widely
studied as well. For instance, we refer to \cite{Rassoul Seppalainen}
and the references therein for the phase transition, to \cite{Lubetzky Sly,Lubetzky Sly2,Lubetzky Sly3}
for the cut-off phenomenon in the high-temperature regime, and to
\cite{Alonso Cerf,BenArous Cerf,Bovier H Metastability,Bovier H Nardi,Bovier H Spitoni,Bovier Manzo,Kim Seo,Nardi Zocca,Neves,Neves Schonmann,Olivieri Vares}
for the metastability in the low-temperature regime. In addition,
we refer to \cite{Griiths Pearce,Kesten Schonmann} for the Potts
model in many spins or large dimensions and to \cite{Bovier MP-RDCW,Hollander J-ERgraph}
for the study of metastability of the Ising model on random graphs.

\section{\label{sec: Model}Model}

In this section, we introduce the formal definition of the Curie--Weiss--Potts
model, which will be analyzed in the present work. Fix an integer
$q\ge3$ and let $S=\{1,\,\dots,\,q\}$ be the set of spins.

\subsection{Curie--Weiss--Potts Model}

For a positive integer $N$, let us denote by\footnote{We write $K_{N}$ to emphasize that our model is on the complete graph}
$K_{N}=\{1,\,\dots,\,N\}$ the set of sites. Let $\Omega_{N}=S^{K_{N}}$
be the configuration space of spins on $K_{N}$. Each configuration
is represented as $\sigma=(\sigma_{1},\,\dots,\,\sigma_{N})\in\Omega_{N}$
where $\sigma_{v}\in S$ denotes a spin at site $v\in K_{N}$. Let
$\bm{h}=(h_{1},\,\dots,\,h_{q})\in\mathbb{R}^{q}$ be the external
magnetic field. The Hamiltonian associated to the Curie--Weiss--Potts
model with the external field $\bm{h}$ is given by
\[
\mathbb{H}_{N}(\sigma)\,=\,-\frac{1}{2N}\sum_{1\le u,v\le N}\bm{1}(\sigma_{u}=\sigma_{v})\,-\,\sum_{v=1}^{N}\sum_{j=1}^{q}h_{j}\bm{1}(\sigma_{v}=j)\ \ ;\ \sigma\in\Omega_{N}\ ,
\]
where $\bm{1}$ denotes the usual indicator function. Then, the Gibbs
measure associated to the Hamiltonian at the (inverse) temperature
$\beta>0$ is given by
\[
\mu_{N}^{\beta}(\sigma)\,=\,\frac{1}{Z_{N}(\beta)}e^{-\beta\mathbb{H}_{N}(\sigma)}\ \ ;\ \sigma\in\Omega_{N}\ ,
\]
where $Z_{N}(\beta)=\sum_{\sigma\in\Omega_{N}}e^{-\beta\mathbb{H}_{N}(\sigma)}$
is the partition function. The measure $\mu_{N}^{\beta}(\cdot)$ denotes
the Curie--Weiss--Potts measure on $\Omega_{N}$ at the inverse
temperature $\beta$. 

\subsection{Heat-bath Glauber Dynamics}

Now, we define a heat-bath Glauber dynamics associated with the Curie--Weiss--Potts
measure $\mu_{N}^{\beta}(\cdot)$. For $\sigma\in\Omega_{N}$, $v\in K_{N}$,
and $k\in S$, denote by $\sigma^{v,\,k}$ the configuration whose
spin $\sigma_{v}$ at site $v$ is flipped to $k$, i.e.,
\[
(\sigma^{v,\,k})_{u}\,=\,\begin{cases}
\sigma_{u} & u\ne v\ ,\\
k & u=v\ .
\end{cases}
\]
Then, we will consider a heat-bath Glauber dynamics associated with
generator $\mathcal{L}_{N}$ which acts on $f:\Omega_{N}\to\mathbb{R}$
as
\[
(\mathcal{L}_{N}f)(\sigma)\,=\,\frac{1}{N}\,\sum_{v=1}^{N}\sum_{k=1}^{q}c_{v,\,k}(\sigma)[f(\sigma^{v,\,k})-f(\sigma)]\ ,
\]
where
\[
c_{v,\,k}(\sigma)\,=\,\exp\left\{ -\frac{\beta}{2}[\mathbb{H}_{N}(\sigma^{v,\,k})-\mathbb{H}_{N}(\sigma)]\right\} \ .
\]
It can be observed that this dynamics is reversible with respect to
the Curie--Weiss--Potts measure $\mu_{N}^{\beta}(\cdot)$. Henceforth,
denote by $\sigma(t)=\sigma^{K_{N}}(t)=(\sigma_{1}(t),\,\dots,\,\sigma_{N}(t))$
the continuous time Markov process associated with the generator $\mathcal{L}_{N}$.

\subsection{\label{subsec: Magnetization}Empirical Magnetization}

For each spin $k\in S$, denote by $r_{N}^{k}(\sigma)$ the proportion
of spin $k$ of configuration $\sigma\in\Omega_{N}$, i.e., 
\[
r_{N}^{k}(\sigma)\,:=\,\frac{1}{N}\sum_{v=1}^{N}\bm{1}(\sigma_{v}=k)\ ,
\]
and define the proportional vector $\bm{r}_{N}(\sigma)$ as 
\[
\bm{r}_{N}(\sigma)\,:=\,(r_{N}^{1}(\sigma),\,\dots,\,r_{N}^{q-1}(\sigma))\ ,
\]
which represents the \textit{empirical magnetization} of the configuration
$\sigma$ containing the macroscopic information of $\sigma$. 

Define $\Xi$ as
\begin{equation}
\Xi\,=\,\{\bm{x}=(x_{1},\,\dots,\,x_{q-1})\in(\mathbb{R}_{\ge0})^{q-1}:\,x_{1}+\cdots+x_{q-1}\le1\}\ ,\label{e: symplex}
\end{equation}
and then define a discretization of $\Xi$ as 
\[
\Xi_{N}=\Xi\cap(\mathbb{Z}/N)^{q-1}\ .
\]
With this notation, we immediately have $\bm{r}_{N}(\sigma)\in\Xi_{N}$
for $\sigma\in\Omega_{N}$. 

For the Markov process $\big(\sigma(t)\big)_{t\ge0}$, we write $\boldsymbol{r}_{N}(\cdot)=\boldsymbol{r}_{N}(\sigma(\cdot))$
which is a stochastic process on $\Xi_{N}$ expressing the evolution
of the empirical magnetization. Since the model is defined on the
complete graph $K_{N}$, we obtain the following proposition.
\begin{prop}
\label{p: r_N MC}The process $\big(\bm{r}_{N}(t)\big)_{t\ge0}$ is
a continuous time Markov chain on $\Xi_{N}$ whose invariant measure
is given by 
\[
\nu_{N}^{\beta}(\bm{x}):=\mu_{N}^{\beta}(\bm{r}_{N}^{-1}(\bm{x}))\ \ ;\ \boldsymbol{x}\in\Xi_{N}
\]
where $\bm{r}_{N}^{-1}(\bm{x})$ denotes the set $\{\sigma\in\Omega_{N}\,:\,\bm{r}_{N}(\sigma)=\bm{x}\}$.
Furthermore, $\bm{r}_{N}(\cdot)$ is reversible with respect to $\nu_{N}^{\beta}$.
\end{prop}

The proof of this proposition including jump rates is given in Section
\ref{subsec: Dyna of magnet}. Let $\mathbb{P}_{\bm{x}}^{N,\,\beta}$
be the law of the Markov chain $\bm{r}_{N}(\cdot)$ starting at $\bm{x}\in\Xi_{N}$
and let $\mathbb{E}_{\bm{x}}^{N,\,\beta}$ be the corresponding expectation. 

\subsubsection*{More on the measure $\nu_{N}^{\beta}(\cdot)$}

For $\bm{y}\in\Xi$, let $\widehat{\bm{y}}=(y,\,\dots,\,y_{q-1},\,y_{q})\in\mathbb{R}^{q}$
where $y_{q}=1-(y_{1}+\cdots+y_{q-1})$. Then, the Hamiltonian $\mathbb{H}_{N}$
can be written as
\[
\mathbb{H}_{N}(\sigma)\,=\,NH(\bm{r}_{N}(\sigma))\ \;;\;\sigma\in\Omega_{N}
\]
where 
\begin{equation}
H(\bm{x})\,=\,-\frac{1}{2}|\widehat{\bm{x}}|^{2}-\bm{h}\cdot\widehat{\bm{x}}\ \;;\;\boldsymbol{x}\in\Xi\;.\label{e: Hamil prop vec}
\end{equation}
Therefore, by Proposition \ref{p: r_N MC}, the invariant measure
$\nu_{N}^{\beta}(\cdot)$ of the process $\bm{r}_{N}(t)$ on $\Xi_{N}$
can be written as
\begin{align}
\nu_{N}^{\beta}(\bm{x}) & \,=\,\sum_{\sigma:\bm{r}_{N}(\sigma)=\bm{x}}\frac{1}{Z_{N}(\beta)}\exp\{-\beta\mathbb{H}_{N}(\sigma)\}\nonumber \\
 & =\,{N \choose (Nx_{1})\cdots(Nx_{q})}\frac{1}{Z_{N}(\beta)}\exp\{-\beta NH(\bm{x})\}\nonumber \\
 & \eqqcolon\,\frac{1}{(2\pi N)^{(q-1)/2}Z_{N}(\beta)}\exp\{-\beta NF_{\beta,\,N}(\bm{x})\}\ ,\label{e: def of F_beta,N}
\end{align}
where, by Stirling's formula, we can write
\[
F_{\beta,\,N}(\bm{x})\,=\,F_{\beta}(\bm{x})+\frac{1}{N}G_{\beta,\,N}(\bm{x})\ ,
\]
where
\begin{equation}
F_{\beta}(\bm{x})\,=\,H(\bm{x})+\frac{1}{\beta}S(\bm{x})\;\ \text{ and}\ \ G_{\beta,\,N}(\bm{x})\,=\,\frac{\log(x_{1}\cdots x_{q})}{2\beta}+O(N^{-1})\ .\label{e: def of F_beta}
\end{equation}
In this equation, $H(\cdot)$ is the energy functional defined in
\eqref{e: Hamil prop vec} and $S(\cdot)$ is the entropy functional
defined by
\[
S(\bm{x})=\sum_{i=1}^{q}x_{i}\log(x_{i})\ ,
\]
and $G_{\beta,\,N}(\bm{x})$ converges to $\log(x_{1}\cdots x_{q})/(2\beta)$
uniformly on every compact subsets of $\text{int}\,\Xi$.

\subsubsection*{Main objectives of the article}

Now, we can express the main purpose of the current article in a more
concrete manner. In this article, we consider the Curie--Weiss--Potts
model when there is no external magnetic field; i.e., $\bm{h}=\bm{0}$.
Therefore, from now on, we assume $\bm{h}=\bm{0}$. Under this assumption,
the first objective is to analyze the function $F_{\beta}(\cdot)$
expressing the energy landscape of the empirical magnetization of
the Curie--Weiss--Potts model. This result will be explained in
Section \ref{sec: Land Result}. The second concern is to investigate
the metastable behavior of the process $\bm{r}_{N}(\cdot)$ in the
low-temperature regime. This will be explained in Section \ref{sec: Meta Result}.
Latter part of the article is devoted to proofs of these results.

\section{\label{sec: Land Result}Main Result for Energy Landscape}

In view of Proposition \ref{p: r_N MC}, \eqref{e: def of F_beta,N},
and \eqref{e: def of F_beta}, the structure of the invariant measure
$\nu_{N}^{\beta}(\cdot)$ of the process $\bm{r}_{N}(\cdot)$ is essentially
captured by the potential function $F_{\beta}(\cdot)$; hence, the
investigation of $F_{\beta}(\cdot)$ is crucial in the analysis of
the energy landscape and the metastable behavior of $\bm{r}_{N}(\cdot)$.
In this section, we explain our detailed analysis of the function
$F_{\beta}(\cdot)$. 

Note that the function $F_{\beta}(\cdot)=H(\cdot)+\beta^{-1}S(\cdot)$
express the competition between the energy and the entropy represented
by $H(\cdot)$ and $S(\cdot)$, respectively. Since there is a $\beta^{-1}$
factor in front of the entropy functional, we can expect that the
entropy dominates the competition when $\beta$ is small (i.e., the
temperature is high). Since entropy is uniquely minimized at the equally
distributed configuration $(1/q,\,\dots,\,1/q)\in\Xi$, we can expect
that the potential $F_{\beta}(\cdot)$ also has the unique minimum
when $\beta$ is small. On the other hand, if $\beta$ is large enough
(i.e., the temperature is low), the energy $H(\cdot)$ with $q$ minima
dominates the system, and therefore, we can expect that the potential
$F_{\beta}$ also has $q$ global minima. In this section we provide
the complete characterization of the complicated pattern of transition
from this high-temperature regime to low-temperature regime in a precise
level. 

In Section \ref{subsec: Main cri}, we define several points that
will be shown to be critical points. In Section \ref{subsec: Main critemp},
we introduce several critical values of (inverse) temperature $\beta$.
In Section \ref{subsec: EL}, we summarize the results on the energy
landscape $F_{\beta}(\cdot)$. In Section \ref{subsec: MF energy},
as a by-product of these results, we compute the mean-field free energy.

\subsection{\label{subsec: Main cri}Critical Points of $F_{\beta}(\cdot)$}

Let us first investigate critical points of $F_{\beta}(\cdot)$. We
recall that
\[
F_{\beta}(\bm{x})\,=\,-\frac{1}{2}\sum_{k=1}^{q}x_{k}^{2}\,+\,\frac{1}{\beta}\sum_{k=1}^{q}x_{k}\log x_{k}\;\ ;\;\boldsymbol{x}\in\Xi\ .
\]

\begin{notation}
\label{nota: symplex}We have following notations for convenience.
\begin{enumerate}
\item Since there is no risk of confusion, we will write the point $\bm{x}=(x_{1},\,\dots,\,x_{q-1})\in\Xi$
as $\bm{x}=(x_{1},\,\dots,\,x_{q-1},\,x_{q})\in\mathbb{R}^{q}$ where
$x_{q}=1-x_{1}-\cdots-x_{q-1}$.
\item Let $\{\bm{e}_{1},\,\dots,\,\bm{e}_{q-1}\}$ be the orthonormal basis
of $\mathbb{R}^{q-1}$ and $\bm{e}_{q}=\bm{0}\in\mathbb{R}^{q-1}$.
\end{enumerate}
\end{notation}

\begin{figure}
\includegraphics{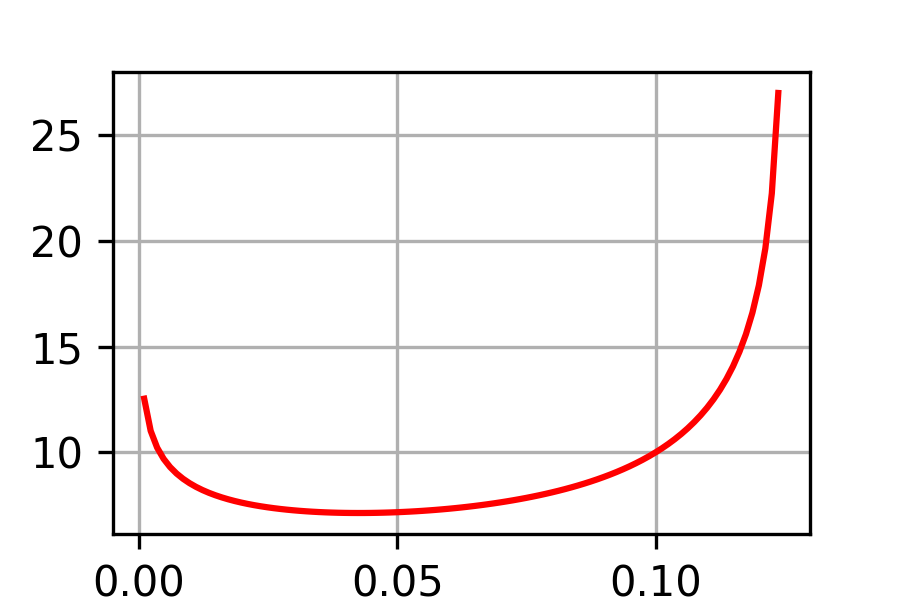}\caption{\label{fig: g2}Graph of $g_{2}(t)$ for $q=10$.}
\end{figure}

Now, we explain the candidates for the critical points of $F_{\beta}(\cdot)$
playing important role in the analysis of the energy landscape. The
first candidate is
\[
{\bf p}\coloneqq(1/q,\,\dots,\,1/q)\in\Xi\ ,
\]
which represents the state where the spins are equally distributed.

In order to introduce the other candidates, we fix $i\in\mathbb{N}\cap[1,q/2]$
and let $j=q-i$. Define $g_{i}:(0,1/j)\to\mathbb{R}$ as
\begin{equation}
g_{i}(t)\,\coloneqq\,\frac{i}{1-qt}\log\Big(\frac{1-jt}{it}\Big)\ ,\label{e: Def of g_i}
\end{equation}
where we set $g_{i}(1/q)=q$ so that $g_{i}$ becomes a continuous
function on $(0,1/j)$. We refer to Figure \ref{fig: g2} for an illustration
of graph of $g_{i}$. Then, it will be verified by Lemma \ref{l: g_i}
in Section \ref{subsec: class cri} (and we can expect from the graph
illustrated in Figure \ref{fig: g2}) that $g_{i}(t)=\beta$ has at
most two solutions. We denote by $u_{i}(\beta)\le v_{i}(\beta)$ these
solutions, provided that they exist. If there is only one solution,
we let $u_{i}(\beta)=v_{i}(\beta)$ be this solution.

For $k\in S$, let 
\begin{align}
{\bf u}_{1}^{k}=\mathbf{u}_{1}^{k}(\beta) & \coloneqq\,\Big(u_{1}(\beta),\,\dots,\,1-(q-1)u_{1}(\beta),\,\dots,\,u_{1}(\beta)\Big)\in\Xi\ ,\label{e: Def of U_1^k}\\
{\bf v}_{1}^{k}=\mathbf{v}_{1}^{k}(\beta) & \coloneqq\,\Big(v_{1}(\beta),\,\dots,\,1-(q-1)v_{1}(\beta),\,\dots,\,v_{1}(\beta)\Big)\in\Xi\ ,\label{e: Def of V_1^k}
\end{align}
where $1-(q-1)u_{1}(\beta)$ and $1-(q-1)v_{1}(\beta)$ are located
at the $k$-th component of ${\bf u}_{1}^{k}$ and ${\bf v}_{1}^{k}$,
respectively. For\footnote{Henceforth, $a,\,b\in S$ implies that $a\in S$, $b\in S$, and $a\ne b$.}
$k,\,l\in S$, let
\begin{align}
{\bf u}_{2}^{k,\,l}=\mathbf{u}_{2}^{k,\,l}(\beta) & \coloneqq\,\Big(u_{2}(\beta),\,\dots,\,\frac{1-(q-2)u_{2}(\beta)}{2},\,\dots,\label{e: Def of U_2^kl}\\
 & \ \ \ \ \ \ \ \ \ \ \ \dots,\,\frac{1-(q-2)u_{2}(\beta)}{2},\,\dots,\,u_{2}(\beta)\Big)\in\Xi\ ,\nonumber 
\end{align}
where $\frac{1-(q-2)u_{2}(\beta)}{2}$ is located at the $k$-th and
$l$-th components. Of course, each of these points is well defined
only when $u_{1}(\beta)$,$v_{1}(\beta)$, or $u_{2}(\beta)$ exists,
respectively. Then, let
\[
\mathcal{U}_{1}:=\{{\bf u}_{1}^{k}:k\in S\}\,,\;\mathcal{U}_{2}\coloneqq\,\{{\bf u}_{2}^{k,\,l}:k,\,l\in S\}\,,\;\text{and\;}\mathcal{V}_{1}\coloneqq\,\{{\bf v}_{1}^{k}:k\in S\}\ .
\]
We remark that these sets depend on $\beta$ although we omit $\beta$
in the expressions for the simplicity of the notation. 

Since we assumed that $\bm{h}=\bm{0}$, by symmetry, we can expect
that the elements in $\mathcal{U}_{1}$ have the same properties;
for instance, for all $k,l\in S$, we have $F_{\beta}({\bf u}_{1}^{k})=F_{\beta}({\bf u}_{1}^{l})$,
and ${\bf u}_{1}^{k}$ is a critical point of $F_{\beta}(\cdot)$
if and only if ${\bf u}_{1}^{l}$ is. Of course the elements in $\mathcal{U}_{2}$
or $\mathcal{V}_{1}$ respectively have the same properties. Thus,
it suffices to analyze their representatives, and hence select these
representatives as 
\begin{equation}
{\bf u}_{1}={\bf u}_{1}^{q}\,,\ {\bf u}_{2}={\bf u}_{2}^{q-1,\,q}\,,\ \text{and}\ {\bf v}_{1}={\bf v}_{1}^{q}\ .\label{eq: simnot}
\end{equation}

Now, we have the following preliminary classification of critical
points. We remark that a saddle point is a critical point at which
the Hessian has only one negative eigenvalue.
\begin{prop}
\label{p: class cri}The following hold.
\begin{enumerate}
\item If ${\bf c}\in\Xi$ is a local minimum of $F_{\beta}$, then ${\bf c}\in\{{\bf p}\}\cup\mathcal{U}_{1}$. 
\item If ${\bf s}\in\Xi$ is a saddle point of $F_{\beta}$, then ${\bf s}\in\mathcal{V}_{1}\cup\mathcal{U}_{2}$
for $q\ge4$ and ${\bf s}\in\mathcal{V}_{1}$ for $q=3$.
\end{enumerate}
\end{prop}

\begin{rem}
\label{remU2}The set $\mathcal{U}_{2}$ is not defined for $q=3$
since the set $\mathcal{U}_{i}$ is defined only when $i\le q/2$.
This will be explained in Section \ref{subsec: class cri}.
\end{rem}

The proof of this proposition is an immediate consequence of Proposition
\ref{p: cri} in Section \ref{subsec: class cri}. The above proposition
permits us to focus only on $\{\mathbf{p}\}\cup\mathcal{U}_{1}\cup\mathcal{U}_{2}\cup\mathcal{V}_{1}$
when we analyze the energy landscape in view of the metastable behavior,
since the critical points of index greater than 1 cannot play any
role, as the metastable transition always happens at the neighborhood
of a saddle point (a critical point of index $1$). 

\subsection{\label{subsec: Main critemp}Critical Temperatures}

In this subsection, we introduce critical temperatures 
\[
0<\beta_{1}(q)<\beta_{2}(q)<\beta_{3}(q)\le q\ ,
\]
at which the phase transitions in the energy landscape occur. The
precise definition of these critical temperatures are given in \eqref{e: Def of cri temp}
of Section \ref{subsec: pre beta}. Henceforth, we write $\beta_{i}=\beta_{i}(q)$,
$1\le i\le3$, since there is no risk of confusion.

\begin{figure}
\subfloat[$q=4$]{\includegraphics[scale=0.27]{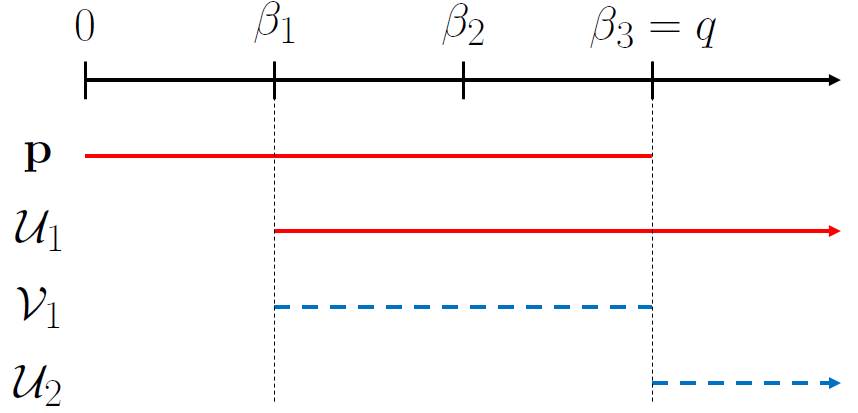}}\ ~\subfloat[$q\ge5$]{\includegraphics[scale=0.27]{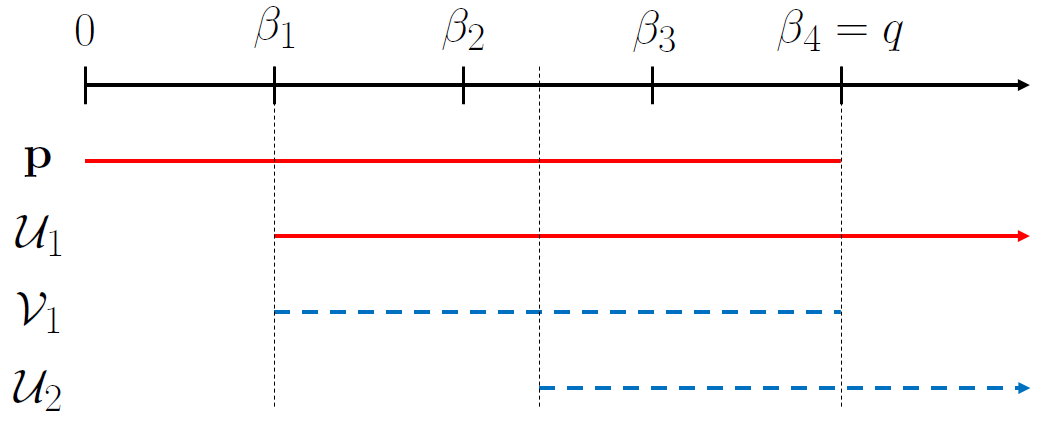}

}

\caption{\label{fig: cri pts}Role of each critical point according to temperature.
Solid lines imply local minima and dashed lines imply saddle points.}

\end{figure}

To describe the role of these critical temperatures, we regard $\beta$
as increasing from $0$ to $\infty$. Figure \ref{fig: cri pts} shows
the role of ${\bf p},\,\mathcal{U}_{1},\,\mathcal{V}_{1}$, and $\mathcal{U}_{2}$
according to inverse temperature. Section \ref{sec: Pre cri} will
prove this figure.

At $\beta=\beta_{1}$, the dynamics exhibits phase transition from
fast mixing to slow mixing, and this is proven in \cite{Cuff DLLPS}.
Furthermore, the behavior of the dynamics changes from cutoff phenomenon
to metastability. This phase transition is due to the appearance of
new local minima $\mathcal{U}_{1}$ of $F_{\beta}(\cdot)$ other than
${\bf p}$ at $\beta=\beta_{1}$. At $\beta=\beta_{2}$, the ground
states of dynamics change from ${\bf p}$ to elements of $\mathcal{U}_{1}$,
as observed in \cite[Theorem 3.1(b)]{Costeniuc Ellis T}. To explain
the role of critical temperatures $\beta_{3}$ and $q$, we have to
divide the explanation into several cases. Let us first assume that
$q\ge5$ so that $\beta_{3}<q$. At $\beta=\beta_{3}$, the saddle
gates among the ground states in $\mathcal{U}_{1}$ is changed from
$\mathcal{V}_{1}$ to $\mathcal{U}_{2}$ (since the heights $F_{\beta}({\bf v}_{1})$
and $F_{\beta}({\bf u}_{2})$ are reversed at this point) and at $\beta=q$,
the local minimum $\mathbf{p}$ becomes a local maximum. On the other
hand, for $q\le4$, we have $\beta_{3}=q$. At $\beta=\beta_{3}$,
the change of the saddle gates and the disappearance of the local
minimum $\mathbf{p}$ occur simultaneously. We refer to \cite{Landim Seo-3spin}
for the detailed description when $q=3$. 

\subsection{\label{subsec: EL}Stable and Metastable Sets}

We define some metastable sets based on the results explained earlier.
If $q\ge4$, define $H_{\beta}$ as (cf. \eqref{eq: simnot})
\begin{equation}
H_{\beta}=\begin{cases}
F_{\beta}({\bf v}_{1})\ , & \beta\in(\beta_{1},\beta_{3})\ ,\\
F_{\beta}({\bf u}_{2})\ , & \beta\in[\beta_{3},\infty)\ .
\end{cases}\label{e: height of saddles}
\end{equation}
When $q=3,$ we set $H_{\beta}=F_{\beta}({\bf v}_{1})$ for all $\beta>\beta_{1}$
(cf. Remark \ref{remU2}). It will be verified in Lemma \ref{l: beta_m}
and \eqref{e: Def of cri temp} that $H_{\beta}$ is the height of
the lowest saddle points.

Let $\widehat{S}:=S\cup\{\frak{o}\}$ and ${\bf u}_{1}^{\frak{o}}:={\bf p}$.
Let\footnote{We define the set $\mathcal{W}_{k}$, $k\in S$, and $\mathcal{W}_{\frak{o}}$
as the empty set if the set $\{F_{\beta}<H_{\beta}\}$ does not contain
$\mathbf{u}_{1}^{k}$ and $\{F_{\beta}<F_{\beta}({\bf v}_{1})\}$
does not contain ${\bf u}_{1}^{\frak{o}}$ respectively.} $\mathcal{W}_{k}=\mathcal{W}_{k}(\beta)$, $k\in S$, be the connected
component of $\{F_{\beta}<H_{\beta}\}$ containing ${\bf u}_{1}^{k}$
and let $\mathcal{W}_{\frak{o}}=\mathcal{W}_{\frak{o}}(\beta)$ be
the connected component of $\{F_{\beta}<F_{\beta}({\bf v}_{1})\}$
containing ${\bf u}_{1}^{\frak{o}}$. For $k,\,l\in\widehat{S}$,
let $\Sigma_{k,\,l}=\Sigma_{k,\,l}(\beta):=\overline{\mathcal{W}_{k}}\cap\overline{\mathcal{W}_{l}}$
be a set of saddle gates of height $H_{\beta}$ between $\mathbf{u}_{1}^{k}$
and $\mathbf{u}_{1}^{l}$.

\begin{figure}
\subfloat[{$\beta\in(0,\beta_{1}]$}]{\includegraphics[width=2in,height=1.2in]{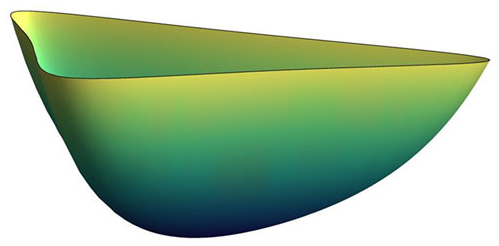}

}\subfloat[$\beta\in(\beta_{1},\beta_{2})$]{\includegraphics[width=2in,height=1.2in]{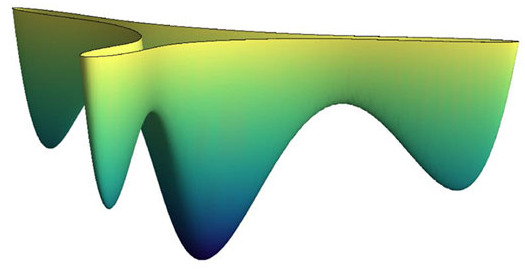}}\subfloat[$\beta=\beta_{2}$]{\includegraphics[width=2in,height=1.2in]{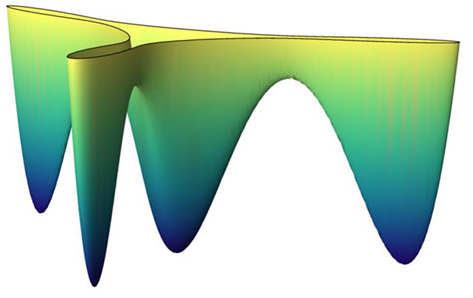}}

\subfloat[$\beta\in(\beta_{2},\beta_{3})$]{\includegraphics[width=2in,height=1.2in]{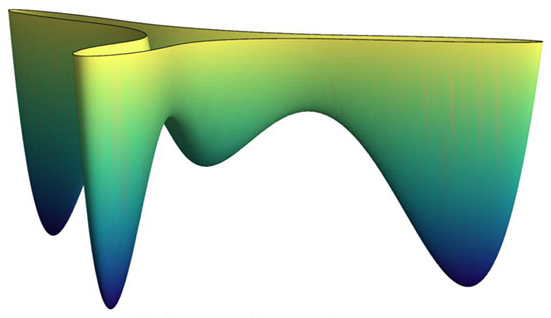}}\subfloat[$\beta=\beta_{3}$]{\includegraphics[width=2in,height=1.2in]{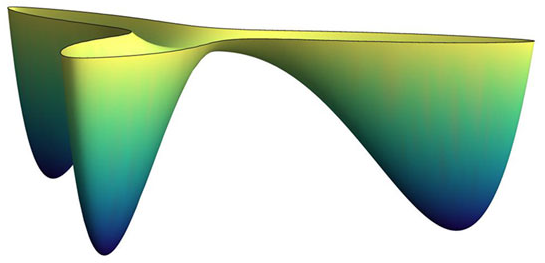}}\subfloat[$\beta\in(\beta_{3},\infty)$]{\includegraphics[width=2in,height=1.2in]{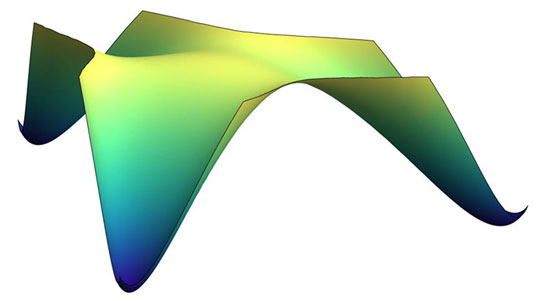}}

\caption{\label{fig: metasets q3}Energy landscape of $F_{\beta}$ when $q=3$.}
\end{figure}

Now, we can state the main result on energy landscape and the proofs
of theorems in this section will presented in Section \ref{sec: Pf metastates}.
The first result holds for all $q\ge3$.
\begin{thm}
\label{t: metasets all q}For $q\ge3$, the following hold.
\begin{enumerate}
\item If $\beta\le\beta_{1}$, there is no critical point other than ${\bf p}$,
which is the global minimum.
\item For $\beta\in(\beta_{1},q)$, we have $\mathcal{W}_{\frak{o}}\ne\emptyset$
and for $\beta\in[q,\infty)$, we have $\mathcal{W}_{\frak{o}}=\emptyset$.
\item Let $\mathcal{M}_{\beta}$ be a set of local minima of $F_{\beta}$.
Then, we have
\[
\mathcal{M}_{\beta}=\begin{cases}
\{{\bf p}\} & \beta\in(0,\beta_{1}]\ ,\\
\{{\bf p}\}\cup\mathcal{U}_{1} & \beta\in(\beta_{1},q)\ ,\\
\mathcal{U}_{1} & \beta\in[q,\infty)\ .
\end{cases}
\]
\item Let $\mathcal{M}_{\beta}^{\star}$ be a set of global minima of $F_{\beta}$.
Then, we have
\[
\mathcal{M}_{\beta}^{\star}=\begin{cases}
\{{\bf p}\} & \beta\in(0,\beta_{2})\ ,\\
\{{\bf p}\}\cup\mathcal{U}_{1} & \beta=\beta_{2}\ ,\\
\mathcal{U}_{1} & \beta\in(\beta_{2},\infty)\ .
\end{cases}
\]
\end{enumerate}
\end{thm}

Since there is only one minimum if $\beta\le\beta_{1}$, we now consider
$\beta>\beta_{1}$. Before we write the main result on metastable
sets, we would like to emphasize that \cite[Proposition 4.4]{Landim Seo-3spin}
proved the case when $q=3$, while the proof for the case $q\ge4$
is the main novel contents of the current article. We first consider
the case $q\le4$. See Figure \ref{fig: metasets q3}\footnote{This figures are excerpt from \cite[Fig 4]{Landim Seo-3spin}}
for the visualization of the following and above theorem.
\begin{thm}
\label{t: metasets q4}For $q\le4$, the following hold.
\begin{enumerate}
\item $\beta_{3}=q$.
\item For $\beta\in(\beta_{1},q)$, the sets $\mathcal{W}_{k}$, $k\in\widehat{S}$,
are nonempty and disjoint. For $k,\,l\in S$, $\Sigma_{k,\,l}=\emptyset$
and for $k\in S$, $\Sigma_{\frak{o},\,k}=\{{\bf v}_{1}^{k}\}$.
\item For $\beta=q$, we have $\mathcal{W}_{\frak{o}}=\emptyset$. The sets
$\mathcal{W}_{k}$, $k\in S$, are nonempty and disjoint. For $k,\,l\in S$,
$\Sigma_{k,\,l}=\{{\bf p}\}$.
\item For $\beta\in(q,\infty)$, we have $\mathcal{W}_{\frak{o}}=\emptyset$.
The sets $\mathcal{W}_{k}$, $k\in S$, are nonempty and disjoint.
For $k,\,l\in S$, 
\[
\Sigma_{k,\,l}=\begin{cases}
\{{\bf v}_{1}^{m}\}\ ,\ \text{where }m\in S\setminus\{k,l\}\ , & \text{if }q=3\ ,\\
\{{\bf u}_{2}^{k,\,l}\}\ , & \text{if }q=4\ .
\end{cases}
\]
\end{enumerate}
\end{thm}

\begin{figure}
\subfloat[$(\beta_{1},\beta_{2})$]{\includegraphics{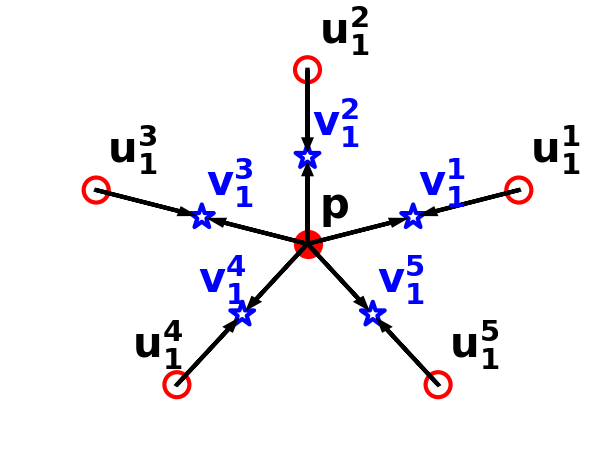}}\subfloat[$\beta=\beta_{2}$]{\includegraphics{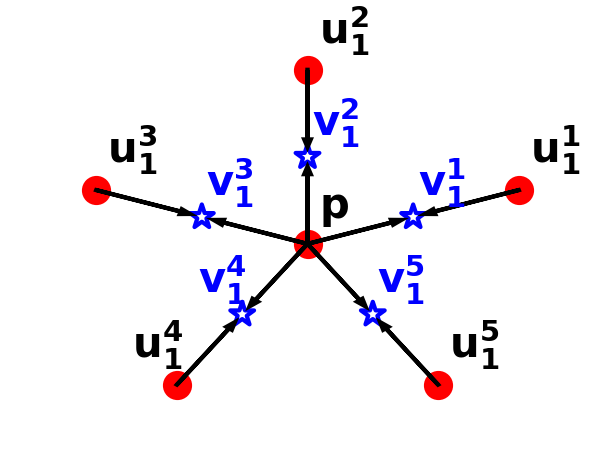}}\subfloat[$(\beta_{2},\beta_{3})$]{\includegraphics{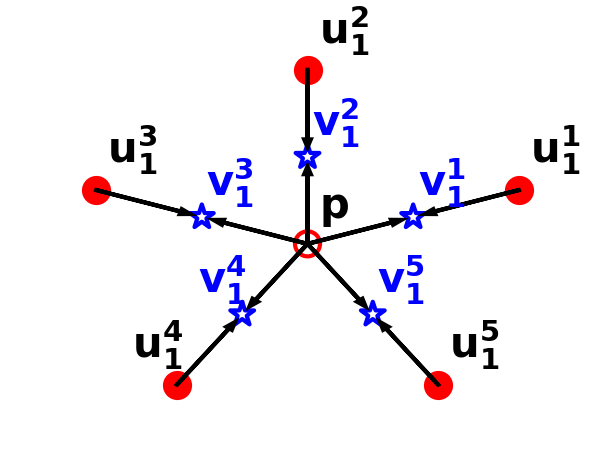}}

\subfloat[$\beta=\beta_{3}$]{\includegraphics{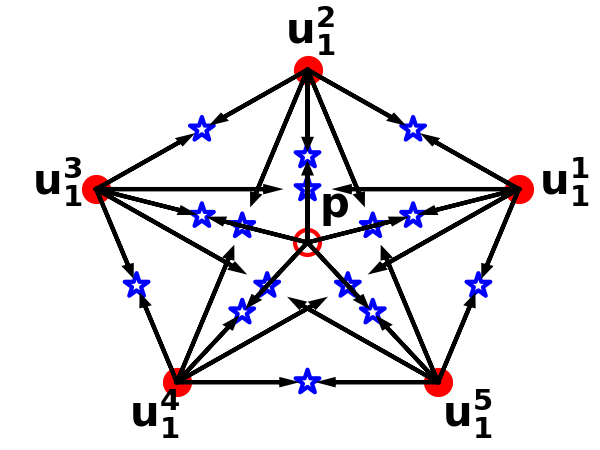}}\subfloat[$(\beta_{3},\infty)$]{\includegraphics{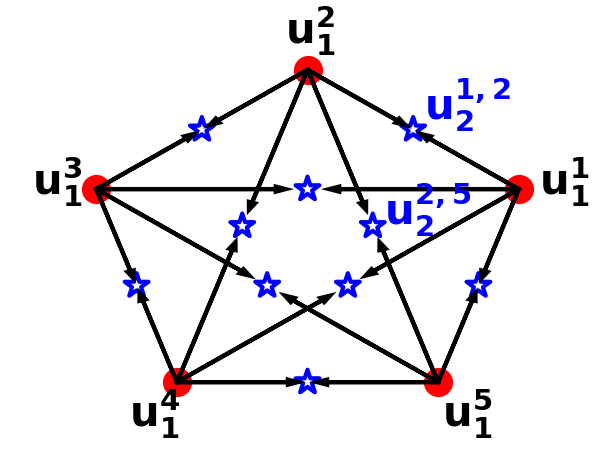}}\subfloat[$(\beta_{3},q)$]{\includegraphics{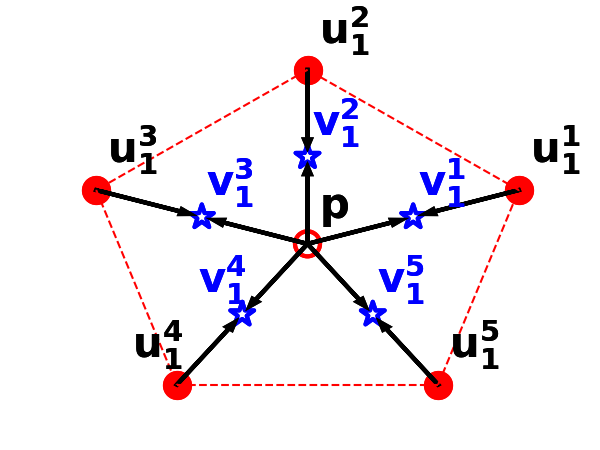}}

\caption{\label{fig: metasets q5}Illustration of energy landscape of $F_{\beta}$
when $q=5$. The first five figures are $\{F_{\beta}\le H_{\beta}\}$
and the last figure is $\{F_{\beta}\le F_{\beta}({\bf v}_{1})\}$.
The star-shaped vertices and circles represent saddle points and local
minima, respectively. The empty circles are shallower minima.}
\end{figure}

Next, we consider the case $q\ge5$. Note that the crucial difference
compared to the previous theorem lies in the third and fifth statements.
See Figure \ref{fig: metasets q5} for the visualization of the following
theorem and Theorem \ref{t: metasets all q}.
\begin{thm}
\label{t: metasets}For $q\ge5$, the following hold.
\begin{enumerate}
\item $\beta_{3}<q$.
\item For $\beta\in(\beta_{1},\beta_{3})$, the sets $\mathcal{W}_{k}$,
$k\in\widehat{S}$, are nonempty and disjoint. For $k,\,l\in S$,
$\Sigma_{k,\,l}=\emptyset$ and for $k\in S$, $\Sigma_{\frak{o},\,k}=\{{\bf v}_{1}^{k}\}$
\item For $\beta=\beta_{3}$, the sets $\mathcal{W}_{k}$, $k\in\widehat{S}$,
are nonempty and disjoint. For $k,\,l\in S$, $\Sigma_{k,\,l}=\{{\bf u}_{2}^{k,\,l}\}$
and for $k\in S$, $\Sigma_{\frak{o},\,k}=\{{\bf v}_{1}^{k}\}$.
\item For $\beta\in(\beta_{3},\infty)$, the sets $\mathcal{W}_{k}$, $k\in S$,
are nonempty and disjoint. For $k,\,l\in S$, $\Sigma_{k,\,l}=\{{\bf u}_{2}^{k,\,l}\}$
and for $k\in S$, $\Sigma_{\frak{o},\,k}=\emptyset$.
\item For $\beta\in(\beta_{3},q)$, we have $F_{\beta}({\bf v}_{1})>H_{\beta}$.
Furthermore, the set $\{F_{\beta}<F_{\beta}({\bf v}_{1})\}$ has only
two connected components, the well $\mathcal{W}_{\frak{o}}$ and the
other containing $\mathcal{U}_{1}$. The saddle points between them
are $\mathcal{V}_{1}$.
\end{enumerate}
\end{thm}

\subsection{\label{subsec: MF energy}Mean-field Free Energy}

In this subsection, we compute the mean-field free energy of the Curie--Weiss--Potts
model defined by 
\begin{equation}
\psi(\beta)\coloneqq-\lim_{N\to\infty}\frac{1}{\beta N}\log Z_{N}(\beta)\ .\label{e: MF free energy}
\end{equation}
It is well known that the Curie--Weiss model with $q=2$ spins exhibits
the second-order phase transition at the unique critical temperature
$\beta=\beta_{c}$, while the Curie--Weiss--Potts model with $q\ge3$
spins exhibits the first-order phase transition at $\beta=\beta_{2}$
(cf. \cite{Costeniuc Ellis T,Ellis Wang,Ostilli Mukhamedov}). We
now reconfirm this folklore by computing the free energy explicitly.
This computation is based on the following observation (cf. \cite[display (2.4)]{Ellis Wang}):
\begin{equation}
\lim_{N\to\infty}\frac{1}{\beta N}\log Z_{N}(\beta)=\sup_{\bm{x}\in\Xi}\{-F_{\beta}(\bm{x})\}\ .\label{e_mfe}
\end{equation}
We give a rigorous proof in Appendix B. 

Now, let us assume that $q\ge3$ so that by \eqref{e: MF free energy},
\eqref{e_mfe}, and Theorems \ref{t: metasets all q}, we can deduce
that 
\begin{equation}
\psi(\beta)=\begin{cases}
\,F_{\beta}({\bf p}) & \text{if }\beta\le\beta_{2}\ ,\\
\,F_{\beta}({\bf u}_{1}) & \text{if }\beta>\beta_{2}\ .
\end{cases}\label{mfe1}
\end{equation}

\begin{cor}
\label{c: 2nd PT}We have that
\begin{equation}
\psi'(\beta)=\begin{cases}
\,-\frac{1}{\beta^{2}}S({\bf p}) & \text{if }\beta<\beta_{2}\ ,\\
\,-\frac{1}{\beta^{2}}S({\bf u}_{1}) & \text{if }\beta>\beta_{2}\ .
\end{cases}\label{mfe2}
\end{equation}
In particular, the Curie--Weiss--Potts model with $q\ge3$ exhibits
the first-order phase transition at $\beta=\beta_{2}$.
\end{cor}

\begin{proof}
Let ${\bf c}(\beta)\in\Xi$ be a critical point of $F_{\beta}(\cdot)$.
Then, since $F_{\beta}=H+\beta^{-1}S$, we have 
\[
\frac{d}{d\beta}F_{\beta}({\bf c}(\beta))=\nabla F_{\beta}({\bf c}(\beta))\cdot\dot{{\bf c}}(\beta)-\frac{1}{\beta^{2}}S({\bf c}(\beta))\ .
\]
Since $\nabla F_{\beta}({\bf c}(\beta))=0$, we get \eqref{mfe2}
from \eqref{mfe1}. Since $S$ attains its unique local minimum at
$\mathbf{p}$ and $\mathbf{u}_{1}\neq\mathbf{p}$, $\psi'(\cdot)$
is discontinuous at $\beta=\beta_{2}$.
\end{proof}

\section{\label{sec: Meta Result}Main Result for Metastability}

In this section, we analyze the metastable behavior of $\bm{r}_{N}(\cdot)$
based on the analysis of the energy landscape carried out in the previous
section and the general results obtained by \cite{Landim Seo-nonrev RW}.
As inverse temperature $\beta$ varies, the behavior of this dynamics
changes both qualitatively and quantitatively thanks to the structural
phase transitions explained in the previous section. 

Since the invariant measure $\nu_{N}^{\beta}$ is exponentially concentrated
in neighborhoods of ground states, the corresponding Markov process
$\bm{r}_{N}(\cdot)$ stays most of the time at these neighborhoods.
The abrupt transitions between such stable states are the metastable
behavior of the process $\boldsymbol{r}_{N}(\cdot)$ and one of the
natural ways of describing these hopping dynamics among the neighborhoods
of the ground states is the\textit{ Markov chain model reduction}.
A comprehensive understanding of such approaches can be obtained from
\cite{Beltran Landim,Beltran Landim2,Landim MMC}.

When the dynamics starts from a local minimum which is not a global
minimum, we have to estimate the mean of the transition time to the
global minimum in order to quantitatively understand the metastable
behavior. This estimation is known as the \textit{Eyring--Kramers
formula}. In this section we provide the Markov chain model reduction
and Eyring--Kramers formula for the metastable process $\boldsymbol{r}_{N}(\cdot)$. 

Such a metastable behavior is observed only when there are multiple
local minima; and hence we cannot expect metastable behavior at the
high-temperature regime $\beta\le\beta_{1}$ for which ${\bf p}$
is the unique local (and global) minimum. Hence, we assume $\beta>\beta_{1}$
in this section.

\subsection{\label{subsec: Def metastabiltiy}Some preliminaries}

In this subsection, we introduce several notions crucial to the description
of the metastable behavior. 

\subsubsection*{Some constants }

Recall the definition of $\{\bm{e}_{1},\,\dots,\,\bm{e}_{q}\}$ from
Notation \ref{nota: symplex}. Define $(q-1)\times(q-1)$ matrices
$\mathbb{A}^{i,\,j}$, $i,j\in S$, and $\mathbb{A}(\bm{x})$ as
\[
\mathbb{A}^{i,\,j}\,=\,(\bm{e}_{j}-\bm{e}_{i})(\bm{e}_{j}-\bm{e}_{i})^{\dagger}\ \;\;\;\text{and \;\;\;\;}\mathbb{A}(\bm{x})=\sum_{1\le i<j\le q}w^{i,\,j}(\bm{x})\mathbb{A}^{i,\,j}\ ,
\]
where $w^{i,\,j}(\bm{x})\,:=\,\sqrt{x_{i}x_{j}}$ . The appearance
of the weight $w^{i,\,j}(\cdot)$ is explained in Section \ref{subsec: Special LS}.
Since $\mathbb{A}^{i,\,j}$, $i,j\in S$, are positive definite, $\mathbb{A}(\bm{x})$
satisfies \cite[display (A.1)]{Landim Seo-nonrev RW} and hence, by
\cite[Lemma A.1]{Landim Seo-nonrev RW}, for all $k,l\in S$, the
matrices $(\nabla^{2}F_{\beta})({\bf u}_{2}^{k,\,l})\mathbb{A}({\bf u}_{2}^{k,\,l})^{\dagger}$
and $(\nabla^{2}F_{\beta})({\bf v}_{1}^{k})\mathbb{A}({\bf v}_{1}^{k})^{\dagger}$
have the unique negative eigenvalue which will be denoted respectively
by $-\mu_{k,\,l}=-\mu_{k,\,l}(\beta)$ and $-\mu_{\frak{o},\,k}=-\mu_{\frak{o},\,k}(\beta)$. 

Let us define the so-called \textit{Eyring--Kramers constant}s corresponding
to our model as
\begin{align}
\omega_{k,\,l}=\omega_{k,\,l}(\beta) & :=\frac{\mu_{k,\,l}(\beta)}{\sqrt{-\det[(\nabla^{2}F_{\beta})({\bf u}_{2}^{k,\,l})]}}e^{-\beta G_{\beta}({\bf u}_{2}^{k,\,l})}\ ,\ \ \ k,\,l\in S\ ,\label{e: Def omega1}\\
\omega_{\frak{o},\,k}=\omega_{\frak{o},\,k}(\beta) & :=\frac{\mu_{\frak{o},\,k}(\beta)}{\sqrt{-\det[(\nabla^{2}F_{\beta})({\bf v}_{1}^{k})]}}e^{-\beta G_{\beta}({\bf v}_{1}^{k})}\ ,\ \ \ k\in S\ .\label{e: Def omega0}
\end{align}
By symmetry, we have $\omega_{k,\,l}=\omega_{k',\,l'}$ for all $k,\,l\in S$
and $k',\,l'\in S$ and $\omega_{o,\,k}=\omega_{\frak{o},\,k'}$ for
all $k,\,k'\in S$. Hence, let us write $\omega_{\frak{o}}=\omega_{\frak{o},\,1}$
and $\omega_{1}=\omega_{1,\,2}$. Next, define
\begin{align}
\nu_{k}=\nu_{k}(\beta)\, & :=\,\frac{\exp(-\beta G_{\beta}({\bf u}_{1}^{k}))}{\sqrt{\beta^{2}\det[(\nabla^{2}F_{\beta})({\bf u}_{1}^{k})]}}\ ,\ \ \ k\in S\ ,\label{e: Def nu1}\\
\nu_{\frak{o}}=\nu_{\frak{o}}(\beta)\, & :=\,\frac{\exp(-\beta G_{\beta}({\bf p}))}{\sqrt{\beta^{2}\det[(\nabla^{2}F_{\beta})({\bf p})]}}\ .\label{e: Def nu0}
\end{align}
By the symmetry, we also obtain $\nu_{1}=\cdots=\nu_{q}$.

\subsubsection*{Time scales}

The constant $H_{\beta}$ defined in \eqref{e: height of saddles}
denotes the height of the lowest saddle points. Let $\theta_{k}=\theta_{k}(\beta)$,
$k\in\widehat{S}$, be the depth of well $\mathcal{W}_{k}(\beta)$,
i.e., 
\[
\begin{cases}
\,\theta_{1}=\cdots=\theta_{q}=\beta[H_{\beta}-F_{\beta}({\bf u}_{1})]\ ,\\
\,\theta_{\frak{o}}=\beta[F_{\beta}({\bf v}_{1})-F_{\beta}({\bf p})]\ .
\end{cases}
\]
Then, $e^{N\theta_{1}}$ and $e^{N\theta_{\frak{o}}}$ represent the
time scales on which $\bm{r}_{N}(\cdot)$ exhibits metastability.
For $\beta\ge q$, the constant $\theta_{\frak{o}}$ is meaningless
since $\mathcal{W}_{\frak{o}}=\emptyset$.

\subsubsection*{Order process and Markov chain model reduction}

Let $\delta=\delta(\beta)>0$ be a small enough number such that $\delta<\min\{\theta_{\frak{o}},\theta_{1}\}$.
If $\beta\ge q$, since $\theta_{\frak{o}}$ is not defined, let $\delta=(1/2)\theta_{1}$.
For $k\in S$, define
\begin{align*}
\mathcal{W}_{k}^{\delta} & =\mathcal{W}_{k}\cap\{\bm{x}\in\Xi\,:\,F_{\beta}(\bm{x})<H_{\beta}-\delta\}\ ,\\
\mathcal{W}_{\frak{o}}^{\delta} & =\mathcal{W}_{\frak{o}}\cap\{\bm{x}\in\Xi\,:\,F_{\beta}(\bm{x})<F_{\beta}({\bf v}_{1})-\delta\}\ .
\end{align*}
For $k\in\widehat{S}$, define $\mathcal{E}_{N}^{k}=\mathcal{E}_{N}^{k}(\beta)$
as
\[
\mathcal{E}_{N}^{k}=\mathcal{W}_{k}^{\delta}\cap\Xi_{N}\ .
\]
This set $\mathcal{E}_{N}^{k}$ is called the metastable set, provided
that it is not an empty set. For $A\subset\widehat{S}$, we write
\[
\mathcal{E}_{N}^{A}=\bigcup_{k\in A}\mathcal{E}_{N}^{k}\ .
\]
Let $T$ be $S,\,\widehat{S},$ or $\{\frak{o},\,S\}$. Denote by
$\Psi_{N}=\Psi_{N}^{\beta}:\Xi_{N}\to T\cup\{N\}$ the projection
map given by
\[
\Psi_{N}(\bm{x})\,=\,\sum_{k\in T}k\bm{1}\{\bm{x}\in\mathcal{E}_{N}^{k}\}+N\bm{1}\{\bm{x}\in\Xi_{N}\setminus\mathcal{E}_{N}^{T}\}\ .
\]
Let us define the so-called \textit{order process} by ${\bf X}_{N}(t)=\Psi_{N}(\bm{r}_{N}(t))$
which represents the index of metastable set at which the process
$\boldsymbol{r}_{N}(\cdot)$ is staying.
\begin{defn}[Markov chain model reduction]
\label{def: Lim Chain}Let ${\bf X}(\cdot)$ be a continuous time
Markov chain on $T$. We say that the \textit{metastable behavior
of the process $\bm{r}_{N}(\cdot)$ is described by a Markov Process
${\bf X}(\cdot)$ in the time scale $\theta_{N}$} if, for all $k\in T$
and for all sequence $(\bm{x}_{N})_{N\ge1}$ such that $\bm{x}_{N}\in\mathcal{E}_{N}^{k}$
for all $N\ge1$, the finite dimensional marginals of the process
${\bf X}_{N}(\theta_{N}\cdot)$ under $\mathbb{P}_{\bm{x}_{N}}^{N,\,\beta}$
converges to that of the Markov chain $\mathbf{X}(\cdot)$ as $N\to\infty$. 
\end{defn}

In the previous definition, it is clear that the Markov chain $\mathbf{X}(\cdot)$
describes the inter-valley dynamics of the process $\boldsymbol{r}_{N}(\cdot)$
accelerated by a factor of $\theta_{N}$. 

\subsection{\label{subsec: Metastability q4}Metastability Results for $q\le4$}

We can now state the main result for the metastable behavior. First,
we consider the case $q\le4$ whose result is essentially the same
as that in \cite[Section 4.3]{Landim Seo-3spin} where only the case
$q=3$ was considered. 

We define limiting Markov chains when $q\le4$.
\begin{defn*}
Let $q\le4$ and $i\in\{\,(1,2),\,(2),\,(2,3),\,(3,\infty),\,(4)\,\}$.
Let ${\bf Y}_{q}^{i}(\cdot)$ be a Markov chain on $T$ with jump
rate $r_{q}^{i}:T\times T\to\mathbb{R}$ given by
\begin{enumerate}
\item $r_{q}^{(1,2)}(k,l)=\bm{1}\{l=\frak{o}\}\frac{\omega_{\frak{o}}}{\nu_{1}}\,,\,T=\widehat{S}$.
\item $r_{q}^{(2)}(k,l)=\bm{1}\{l=\frak{o}\}\frac{\omega_{\frak{o}}}{\nu_{1}}+\bm{1}\{k=\frak{o}\}\frac{\omega_{\frak{o}}}{\nu_{\frak{o}}}\,,\,T=\widehat{S}$.
\item $r_{q}^{(2,3)}(k,l)=\frac{\omega_{\frak{o}}}{q\nu_{1}}\,,\,T=S$.
\item $r_{q}^{(3,\infty)}(k,l)=\begin{cases}
\,\frac{\omega_{0}}{\nu_{1}}\ , & q=3\\
\,\frac{\omega_{1}}{\nu_{1}}\ , & q=4
\end{cases}\,,\,T=S$.
\item $r_{q}^{(4)}(k,l)=\bm{1}\{k=\frak{o}\}\frac{\omega_{\frak{o}}}{\nu_{\frak{o}}}\,,\,T=\widehat{S}$.
\end{enumerate}
\end{defn*}
The following theorem is the metastability result for $q\le4$. We
remark that the case when $q=3$ is proven in \cite[Section 4.3]{Landim Seo-3spin}.
\begin{thm}
\label{t: metaresult q4}Let $q\le4$. Then, the metastable behavior
of the process $\bm{r}_{N}(\cdot)$ is described by (cf. Definition
\ref{def: Lim Chain})
\begin{enumerate}
\item $\beta\in(\beta_{1},\beta_{2})$: the process ${\bf Y}_{q}^{(1,2)}(\cdot)$
in the time scale $2\pi Ne^{\theta_{1}}$.
\item $\beta=\beta_{2}$: the process ${\bf Y}_{q}^{(2)}(\cdot)$ in the
time scale $2\pi Ne^{\theta_{1}}$.
\item $\beta\in(\beta_{2},q)$: the process ${\bf Y}_{q}^{(2,3)}(\cdot)$
in the time scale $2\pi Ne^{\theta_{1}}$ and by the process ${\bf Y}_{q}^{(4)}(\cdot)$
in the time scale $2\pi Ne^{\theta_{\frak{o}}}$.
\item $\beta\in(q,\infty)$: the process ${\bf Y}_{q}^{(3,\infty)}(\cdot)$
in the time scale $2\pi Ne^{\theta_{1}}$.
\end{enumerate}
\end{thm}

The proof follows from Theorems \ref{t: metasets all q} and \ref{t: metasets q4},
Proposition \ref{p: EK constants}, and \cite[Theorem 2.2, Remark 2.10, 2.11]{Landim Seo-nonrev RW}.
\begin{rem}
As mentioned in \cite{Landim Seo-3spin}, we cannot investigate the
case $\beta=\beta_{3}=q$ with the current method since ${\bf p}$
is a degenerate saddle point.
\end{rem}

\begin{rem}
\label{rem: q=00003D4,3}The qualitative feature of the metastable
behavior of $\bm{r}_{N}(\cdot)$ is essentially same for $q=3$ and
$q=4$. The only difference is that the saddle points between metastable
sets are defined in different ways; however, when $\beta>\beta_{3}$,
the points in $\mathcal{V}_{1}$ for $q=3$ and $\mathcal{U}_{2}$
for $q=4$ play the same role since all the points belonging to these
sets represent states in which most sites are aligned to two spins
equally.
\end{rem}

\subsection{\label{subsec: Metastability q5}Metastability Results for $q\ge5$}

As in the previous subsection, we start by defining limiting Markov
chains. Note that there are two different Markov chains.
\begin{defn*}
Let $q\ge5$ and $i\in\{\,(1,2),\,(2),\,(2,3),\,(3),\,(3,\infty),\,(4),\,(5)\,\}$.
Let ${\bf Y}_{q}^{i}(\cdot)$ be a Markov chain on $T$ with jump
rate $r_{q}^{i}:T\times T\to\mathbb{R}$ with jump rate $r_{q}^{i}:T\times T\to\mathbb{R}$
given by
\begin{enumerate}
\item $r_{q}^{(1,2)}(k,l)=\bm{1}\{l=\frak{o}\}\frac{\omega_{\frak{o}}}{\nu_{1}}\,,\,T=\widehat{S}$.
\item $r_{q}^{(2)}(k,l)=\bm{1}\{l=\frak{o}\}\frac{\omega_{\frak{o}}}{\nu_{1}}+\bm{1}\{k=\frak{o}\}\frac{\omega_{\frak{o}}}{\nu_{\frak{o}}}\,,\,T=\widehat{S}$.
\item $r_{q}^{(2,3)}(k,l)=\frac{\omega_{\frak{o}}}{q\nu_{1}}\,,\,T=S$.
\item $r_{q}^{(3)}(k,l)=\frac{1}{\nu_{1}}(\frac{\omega_{\frak{o}}}{q}+\omega_{1})\,,\,T=S$.
\item $r_{q}^{(3,\infty)}(k,l)=\frac{\omega_{1}}{\nu_{1}}\,,\,T=S$.
\item $r_{q}^{(4)}(k,l)=\bm{1}\{k=\frak{o}\}\frac{\omega_{\frak{o}}}{\nu_{\frak{o}}}\,,\,T=\widehat{S}$.
\item $r_{q}^{(5)}(k,l)=\bm{1}\{k=\frak{o}\}\frac{q\omega_{\frak{o}}}{\nu_{\frak{o}}}\,,\,T=\{\frak{o},\,S\}$.
\end{enumerate}
\end{defn*}
Now, we present the metastability result for $q\ge5$. The new metastable
behaviors are observed when $\beta=\beta_{3}$ and $\beta\in(\beta_{3},q)$.
\begin{thm}
\label{t: metaresult}Let $q\ge5$. Then, the metastable behavior
of the process $\bm{r}_{N}(\cdot)$ is described by
\begin{enumerate}
\item $\beta\in(\beta_{1},\beta_{2})$: the process ${\bf Y}_{q}^{(1,2)}(\cdot)$
in the time scale $2\pi Ne^{\theta_{1}}$.
\item $\beta=\beta_{2}$: the process ${\bf Y}_{q}^{(2)}(\cdot)$ in the
time scale $2\pi Ne^{\theta_{1}}$.
\item $\beta\in(\beta_{2},\beta_{3})$: the process ${\bf Y}_{q}^{(2,3)}(\cdot)$
in the time scale $2\pi Ne^{\theta_{1}}$ and by the process ${\bf Y}_{q}^{(4)}(\cdot)$
in the time scale $2\pi Ne^{\theta_{\frak{o}}}$.
\item $\beta=\beta_{3}$: the process ${\bf Y}_{q}^{(3)}(\cdot)$ in the
time scale $2\pi Ne^{\theta_{1}}$ and by the process ${\bf Y}_{q}^{(4)}(\cdot)$
in the time scale $2\pi Ne^{\theta_{\frak{o}}}$.
\item $\beta\in(\beta_{3},q)$: the process ${\bf Y}_{q}^{(3,\infty)}(\cdot)$
in the time scale $2\pi Ne^{\theta_{1}}$and by the process ${\bf Y}_{q}^{(5)}(\cdot)$
in the time scale $2\pi Ne^{\theta_{\frak{o}}}$.
\item $\beta\in[q,\infty)$: the process ${\bf Y}_{q}^{(3,\infty)}(\cdot)$
in the time scale $2\pi Ne^{\theta_{1}}$.
\end{enumerate}
\end{thm}

The proof follows from Theorems \ref{t: metasets all q} and \ref{t: metasets},
Proposition \ref{p: EK constants}, and \cite[Theorem 2.2, Remarks 2.10, 2.11]{Landim Seo-nonrev RW}.
\begin{rem}
Notably, in contrast to the case $q\le4$, we can describe the metastable
behavior for all $\beta\in(\beta_{1},\infty)$ since the saddle points
are nondegenerate when $\beta=\beta_{3}$.
\end{rem}

\begin{figure}
\subfloat[${\bf Y}_{q}^{(1,2)}$]{\includegraphics[scale=0.75]{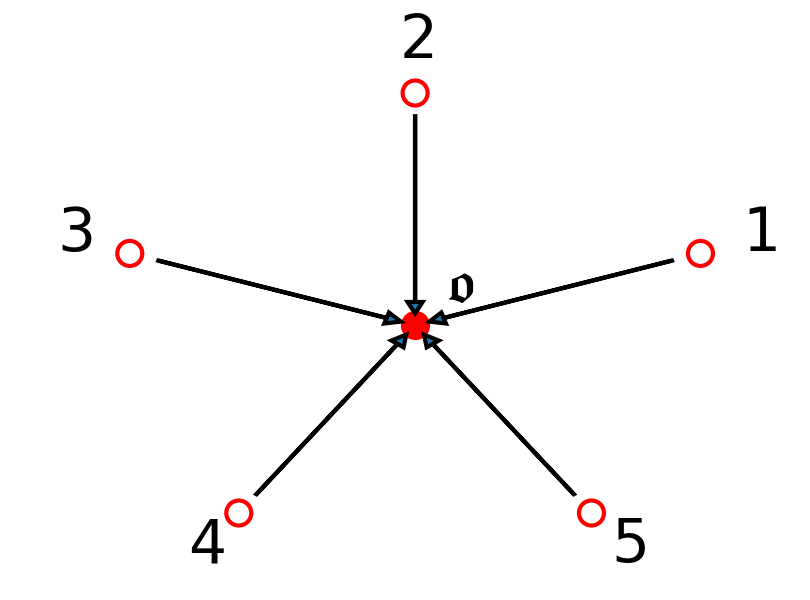}}\subfloat[${\bf Y}_{q}^{(2)}$]{\includegraphics[scale=0.75]{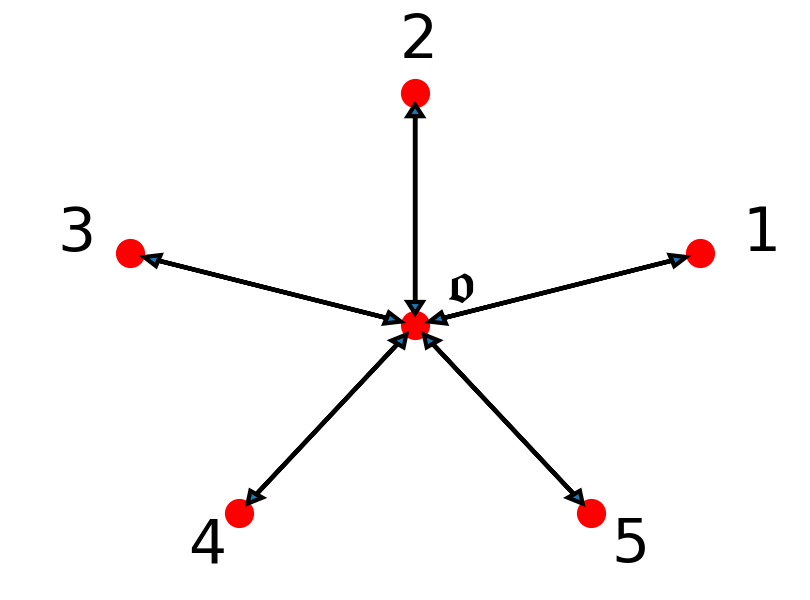}}\subfloat[${\bf Y}_{q}^{(2,3)}$]{\includegraphics[scale=0.75]{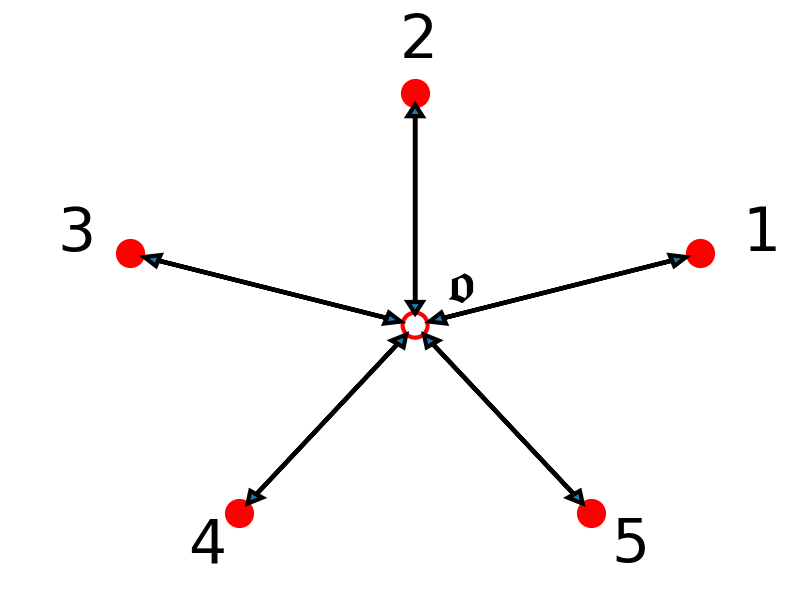}}

\subfloat[${\bf Y}_{q}^{(3)}$]{\includegraphics[scale=0.75]{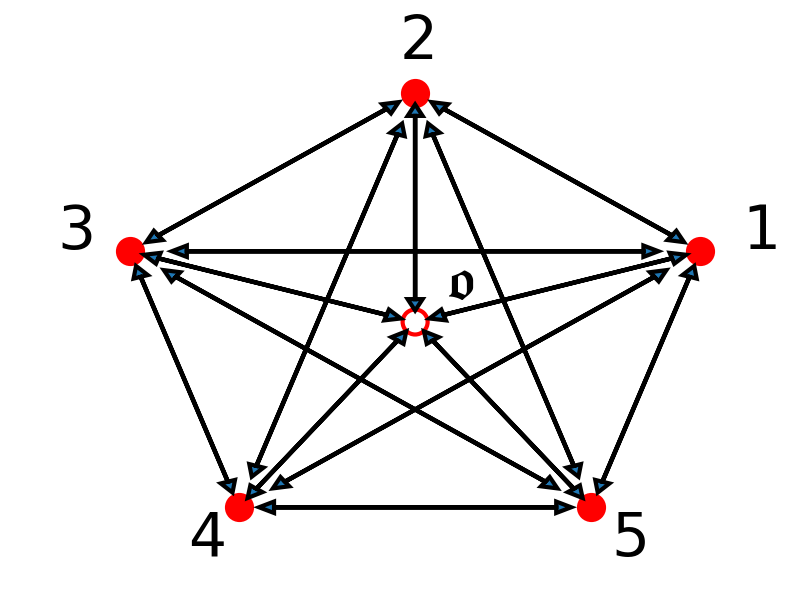}}\subfloat[${\bf Y}_{q}^{(3,\infty)}$]{\includegraphics[scale=0.75]{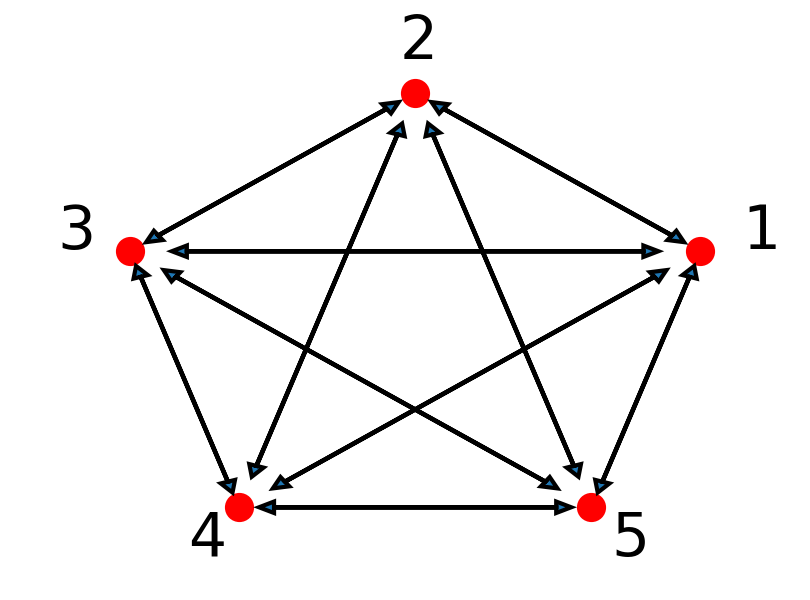}}\subfloat[${\bf Y}_{q}^{(4)}$]{\includegraphics[scale=0.75]{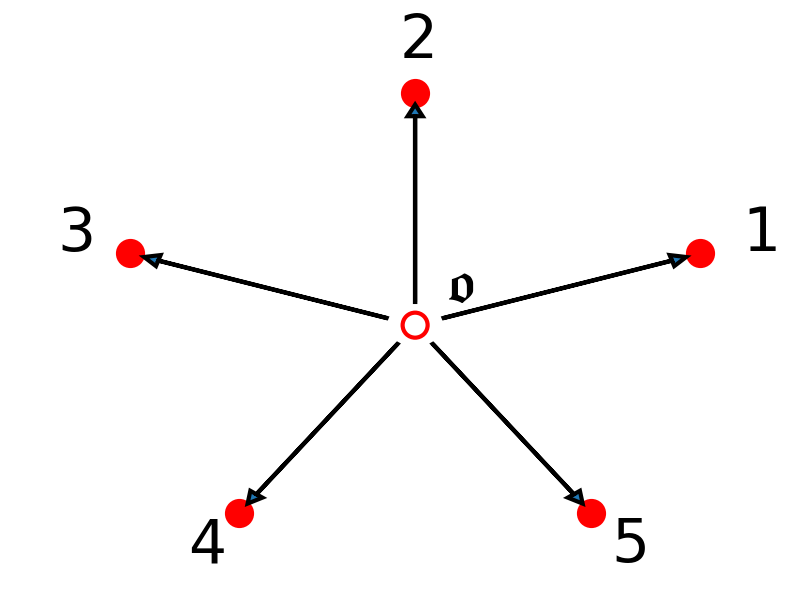}}

\caption{\label{fig: metastability graph}Figure about metastability when $q=5$}
\end{figure}

We can now provide a more intuitive explanation of Theorem \ref{t: metaresult}.
See Figure \ref{fig: metastability graph} also for the description
of metastable behavior. Note that if $\beta_{2}<\beta<q$, there are
two time scales.
\begin{itemize}
\item ${\bf Y}_{q}^{(1,2)}$: If $\beta_{1}<\beta<\beta_{2}$, in the time
scale $2\pi Ne^{\theta_{1}}$, $\bm{r}_{N}(\cdot)$ starting from
$\mathcal{E}_{N}^{S}$, reaches $\mathcal{E}_{N}^{\frak{o}}$ and
stays there forever. When it goes from $\mathcal{E}_{N}^{k}$, $k\in S$,
to $\mathcal{E}_{N}^{\frak{o}}$, it visits the neighborhood of ${\bf v}_{1}^{k}$.
\item ${\bf Y}_{q}^{(2)}$: If $\beta=\beta_{2}$, in the time scale $2\pi Ne^{\theta_{1}}$,
the process $\bm{r}_{N}(\cdot)$ goes around each well in $\mathcal{E}_{N}^{\widehat{S}}$.
However, when $\bm{r}_{N}(\cdot)$ starting from $\mathcal{E}_{N}^{k}$,
$k\in S$, goes to $\mathcal{E}_{N}^{l}$, $l\in S\setminus\{k\}$,
it must pass through $\mathcal{E}_{N}^{\frak{o}}$ and as in the case
$\beta\in(\beta_{1},\beta_{2})$, it visits the neighborhood of ${\bf v}_{1}^{k}$
and then the neighborhood of ${\bf v}_{1}^{l}$.
\item ${\bf Y}_{q}^{(2,3)}$: If $\beta_{2}<\beta<\beta_{3}$, in the time
scale $2\pi Ne^{\theta_{1}}$, the process $\bm{r}_{N}(\cdot)$ starting
from $\mathcal{E}_{N}^{S}$ travels $\mathcal{E}_{N}^{\widehat{S}}$,
however, it stays in $\mathcal{E}_{N}^{\frak{o}}$ in negligible time.
Furthermore, when $\bm{r}_{N}(\cdot)$ goes from $\mathcal{E}_{N}^{k}$,
$k\in S$, to $\mathcal{E}_{N}^{l}$, $l\in S\setminus\{k\}$, it
must visit $\mathcal{E}_{N}^{\frak{o}}$.
\item ${\bf Y}_{q}^{(3)}$: If $\beta=\beta_{3}$, in the time scale $2\pi Ne^{\theta_{1}}$,
the process $\bm{r}_{N}(\cdot)$ starting from $\mathcal{E}_{N}^{S}$
travels $\mathcal{E}_{N}^{\widehat{S}}$, however, it stays in $\mathcal{E}_{N}^{\frak{o}}$
in negligible time. Furthermore, there are two ways in which $\bm{r}_{N}(\cdot)$
goes from $\mathcal{E}_{N}^{k}$, $k\in S$, to $\mathcal{E}_{N}^{l}$,
$l\in S\setminus\{k\}$. First, it goes to $\mathcal{E}_{N}^{l}$
directly and must pass through the neighborhood of ${\bf u}_{2}^{k,\,l}$.
Second, it visits $\mathcal{E}_{N}^{\frak{o}}$ and stays there for
a negligible period of time.
\item ${\bf Y}_{q}^{(3,\infty)}$: If $\beta>\beta_{3}$, in the time scale
$2\pi Ne^{\theta_{1}},$ the process $\bm{r}_{N}(\cdot)$ starting
from $\mathcal{E}_{N}^{S}$ travels $\mathcal{E}_{N}^{S}$ without
visiting $\mathcal{E}_{N}^{\frak{o}}$. As in the case $\beta=\beta_{3}$,
it must pass through the neighborhood of ${\bf u}_{2}^{k\,l}$, $k,l\in S$,
when it goes from $\mathcal{E}_{N}^{k}$ to $\mathcal{E}_{N}^{l}$.
\item ${\bf Y}_{q}^{(4)}$: If $\beta_{2}<\beta\le\beta_{3}$, in the second
time scale $2\pi Ne^{\theta_{\frak{o}}}$, the process $\bm{r}_{N}(\cdot)$
starting from $\mathcal{E}_{N}^{\frak{o}}$, goes to $\mathcal{E}_{N}^{k}$,
$k\in S$, and stays there forever. As ${\bf Y}_{q}^{(1,2)}$, ${\bf Y}_{q}^{(2)}$,
and ${\bf Y}_{q}^{(2,3)}$, it passes through the neighborhood of
${\bf v}_{1}^{k}$ when it moves to $\mathcal{E}_{N}^{k}$ from $\mathcal{E}_{N}^{\frak{o}}$.
\item ${\bf Y}_{q}^{(5)}$: If $\beta_{3}<\beta<q$, in the second time
scale $2\pi Ne^{\theta_{\frak{o}}}$, the process $\bm{r}_{N}(\cdot)$
starting from $\mathcal{E}_{N}^{\frak{o}}$, goes to $\mathcal{E}_{N}^{S}$
and stays there forever. This dynamics is similar to ${\bf Y}_{q}^{(4)}$;
however, $\mathcal{E}_{N}^{k}$, $k\in S$, are not distinguishable.
\end{itemize}

\subsection{Eyring--Kramers Formulae}

In this subsection, we present the Eyring--Kramers formula with regard
to metastable behavior. Before we state the result, we introduce some
notations. Let $[\bm{x}]_{N}$ be the nearest point in $\Xi_{N}$
of $\bm{x}\in\Xi$. If there is more than one such point, one of them
is chosen arbitrarily. For $\mathcal{A}\subset\Xi$, define $[\mathcal{A}]_{N}$
as
\[
[\mathcal{A}]_{N}\,=\,\{[\bm{x}]_{N}\,:\,\bm{x}\in\mathcal{A}\}\ .
\]
Denote by $H_{\mathcal{A}}$ the hitting time of the set $[\mathcal{A}]_{N}$
by the process $\bm{r}_{N}(\cdot)$:
\[
H_{\mathcal{A}}\,=\,\inf\{t>0\,:\,\bm{r}_{N}(t)\in[\mathcal{A}]_{N}\}\ .
\]
If $\mathcal{A}=\{\bm{x}\}$, we simply write $H_{\mathcal{A}}=H_{\bm{x}}$.

We have the following theorem.
\begin{thm}
\label{t: EK For}Let $q\ge3$. We have the following.
\begin{enumerate}
\item For $\beta_{1}<\beta\le\beta_{2}$ and $k\in S$, we have
\[
\mathbb{E}_{{\bf u}_{1}^{k}}^{N,\,\beta}[H_{{\bf p}}]\,=\,[1+o_{N}(1)]\frac{\nu_{1}}{\omega_{\frak{o}}}2\pi N\exp(N\theta_{1})\ .
\]
\item For $\beta_{2}\le\beta<q$, we have
\[
\mathbb{E}_{{\bf p}}^{N,\,\beta}[H_{\mathcal{U}_{1}}]\,=\,[1+o_{N}(1)]\frac{\nu_{\frak{o}}}{q\omega_{\frak{o}}}2\pi N\exp(N\theta_{0})\ .
\]
\item For $\beta>\beta_{3}$ and $k\in S$, we have
\[
\mathbb{E}_{{\bf u}_{1}^{k}}^{N,\,\beta}[H_{\mathcal{U}_{1}\setminus\{{\bf u}_{1}^{k}\}}]\,=\,[1+o_{N}(1)]\frac{\nu_{1}}{(q-1)\omega_{1}}2\pi N\exp(N\theta_{1})\ .
\]
\end{enumerate}
\end{thm}

The proof follows from Theorems \ref{t: metasets all q}-\ref{t: metasets},
Proposition \ref{p: EK constants}, and \cite[Theorem 2.5, Remarks 2.10, 2.11]{Landim Seo-nonrev RW}.

\section{\label{sec: pre result}Preliminary Analysis on Potential and Generator}

In this section, we conduct several preliminary analyses. In Section
\ref{subsec: Dyna of magnet}, we prove Proposition \ref{p: r_N MC}.
In particular, we compute the jump rates of Markov chain $\bm{r}_{N}(\cdot)$.
In Section \ref{subsec: Cyclic}, we decompose the generator $\mathscr{L}_{N}$
into several simple generators $\mathscr{L}_{N,\,\bm{x}}^{i,\,j}$,
$\bm{x}\in\Xi_{N}$, $i,j\in S$. Via this decomposition of $\mathscr{L}_{N}$,
we can handle $\mathscr{L}_{N}$ using the method developed in \cite{Landim Seo-nonrev RW}
since our model is a special case of \cite[Remarks 2.10, 2.11]{Landim Seo-nonrev RW};
this correspondence will be explained in Section \ref{subsec: Special LS}.

\subsection{\label{subsec: Dyna of magnet}Dynamics of Proportion Vector.}

We prove Proposition \ref{p: r_N MC} in this section.
\begin{proof}[Proof of Proposition \ref{p: r_N MC}]
Let $\bm{e}_{j}^{N}:=\frac{1}{N}\bm{e}_{j}$, $j\in S$ (cf. Notation
\ref{nota: symplex}). Fix configurations $\sigma,\tau\in\Omega_{N}$
such that $\bm{r}_{N}(\sigma)=\bm{r}_{N}(\tau)$ and let 
\[
\bm{x}=(x_{1},\,\dots,\,x_{q-1})=\bm{r}_{N}(\sigma)=\bm{r}_{N}(\tau)\in\Xi_{N}\ .
\]
For some sites $u_{1},u_{2},v_{1},v_{2}\in K_{N}$ such that $\sigma_{u_{1}}=\sigma_{v_{1}}=\tau_{u_{2}}=\tau_{v_{2}},$
let $i=\sigma_{v_{1}}$. Then, the Markovity of the process $\bm{r}_{N}(t)$
can be inferred from the identity
\begin{align*}
c_{u_{1},\,j}(\sigma)\,=\,c_{v_{1},\,j}(\sigma)\,=\,c_{u_{2},\,j}(\tau)\,=\,c_{v_{2},\,j}(\tau) & =\,\exp\left\{ -\frac{N\beta}{2}[H(\bm{r}_{N}(\sigma^{v_{1},\,j}))-H(\bm{r}_{N}(\sigma))]\right\} \\
 & =\,\exp\left\{ -\frac{N\beta}{2}[H(\bm{x}+\bm{e}_{j}^{N}-\bm{e}_{i}^{N})-H(\bm{x})]\right\} \ ,
\end{align*}
for $j\ne i$. Hence, $\bm{r}_{N}(\cdot)$ is a Markov chain.

Since there are $Nx_{i}$ sites whose spins are $i$, the jump rate
$R_{N}(\cdot,\cdot)$ of $\bm{r}_{N}(\cdot)$ can be written as
\begin{equation}
R_{N}(\bm{x},\bm{x}+\bm{e}_{j}^{N}-\bm{e}_{i}^{N})\,=\,\frac{Nx_{i}}{N}\exp\left\{ -\frac{N\beta}{2}[H(\bm{x}+\bm{e}_{j}^{N}-\bm{e}_{i}^{N})-H(\bm{x})]\right\} \ .\label{e: rate of r_N}
\end{equation}
Hence, the generator $\mathscr{L}_{N}$ of $\bm{r}_{N}(\cdot)$ is
given as
\[
\mathscr{L}_{N}f(\bm{x})\,=\,\sum_{i,\,j\in S}R_{N}(\bm{x},\bm{x}+\bm{e}_{j}^{N}-\bm{e}_{i}^{N})\,[f(\bm{x}+\bm{e}_{j}^{N}-\bm{e}_{i}^{N})-f(\bm{x})]\ ,
\]
for $f:\Xi_{N}\to\mathbb{R}$.

Finally, this dynamics is reversible with respect to $\nu_{\beta}^{N}$
since we have the following detailed balance condition
\[
\nu_{\beta}^{N}(\bm{x})\,R_{N}(\bm{x},\bm{x}+\bm{e}_{j}^{N}-\bm{e}_{i}^{N})\,=\,\nu_{\beta}^{N}(\bm{x}+\bm{e}_{j}^{N}-\bm{e}_{i}^{N})\,R_{N}(\bm{x}+\bm{e}_{j}^{N}-\bm{e}_{i}^{N},\bm{x})\ ,
\]
so that $\nu_{\beta}^{N}$ is the invariant measure.
\end{proof}

\subsection{\label{subsec: Cyclic}Cyclic Decomposition}

For $1\le i<j\le q$, let $\gamma_{N}^{i,\,j}$ be the cycle $(\bm{e}_{0}^{N},\bm{e}_{j}^{N}-\bm{e}_{i}^{N},\bm{e}_{0}^{N})$
of length 2 on $(\mathbb{Z}/N)^{q-1}$ and let $(\gamma_{N}^{i,\,j})_{\bm{x}}=\bm{x}+\gamma_{N}^{i,\,j}$.
Define $\widehat{\Xi}_{N}^{i,\,j}$ as
\[
\widehat{\Xi}_{N}^{i,\,j}\,=\,\{\bm{x}\in\Xi_{N}\,:\,(\gamma_{N}^{i,\,j})_{\bm{x}}\subset\Xi_{N}\}\,=\,\{\bm{x}\in\Xi_{N}\,:\,x_{j}\le1-N^{-1},\,x_{i}\ge N^{-1}\}\ .
\]
Define a jump rate $\widetilde{R}_{N}^{i,\,j}$ associated with this
cycle as
\begin{align*}
\widetilde{R}_{N,\,0}^{i,\,j}(\bm{x}) & \,=\,\exp\left\{ -N\beta[\overline{F}_{\beta,\,N}^{i,\,j}(\bm{x})-F_{\beta,\,N}(\bm{x})]\right\} \ ,\\
\widetilde{R}_{N,\,1}^{i,\,j}(\bm{x}) & \,=\,\exp\left\{ -N\beta[\overline{F}_{\beta,\,N}^{i,\,j}(\bm{x})-F_{\beta,\,N}(\bm{x}+\bm{e}_{j}^{N}-\bm{e}_{i}^{N})]\right\} \ ,
\end{align*}
where
\[
\overline{F}_{\beta,\,N}^{i,\,j}(\bm{x})\,=\,\frac{1}{2}[F_{\beta,\,N}(\bm{x})+F_{\beta,\,N}(\bm{x}+\bm{e}_{j}^{N}-\bm{e}_{i}^{N})]\ .
\]
Let  $\mathscr{L}_{N,\,\bm{x}}^{i,\,j},\,\bm{x}\in\widehat{\Xi}_{N}$,
be a generator acting on $f:\Xi_{N}\to\mathbb{R}$ as
\begin{equation}
(\mathscr{L}_{N,\,\bm{x}}^{i,\,j}f)(\bm{y})=\begin{cases}
\,\widetilde{R}_{N,\,0}^{i,\,j}(\bm{x})\,\Big[f(\bm{x}+\bm{e}_{j}^{N}-\bm{e}_{i}^{N})-f(\bm{x})\Big] & \bm{y}=\bm{x}\ ,\\
\,\widetilde{R}_{N,\,1}^{i,\,j}(\bm{x})\,\Big[f(\bm{x})-f(\bm{x}+\bm{e}_{j}^{N}-\bm{e}_{i}^{N})\Big] & \bm{y}=\bm{x}+\bm{e}_{j}^{N}-\bm{e}_{i}^{N}\ ,\\
\,0 & \text{otherwise}\ .
\end{cases}\label{e: Def of cycl gen}
\end{equation}
Here, $\mathscr{L}_{N,\,\bm{x}}^{i,\,j}$ can be regarded as a generator
of a Markov chain on the cycle $(\gamma_{N}^{i,\,j})_{\bm{x}}$.

Let
\[
w^{i,\,j}(\bm{x})\,:=\,\sqrt{x_{i}x_{j}}\ ,\ \text{and}\ \ w_{N}^{i,\,j}(\bm{x})\,:=\,\sqrt{x_{i}(x_{j}+\frac{1}{N})}\ .
\]
By \eqref{e: def of F_beta,N}, we can write
\[
\exp\{-\beta N[F_{\beta,\,N}(\bm{x})-H(\bm{x})]\}\,=\,(2\pi N)^{(q-1)/2}{N \choose (Nx_{1})\cdots(Nx_{q})}\ .
\]
By elementary computations, we obtain
\[
R_{N}(\bm{x},\bm{x}+\bm{e}_{j}^{N}-\bm{e}_{i}^{N})/\widetilde{R}_{N,\,0}^{i,\,j}(\bm{x})\,=\,R_{N}(\bm{x}+\bm{e}_{j}^{N}-\bm{e}_{i}^{N},\bm{x})/\widetilde{R}_{N,\,1}^{i,\,j}(\bm{x})\,=\,w_{N}^{i,\,j}(\bm{x})\ ,
\]
so that
\begin{align*}
R_{N}(\bm{x},\bm{x}+\bm{e}_{j}^{N}-\bm{e}_{i}^{N})[f(\bm{x}+\bm{e}_{j}^{N}-\bm{e}_{i}^{N})-f(\bm{x})] & =w_{N}^{i,\,j}(\bm{x})\,\mathscr{L}_{N,\,\bm{x}}^{i,\,j}f(\bm{x})\\
R_{N}(\bm{x},\bm{x}+\bm{e}_{i}^{N}-\bm{e}_{j}^{N})[f(\bm{x}+\bm{e}_{i}^{N}-\bm{e}_{j}^{N})-f(\bm{x})] & =w_{N}^{i,\,j}(\bm{x}+\bm{e}_{i}^{N}-\bm{e}_{j}^{N})\,\mathscr{L}_{N,\,\bm{x}+\bm{e}_{i}^{N}-\bm{e}_{j}^{N}}^{i,\,j}f(\bm{x})\ .
\end{align*}
Hence, by \eqref{e: Def of cycl gen}, we can write
\begin{align}
\mathscr{L}_{N}f(\bm{x}) & =\sum_{1\le i<j\le q}[w_{N}^{i,\,j}(\bm{x})\,\mathscr{L}_{N,\,\bm{x}}^{i,\,j}+w_{N}^{i,\,j}(\bm{x}+\bm{e}_{i}^{N}-\bm{e}_{j}^{N})\,\mathscr{L}_{N,\,\bm{x}+\bm{e}_{i}^{N}-\bm{e}_{j}^{N}}^{i,\,j}]f(\bm{x})\nonumber \\
 & =\sum_{1\le i<j\le q}\sum_{\bm{y}\in\widehat{\Xi}_{N}^{i,\,j}}w_{N}^{i,\,j}(\bm{y})\mathscr{L}_{N,\,\bm{y}}^{i,\,j}f(\bm{x})\ .\label{e: cyc rep of gen}
\end{align}
Since $w_{N}^{i,\,j}$ converges uniformly to $w^{i,\,j}$ and is
uniformly Lipschitz on every compact subset of $\text{int}\,\Xi$,
our model is a special case of \cite[Remarks 2.10, 2.11]{Landim Seo-nonrev RW}
provided that several technical requirements are verified. These requirements
will be verified in the next subsection.
\begin{rem}
\cite[Section 2]{Landim Seo-nonrev RW} assumes that for $\gamma_{N}^{i,\,j}=(\bm{z}_{0},\bm{z}_{1})$,
$\bm{z}_{1}-\bm{z}_{0}$ generates $\mathbb{Z}^{q-1}$; however, this
requirement is needed to make $\bm{r}_{N}(\cdot)$ be irreducible.
Since $(\gamma_{N}^{i,\,j})_{1\le i<j\le q}$ generate $\mathbb{Z}^{q-1}$,
we do not need this assumption.
\end{rem}

\subsection{\label{subsec: Special LS}Requirements for $F_{\beta,\,N}$ and
$\mathscr{L}_{N}$}

In this subsection, we verify that our model is a special case of
\cite[Remarks 2.10, 2.11]{Landim Seo-nonrev RW}.

First, we give some properties of $F_{\beta}(\cdot)$ and $G_{\beta,\,N}(\cdot)$.
By the following proposition, $F_{\beta}(\cdot)$ and $G_{\beta,\,N}(\cdot)$
fulfill the requirements in the first paragraph of \cite[Section 2]{Landim Seo-nonrev RW}.
\begin{prop}
\label{p: F,G}The functions $F_{\beta}(\cdot)$ and $G_{\beta,\,N}(\cdot)$
satisfy the following properties. 
\begin{enumerate}
\item $F_{\beta}$ is twice-differentiable and there is no critical points
at $\partial\,\Xi$. For all $\bm{x}\in\partial\Xi$, $\nabla F_{\beta}(\bm{x})\cdot\bm{n}(\bm{x})>0$.
\item The second partial derivatives of $F_{\beta}(\cdot)$ are Lipschitz-continuous
on every compact subset of $\text{int }\Xi$.
\item On each compact subset of $\text{int }\Xi$, $G_{\beta,\,N}(\cdot)$
is uniformly Lipschitz and converges uniformly to $G_{\beta}(\bm{x}):=(1/2\beta)\log(x_{1}\cdots x_{q})$
as $N\to\infty$.
\item There are finitely many critical points of $F_{\beta}(\cdot)$.
\end{enumerate}
\end{prop}

\begin{proof}
(1)-(3) are straightforward. By Lemma \ref{l: cand cri} in Section
\ref{subsec: class cri}, there are finitely many critical points
of $F_{\beta}(\cdot)$. 
\end{proof}
Next, fix one of saddle points ${\bf s}$. Note that $\nabla^{2}F_{\beta}({\bf s})$
has a unique negative eigenvalue. As in \cite[Section 4.3]{Landim Seo-nonrev RW},
define $(\mathscr{L}_{N}^{i,\,j})^{{\bf s}}$ as
\[
(\mathscr{L}_{N}^{i,\,j})^{{\bf s}}f(\bm{x})\,=\,\frac{1}{N^{2}}(\bm{e}_{j}-\bm{e}_{i})^{\dagger}\nabla^{2}f(\bm{x})(\bm{e}_{j}-\bm{e}_{i})-\frac{1}{N}\mathbb{A}^{i,\,j}\nabla^{2}F({\bf s})(\bm{x}-{\bf s})\cdot\nabla f(\bm{x})\ .
\]
Denote by $-\lambda_{1}^{{\bf s}},\lambda_{2}^{{\bf s}},\dots,\lambda_{q-1}^{{\bf s}}$
($\lambda_{1}^{{\bf s}}\sim\lambda_{q}^{{\bf s}}>0$) the eigenvalues
of $\nabla^{2}F_{\beta}({\bf s})$, and by $\bm{a}_{1}^{{\bf s}},\bm{a}_{2}^{{\bf s}},\dots,\bm{a}_{q-1}^{{\bf s}}$
the corresponding eigenvectors. Let $\epsilon_{N}:=N^{-2/5}\ll N^{-1}$
so that $\epsilon_{N}$ satisfies \cite[displays (4.7), (4.8)]{Landim Seo-nonrev RW}.
Define a neighborhood of ${\bf s}$ as
\[
\mathcal{C}_{N}^{{\bf s}}\,\coloneqq\,\bigg\{{\bf s}+\sum_{k=1}^{q-1}x_{k}{\bf a}_{k}^{{\bf s}}:|x_{1}|\le\epsilon_{N},\,|x_{k}|\le\sqrt{\frac{2\lambda_{1}^{{\bf s}}}{\lambda_{k}}}\epsilon_{N},\,2\le k\le q-1\bigg\}\cap\Xi_{N}\ .
\]
Then, by the next proposition, definitions \eqref{e: Def omega1}-\eqref{e: Def nu0}
are consistent with \cite[Remarks 2.10, 2.11]{Landim Seo-nonrev RW}.
\begin{prop}
\label{p: EK constants}For a smooth function $f:\Xi\to\mathbb{R}$,
we have uniformly on $\mathcal{C}_{N}^{{\bf s}}$,
\[
\mathscr{L}_{N}f\,=\,[1+O(\epsilon_{N})]\sum_{1\le i<j\le q}w^{i,\,j}({\bf s})(\mathscr{L}_{N}^{i,\,j})^{{\bf s}}f\ .
\]
\end{prop}

\begin{proof}
Since $|\bm{x}-{\bf s}|=O(\epsilon_{N})$, by \eqref{e: Def of cycl gen}
and the second order Taylor expansion on $\mathcal{C}_{N}^{{\bf s}}$,
we have
\[
\sum_{\bm{y}\in\widehat{\Xi}_{N}^{i,\,j}}\mathscr{L}_{N,\,\bm{y}}^{i,\,j}f(\bm{x})\,=\,[1+O(\epsilon_{N})](\mathscr{L}_{N}^{i,\,j})^{{\bf s}}f(\bm{x})\ .
\]
Hence, on $\mathcal{C}_{N}^{{\bf s}}$, since $w_{N}^{i,\,j}(\bm{x})=[1+O(N^{-1})]w^{i,\,j}(\bm{x})=[1+O(\epsilon_{N})]w^{i,\,j}({\bf s})$,
we have
\begin{align*}
\mathscr{L}_{N}f(\bm{x})\, & =\,\sum_{1\le i<j\le q}\sum_{\bm{y}\in\widehat{\Xi}_{N}^{i,\,j}}w_{N}^{i,\,j}(\bm{y})\mathscr{L}_{N,\,\bm{y}}^{i,\,j}f(\bm{x})\\
 & =\,[1+O(\epsilon_{N})]\sum_{1\le i<j\le q}w^{i,\,j}({\bf s})\sum_{\bm{y}\in\widehat{\Xi}_{N}^{i,\,j}}\mathscr{L}_{N,\,\bm{y}}^{i,\,j}f(\bm{x})\\
 & =\,[1+O(\epsilon_{N})]\sum_{1\le i<j\le q}w^{i,\,j}({\bf s})(\mathscr{L}_{N}^{i,\,j})^{{\bf s}}f(\bm{x})\ .
\end{align*}
\end{proof}

\section{\label{sec: Pre cri}Investigation of Critical Points and Temperatures}

This section is devoted to the investigation of critical points and
temperatures including their definitions. We will provide a preliminary
analysis of the critical points in Section \ref{subsec: class cri}
and of the critical temperatures in section \ref{subsec: pre beta}.

\subsection{\label{subsec: class cri}Classification of Critical Points}

We recall that
\[
F_{\beta}(\bm{x})\,=\,-\frac{1}{2}\sum_{k=1}^{q}x_{k}^{2}\,+\,\frac{1}{\beta}\sum_{k=1}^{q}x_{k}\log x_{k}\ ,
\]
and that $x_{q}=1-(x_{1}+\cdots+x_{q-1})$. For $1\le k\le q-1$,
\[
\frac{\partial}{\partial x_{k}}F_{\beta}(\bm{x})\,=\,-(x_{k}-x_{q})\,+\,\frac{1}{\beta}(\log x_{k}-\log x_{q})\ .
\]
If $\frac{\partial}{\partial x_{k}}F_{\beta}(\bm{x})=0$, we must
have $x_{k}-\frac{1}{\beta}\log x_{k}=x_{q}-\frac{1}{\beta}\log x_{q}$.
Hence, 
\begin{equation}
\nabla F_{\beta}(\bm{x})=0\ \text{if and only if}\ x_{k}-\frac{1}{\beta}\log x_{k},\ 1\le k\le q,\ \text{are the same}\ .\label{e: crit}
\end{equation}
By \eqref{e: crit}, ${\bf p}=(1/q,\,\dots,\,1/q)$ is a critical
point.

By elementary computation, we can check that the equation $t-\frac{1}{\beta}\log t=c$
has at most two positive real solutions for fixed $\beta,\,c>0$.
Hence, if $(x_{1},\dots,x_{q})$ is a critical point\footnote{Recall Notaion \ref{nota: symplex}.},
$x_{k}$'s can have at most 2 values by \eqref{e: crit}. Hereafter,
we assume ${\bf c}$ is a critical point and
\[
{\bf c}\,=\,(t,\,\dots,\,t,(1-jt)/i,\,\dots,\,(1-jt)/i)\ ,
\]
where $j$ is the number of $t$'s and $i=q-j$. Observe that by symmetry,
each permutation of coordinates of ${\bf c}$ has the same properties.
Without loss of generality, we may assume
\[
1\le i\le q/2\le j\le q-1\ \text{and}\ t\ne1/q\ .
\]
The point ${\bf p}$ will be analyzed separately.

\begin{figure}
\subfloat[$i=2$.]{\includegraphics[scale=0.7]{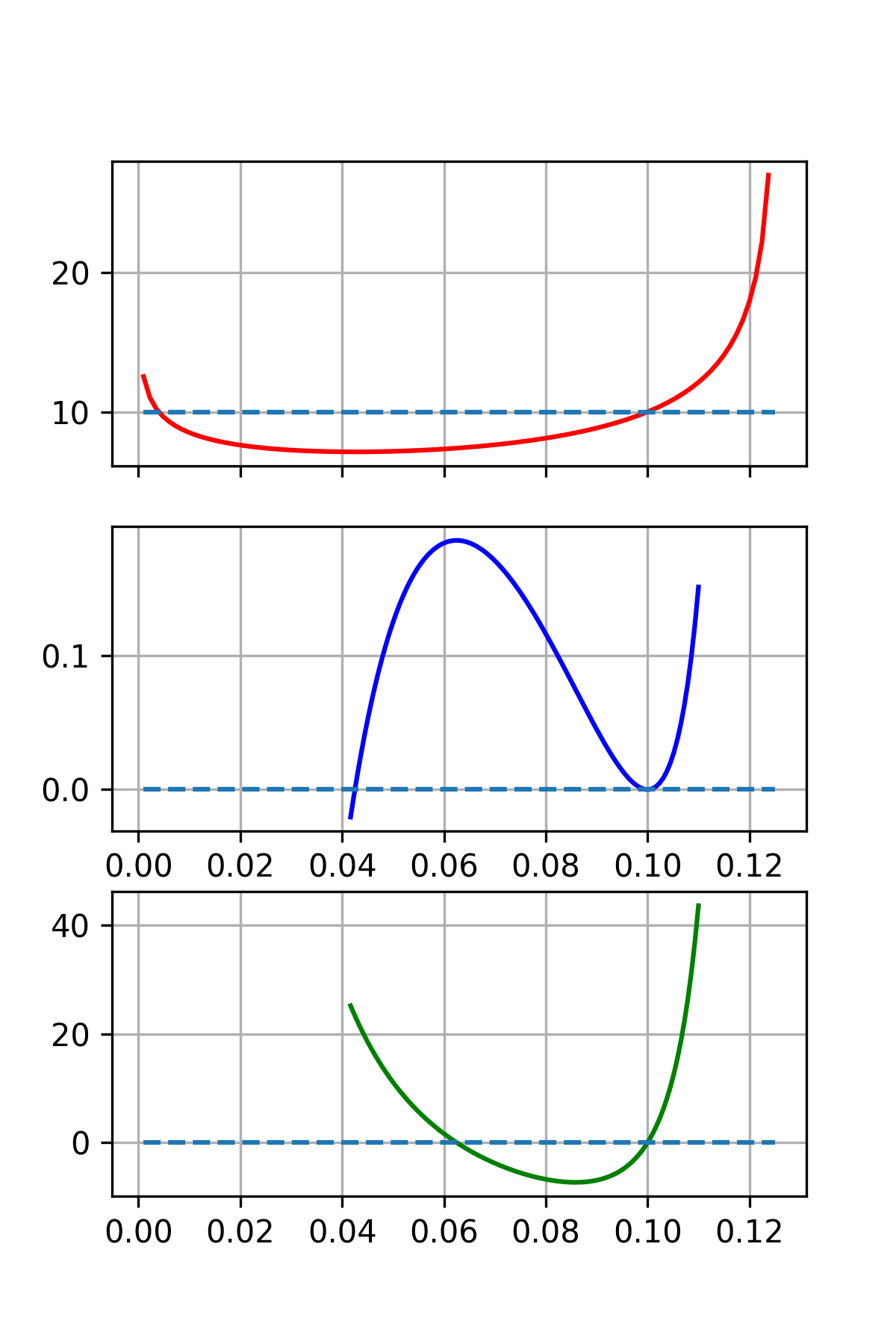}

}\subfloat[$i=q/2$.]{\includegraphics[scale=0.7]{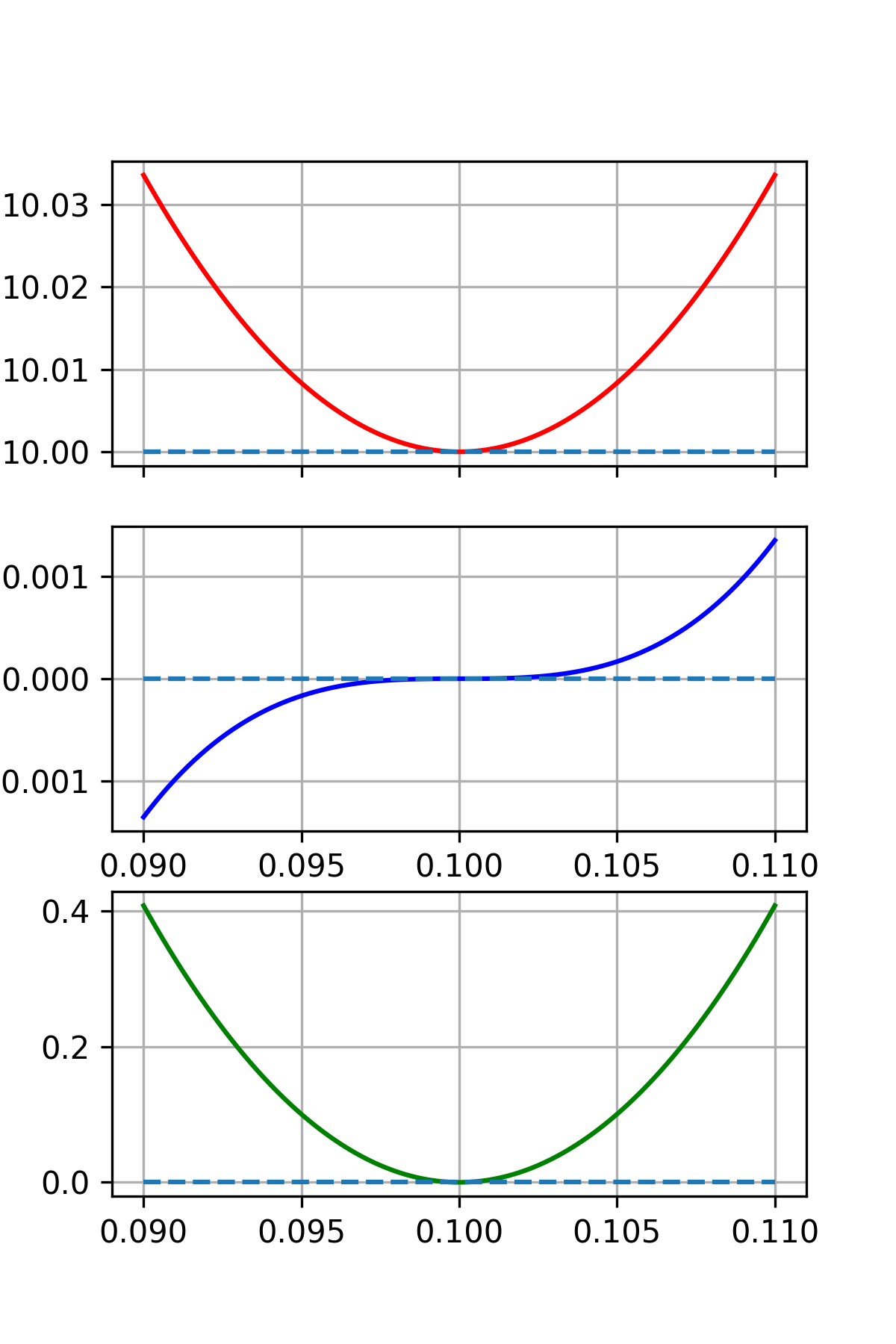}}\caption{\label{fig: g2h2h2'}Graphs of $g_{i}(t)$, $h_{i}(t)$, and $h_{i}'(t)$
when $q=10$.}
\end{figure}

By \eqref{e: crit}, we obtain
\[
t-\frac{1}{\beta}\log t\,=\,\frac{1-jt}{i}-\frac{1}{\beta}\log\Big(\frac{1-jt}{i}\Big)\ ,
\]
which implies 
\begin{equation}
\beta\,=\,\frac{i}{1-qt}\log\Big(\frac{1-jt}{it}\Big)\,=\,g_{i}(t)\ .\label{e: beta=00003Dg_i}
\end{equation}

\begin{lem}
\label{l: g_i}Fix $q\ge3$, $1\le i\le q/2$ and $j=q-i$. Then,
the function $g_{i}:(0,1/j)\to\mathbb{R}$ has the unique minimum,
say $m_{i}$. Furthermore, if $\beta>g_{i}(m_{i})$, $\beta=g_{i}(t)$
has two solutions.
\end{lem}

\begin{proof}
Define $h_{i}:(0,1/j)\to\mathbb{R}$ as\footnote{As $g_{i}(\cdot)$; the function $h_{i}(\cdot)$ can be continuously
extended to $(1,1/j)$.}
\begin{equation}
h_{i}(t)\,\coloneqq\,\log\frac{1-jt}{it}+\frac{qt-1}{qt(1-jt)}\ .\label{e: Def of h_i}
\end{equation}
By elementary computation, we obtain
\begin{equation}
g_{i}'(t)\,=\,\frac{qi}{(1-qt)^{2}}h_{i}(t)\ \ \text{and}\ \ h_{i}'(t)\,=\,\frac{(qt-1)(2jt-1)}{q(1-jt)^{2}t^{2}}\ .\label{e: deriv of g,h}
\end{equation}

There are two cases, where $i<q/2$ and $i=q/2$. By elementary computation,
we can show that the graphs of $g_{i},\,h_{i},\,h_{i}'$ are given
by Figure \ref{fig: g2h2h2'}, which completes the proof.
\end{proof}
For $1\le i\le q/2$, let
\begin{equation}
\beta_{s,\,i}=\beta_{s,\,i}(q)\,\coloneqq\,g_{i}(m_{i})\ ,\label{e: Def of beta_si}
\end{equation}
where $m_{i}$ is the unique minimum of $g_{i}(\cdot)$ given in the
above lemma.

If $\beta\ge\beta_{s,\,i}$, there are one or two solutions of $\beta=g_{i}(t)$
which will be denoted by $u_{i}=u_{i}(\beta),\,v_{i}=v_{i}(\beta)$
where $u_{i}\le v_{i}$. Let
\begin{align*}
\mathcal{U}_{i}=\mathcal{U}_{i}(\beta) & =\{\text{permutations of }(u_{i},\,\dots,\,u_{i},\,(1-ju_{i})/i,\,\dots,\,(1-ju_{i})/i)\}\ ,\\
\mathcal{V}_{i}=\mathcal{V}_{i}(\beta) & =\{\text{permutations of }(v_{i},\,\dots,\,v_{i},\,(1-jv_{i})/i,\,\dots,\,(1-jv_{i})/i)\}\ ,
\end{align*}
for $\beta\ge\beta_{s,\,i}$ . We have the following candidates of
the critical points of $F_{\beta}$.
\begin{lem}
\label{l: cand cri}A critical point of $F_{\beta}$ is exactly one
of the following cases.
\begin{enumerate}
\item ${\bf p}=(1/q,\dots,1/q)$.
\item For $1\le i\le q/2$ and $\beta\in(\beta_{s,\,i},\infty)$, elements
of \textup{$\mathcal{U}_{i}$}.
\item For $1\le i\le q/2$ and $\beta\in(\beta_{s,\,i},\infty)\setminus\{q\}$,
elements of \textup{$\mathcal{V}_{i}$}.
\item For $1\le i<q/2$ and $\beta=\beta_{s,\,i}$, elements of \textup{$\mathcal{U}_{i}=\mathcal{V}_{i}$.}
\end{enumerate}
\end{lem}

\begin{proof}
By part (1) of Proposition \ref{p: F,G}, points in $\partial\Xi$
cannot be critical points. Then, the proof follows from \eqref{e: crit}
and Lemma \ref{l: g_i}.
\end{proof}
Finally, we have the following results for critical points. The proof
for $q=3$ is given in \cite[Proposition 4.2]{Landim Seo-3spin}.
\begin{prop}
\label{p: cri}The minima and saddle points of $F_{\beta}$ for $q=3$,
$q=4$, and $q\ge5$ are given by Table 1, 2, and 3, respectively.

\begin{table}[H]
\begin{tabular}{|c|c|c|c|}
\hline 
 & ${\bf p}$ & $\mathcal{U}_{1}(\beta)$ & $\mathcal{V}_{1}(\beta)$\tabularnewline
\hline 
$\beta\in(0,\beta_{s,\,1})$ & the only minimum &  & \tabularnewline
\hline 
$\beta=\beta_{s,\,1}$ & the only minimum & degenerate & degenerate\tabularnewline
\hline 
$\beta\in(\beta_{s,\,1},q)$ & local minimum & local minima & saddle points\tabularnewline
\hline 
$\beta=q$ & degenerate & local minima & degenerate\tabularnewline
\hline 
$\beta\in(q,\infty)$ & local maximum & local minima & saddle points\tabularnewline
\hline 
\end{tabular}

\caption{Classification of critical points when $q=3$}
\end{table}

\begin{table}[H]
\begin{tabular}{|c|c|c|>{\centering}p{2.5cm}|c|}
\hline 
 & ${\bf p}$ & $\mathcal{U}_{1}(\beta)$ & $\mathcal{V}_{1}(\beta)$ & $\mathcal{U}_{2}(\beta)=\mathcal{V}_{2}(\beta)$\tabularnewline
\hline 
$\beta\in(0,\beta_{s,\,1})$ & the only minimum &  &  & \tabularnewline
\hline 
$\beta=\beta_{s,\,1}$ & the only minimum & degenerate & degenerate & \tabularnewline
\hline 
$\beta\in(\beta_{s,\,1},q)$ & local minimum & local minima & saddle points & \tabularnewline
\hline 
$\beta=q$ & degenerate & local minima & degenerate & degenerate\tabularnewline
\hline 
$\beta\in(q,\infty)$ & local maximum & local minima & index $\ge2$ & saddle points\tabularnewline
\hline 
\end{tabular}

\caption{Classification of critical points when $q=4$}
\end{table}

\begin{table}[H]
\begin{tabular}{|c|c|c|>{\centering}p{2.5cm}|c|}
\hline 
 & ${\bf p}$ & $\mathcal{U}_{1}(\beta)$ & $\mathcal{V}_{1}(\beta)$ & $\mathcal{U}_{2}(\beta)$\tabularnewline
\hline 
$\beta\in(0,\beta_{s,\,1})$ & the only minimum &  &  & \tabularnewline
\hline 
$\beta=\beta_{s,\,1}$ & the only minimum & degenerate & degenerate & \tabularnewline
\hline 
$\beta\in(\beta_{s,\,1},\beta_{s,\,2})$ & local minimum & local minima & saddle points & \tabularnewline
\hline 
$\beta=\beta_{s,\,2}$ & local minimum & local minima & saddle points & degenerate\tabularnewline
\hline 
$\beta\in(\beta_{s,\,2},q)$ & local minimum & local minima & saddle points & saddle points\tabularnewline
\hline 
$\beta=q$ & degenerate & local minima & degenerate & degenerate\tabularnewline
\hline 
$\beta\in(q,\infty)$ & local maximum & local minima & index $\ge2$ & saddle points\tabularnewline
\hline 
\end{tabular}

\caption{Classification of critical points when $q=5$}
\end{table}
\end{prop}

Section \ref{sec: Pf Equi points} proves the above proposition. Until
now, we classified all minima and saddle points for all $q\ge3$.

\subsection{\label{subsec: pre beta}Definition of Critical Temperatures}

In the previous subsection, we defined several temperatures $\beta_{s,\,i}$,
$1\le i\le q/2$. In this subsection, we prove several properties
of such temperatures and moreover introduce new temperatures. Then,
we select the critical temperatures at which phase transitions occur.

The first lemma is about the order of $\beta_{s,\,i}$.
\begin{lem}
\label{l: beta_s1<beta_s2} We have $\beta_{s,\,1}<\beta_{s,\,2}<\cdots<\beta_{s,\,\lfloor q/2\rfloor}$.
If $q$ is even, we have $\beta_{s,\,q/2}=q$ and otherwise, $\beta_{s,\,\lfloor q/2\rfloor}<q$.
\end{lem}

\begin{proof}
In this proof, we regard $i$ as a real number and claim that $g_{i}(t)$
increases as $i\in[1,q]$ increases for fixed $t<1/q$. By elementary
computation, we obtain
\[
\frac{d}{di}g_{i}(t)\,=\,\frac{1}{1-qt}\Big(\log\frac{1-jt}{it}+\frac{it}{1-jt}-1\Big)\ .
\]
By the inequality $x-1\,>\,\log x$, we can conclude that $\frac{d}{di}g_{i}(t)>0$.
Hence, $g_{i}(t)<g_{i+1}(t)$ if $t<1/q$.

Hereafter, let $i\in\mathbb{Z}$. Suppose $i<q/2-1$. Since $m_{i},\,m_{i+1}<1/q$,
we obtain
\[
\beta_{s,\,i}=g_{i}(m_{i})\le g_{i}(m_{i+1})<g_{i+1}(m_{i+1})=\beta_{s,\,i+1}\ ,
\]
by the above claim. The first inequality holds since $m_{i}$ is a
minimum of $g_{i}$. If $i=q/2-1$, since $m_{i}<1/q=m_{i+1}$, we
obtain
\[
\beta_{s,\,i}=g_{i}(m_{i})<g_{i}(m_{i+1})=q=\beta_{s,\,q/2}\ .
\]

If $i<q/2$, we have $m_{i}<1/q$ so that $\beta_{s,\,i}<g_{i}(1/q)=q$.
This with the above argument prove the second assertion.
\end{proof}
\begin{rem*}
In particular, by the above lemma, we have $\beta_{s,\,1}<\beta_{s,\,2}=q$
for $q=4$ and $\beta_{s,\,1}<\beta_{s,\,2}<q$ for $q\ge5$.
\end{rem*}
The relative order of heights of critical points changes with changes
in $\beta$, and the phase transition is owing to this fact. We will
explain when and how this order is changed. Since the proofs are technical,
they are postponed to Section \ref{sec: Pf beta}.

\subsubsection*{Order of heights of ${\bf p}$ and $\mathcal{U}_{1}$}

Define $\beta_{c}$ as
\begin{equation}
\beta_{c}(q)\,:=\,\frac{2(q-1)}{q-2}\log(q-1)\ ,\label{e: Def of beta_c}
\end{equation}
which is introduced in \cite[display (3.3)]{Costeniuc Ellis T}. Then,
we obtain the following.
\begin{lem}
\label{l: beta_s1<beta_c}For $q\ge3$, we have $\beta_{s,\,1}<\beta_{c}$
and for $q\ge4$, we have $\beta_{s,\,1}<\beta_{c}<\beta_{s,\,2}$.
\end{lem}

The proof of the lemma is given in Section \ref{subsec: Pf  l:beta_s<beta_c2}.
The following lemma is an important property of $\beta_{c}$.
\begin{lem}
\label{l: beta_c}Let $q\ge3$. Then, we have
\begin{equation}
\begin{cases}
F_{\beta}({\bf p})<F_{\beta}({\bf u}_{1}) & \text{if }\beta\in(\beta_{s,\,1},\beta_{c})\ ,\\
F_{\beta}({\bf p})=F_{\beta}({\bf u}_{1}) & \text{if }\beta=\beta_{c}\ ,\\
F_{\beta}({\bf p})>F_{\beta}({\bf u}_{1}) & \text{if }\beta\in(\beta_{c},\infty)\ .
\end{cases}\label{e: Prop of beta_c}
\end{equation}
\end{lem}

This result is the same as \cite[Theorem 3.1(b)]{Costeniuc Ellis T}.
The proof is provided in \cite[Appendices A, B]{Costeniuc Ellis T}
via convex-duality.

We may assume that $\beta$ increases from a very small positive number.
Observe that the elements of $\mathcal{U}_{1}$ and $\mathcal{V}_{1}$
simultaneously appear when $\beta=\beta_{s,\,1}$ and the elements
of $\mathcal{U}_{2}$ appear when $\beta=\beta_{s,\,2}$ . By the
above two lemmas, before the appearance of critical points in $\mathcal{U}_{2}$,
the heights of ${\bf p}$ and ${\bf u}_{1}$ are reversed.

\subsubsection*{Order of heights of $\mathcal{V}_{1}$ and $\mathcal{U}_{2}$}

We have the following lemma about the heights of ${\bf u}_{2}$ and
${\bf v}_{1}$. The critical temperature $\beta_{m}$ given in the
following lemma is the crucial development of this article.
\begin{lem}
\label{l: beta_m}Let $q\ge5$. We have a critical temperature $\beta_{m}\in(\beta_{s,\,2},q)$
such that
\begin{equation}
\begin{cases}
F_{\beta}({\bf v}_{1})<F_{\beta}({\bf u}_{2}) & \text{if }\beta_{s,\,2}\le\beta<\beta_{m}\ ,\\
F_{\beta}({\bf v}_{1})=F_{\beta}({\bf u}_{2}) & \text{if }\beta=\beta_{m}\ ,\\
F_{\beta}({\bf v}_{1})>F_{\beta}({\bf u}_{2}) & \text{if }\beta_{m}<\beta\le q\ .
\end{cases}\label{e: Prop of beta_m}
\end{equation}
\end{lem}

The proof of the lemma is given in Section \ref{subsec: Pf l: beta_m}.

Up to this point, we have obtained four critical values
\[
0<\beta_{s,\,1}<\beta_{c}<\beta_{s,\,2}<\beta_{m}<q\ ,
\]
when $q\ge5$. If $q=4$, we have $\beta_{s,\,2}=q$, else if $q=3$,
$\beta_{s,\,2}$ is not defined. Thus, if $q\le4$, define $\beta_{m}=q$
so that
\[
0<\beta_{s,\,1}<\beta_{c}<\beta_{m}=q\ .
\]

We conclude this section with the definition of the critical temperatures
at which the phase transitions occur. We can now define critical temperatures
$\beta_{1},\beta_{2},\beta_{3}$ appearing in Section \ref{subsec: Main critemp}.
The critical temperatures are given by
\begin{equation}
\beta_{1}(q)\coloneqq\beta_{s,\,1}(q),\ \ \beta_{2}(q)\coloneqq\beta_{c}(q),\ \ \beta_{3}(q)\coloneqq\beta_{m}(q)\ .\label{e: Def of cri temp}
\end{equation}

\section{\label{sec: Pf Equi points}Critical Points of $F_{\beta}$}

In this section, we prove Proposition \ref{p: cri} for $q\ge4$.
For the case $q=3$, we refer to \cite{Landim Seo-3spin} and we will
only highlight the difference.

\subsection{Eigenvalues of Hessian of $F_{\beta}$ at Critical Points}

First, we investigate ${\bf p}=(1/q,\,\dots,\,1/q)$, which is always
a critical point for all $\beta>0$. The following lemma proves the
property of ${\bf p}$.
\begin{lem}
The point ${\bf p}$ is a local minimum of $F_{\beta}$ if $\beta<q$,
a local maximum of $F_{\beta}$ if $\beta>q$, and a degenerate critical
point when $\beta=q$.
\end{lem}

\begin{proof}
Let $\mathbbm1=(1,\dots,1)^{\dagger}$ be a $(q-1)\times1$ matrix.
By elementary computation, we obtain
\[
\nabla^{2}F_{\beta}({\bf p})\,=\,\frac{q-\beta}{\beta}\Big(\text{diag}(1,\dots,1)+\mathbbm1\mathbbm1^{\dagger}\Big)
\]
whose eigenvalues are $(q-\beta)/\beta$ with multiplicity $q-2$
and $q(q-\beta)/\beta$ with 1. This completes the proof.
\end{proof}
Now, for $i\in[1,q/2]\cap\mathbb{N}$, $j=q-i$, and $\beta=g_{i}(t)$,
define $a\in\mathbb{R}$ and $b\in\mathbb{R}$ as
\begin{equation}
a=a(i,t)=-1+1/\beta t\ ,\ \ b=b(i,t)=-1+i/\beta(1-jt)\ .\label{e: def ab}
\end{equation}
We have the following lemma about eigenvalues of Hessian of $F_{\beta}$
at critical points.
\begin{lem}
\label{l:ev}Let $i\in[1,q/2]\cap\mathbb{N}$ and $j=q-i$. Moreover,
let $t\in(0,1/j)$ and $\beta=g_{i}(t)$. Then, ${\bf c}=(t,\dots,t,(1-jt)/i,\dots,(1-jt)/i)$
is a critical point of $F_{\beta}$ and eigenvalues of $\nabla^{2}F_{\beta}({\bf c})$
constitute one of the following cases.
\begin{enumerate}
\item If $i\ge2$, all eigenvalues of $\nabla^{2}F_{\beta}(\bm{c})$ are
$a$, $b$ with multiplicative $j-1$, $i-2$, respectively, and the
roots of $\lambda^{2}-(a+qb)\lambda+b(ia+jb)$.
\item If $i=1$, all eigenvalues of $\nabla^{2}F_{\beta}(\bm{c})$ are $a$
with multiplicative $j-1$ and $a+(q-1)b$ with multiplicative 1.
\end{enumerate}
\end{lem}

\begin{proof}
By Lemma \ref{l: cand cri}, ${\bf c}$ is a critical point of $F_{\beta}$
since $\beta=g_{i}(t)$. By elementary computation, we have
\begin{align*}
\frac{\partial^{2}}{\partial x_{k}^{2}}F_{\beta}(\bm{x}) & \,=\,-1+\frac{1}{\beta x_{k}}+\Big(-1+\frac{1}{\beta x_{q}}\Big)\ ,\\
\frac{\partial^{2}}{\partial x_{k}\partial x_{l}}F_{\beta}(\bm{x}) & \,=\,-1+\frac{1}{\beta x_{q}}\ ,
\end{align*}
so that
\[
\frac{\partial^{2}}{\partial x_{k}\partial x_{l}}F_{\beta}({\bf c})=\begin{cases}
-1+\frac{1}{\beta t}+(-1+\frac{i}{\beta(1-jt)}) & \text{if }1\le k=l\le j\\
2(-1+\frac{i}{\beta(1-jt)}) & \text{if }j+1\le k=l\le q-1\\
-1+\frac{i}{\beta(1-jt)} & \text{if }k\ne l
\end{cases}
\]
Then, we can write $\nabla^{2}F_{\beta}({\bf c})$ as
\[
\nabla^{2}F_{\beta}({\bf c})\,=\,\mathbb{D}+b\mathbbm1\mathbbm1^{\dagger}\ ,
\]
where
\[
\mathbb{D}\,=\,\text{diag}(\underset{j}{\underbrace{a,\dots,a}},\underset{q-1-j}{\underbrace{b,\dots,b}})\ .
\]
Let $\mathbb{I}=\mathbb{I}_{q-1}$ be a $(q-1)$-identity matrix.
By the formula 
\[
\det(A+\bm{v}\bm{w}^{\dagger})=\det A(1+\bm{v}^{\dagger}A^{-1}\bm{w})\ ,
\]
we can write
\[
\det(\nabla^{2}F_{\beta}({\bf c})-\lambda\mathbb{I})\,=\,\det(\mathbb{D}-\lambda\mathbb{I}+b\mathbbm1\mathbbm1^{\dagger})\,=\,(a-\lambda)^{j}(b-\lambda)^{i-1}\Big[1+b\big(\frac{j}{a-\lambda}+\frac{i-1}{b-\lambda}\big)\Big]\ .
\]
Hence, we obtain
\[
\det(\nabla^{2}F_{\beta}({\bf c})-\lambda\mathbb{I})=\begin{cases}
(a-\lambda)^{j-1}(b-\lambda)^{i-2}(\lambda^{2}-(a+qb)\lambda+b(ia+jb)) & \text{if }i\ge2\ ,\\
(a-\lambda)^{j-1}(a+jb-\lambda)=(a-\lambda)^{q-2}(a+(q-1)b-\lambda) & \text{if }i=1\ .
\end{cases}
\]
The proof of the lemma arises directly from this explicit computation
of characteristic polynomial of Hessian of $F_{\beta}({\bf c})$.
\end{proof}
We have the following lemma about the sign of the eigenvalues of $\nabla^{2}F_{\beta}({\bf c})$.
Recall the definition of $m_{i}$ from Lemma \ref{l: g_i}.
\begin{lem}
\label{l:sign ab}Let $i\in[1,q/2]\cap\mathbb{N}$ and $j=q-i$. Moreover,
let $t\in(0,1/j)$ and $\beta=g_{i}(t)$. Then, we have the following
table regarding the sign of each value. If $i=q/2$, we ignore $t=m_{i}$
and $t\in(m_{i},1/q)$.

\begin{table}[H]
\begin{tabular}{|c|c|c|c|c|c|}
\hline 
 & $t\in(0,m_{i})$ & $t=m_{i}$ & $t\in(m_{i},1/q)$ & $t=1/q$ & $t\in(1/q,1/j)$\tabularnewline
\hline 
$a$ & $+$ & $+$ & $+$ & $0$ & $-$\tabularnewline
\hline 
$b$ & $-$ & $-$ & $-$ & $0$ & $+$\tabularnewline
\hline 
$ia+jb$ & $+$ & $0$ & $-$ & $0$ & $+$\tabularnewline
\hline 
$b(ia+jb)$ & $-$ & $0$ & $+$ & $0$ & $+$\tabularnewline
\hline 
\end{tabular}

\end{table}
\end{lem}

\begin{proof}
First, suppose that $t<1/q$. Then,
\begin{align*}
a>0\, & \Longleftrightarrow\,\frac{1}{t}>\beta=\frac{i}{1-qt}\log\Big(\frac{1-jt}{it}\Big)\Leftrightarrow\,\frac{1-qt}{it}>\log\Big(\frac{1-jt}{it}\Big)\ .
\end{align*}
By substituting $x=(1-jt)/(it)$, one can deduce that $a>0$ is equivalent
to $t\ne1/q$ which implies $a>0$. Moreover, by the same argument
above, we have $b<0$. In the same manner, if $t>1/q$, we obtain
$a<0$ and $b>0$.

Now, we investigate the sign of $ia+jb$. We write
\[
ia+jb\,=\,-i+\frac{i}{\beta t}-j+\frac{ij}{\beta(1-jt)}\,=\,-q+\frac{i}{\beta t(1-jt)}\ .
\]
By elementary computation, $ia+jb=0$ if $t=1/q$. Hence, $ia+jb>0$
if and only if 
\[
\frac{i}{qt(1-jt)}\,>\,\beta\,=\,\frac{i}{1-qt}\log\Big(\frac{1-jt}{it}\Big)\ .
\]
First, assume $t<1/q$. Then, $ia+jb>0$ if and only if
\[
h_{i}(t)\,=\,\log\Big(\frac{1-jt}{it}\Big)+\frac{qt-1}{qt(1-jt)}\,<\,0\ .
\]
By investigating the graph of $h_{i}$ (cf. Figure \ref{fig: g2h2h2'}),
the above inequality holds if and only if $t<m_{i}$. Second, assume
$t>1/q$. Then, $ia+jb>0$ if and only if $h_{i}(t)>0$ if and only
if $t>1/q$. Hence, $ia+jb>0$ if and only if $t<m_{i}$ or $t>1/q$. 

The case when $t=1/q$ can be proven by the argument in the first
paragraph of this proof. If $t=m_{i}$, then $ia+jb=0$ since $h_{i}(m_{i})=0$.
The above argument can prove the case when $i=q/2$ since $m_{q/2}=1/q$.
\end{proof}

\subsection{Relevant Critical Points of $F_{\beta}$}

In this and the next subsection, we classify nondegenerate critical
points. When we consider critical points in $\mathcal{U}_{i}$ or
$\mathcal{V}_{i}$ , we assume that $\beta>\beta_{s,\,i}$ since when
$\beta=\beta_{s,\,i}$, the elements of $\mathcal{U}_{i}=\mathcal{V}_{i}$
are degenerate. The case when $\beta=\beta_{s,\,i}$ is treated in
Section \ref{subsec: At cri temp}.

By the Morse theory, critical points with more than 2 negative eigenvalues
can be neither saddle points nor minima. Hence, the critical points
with only positive eigenvalues or only one negative eigenvalue and
$q-2$ positive eigenvalues are relevant to the landscape of $F_{\beta}$.
We select these critical points in this subsection. 

As in \eqref{e: def ab}, for $i\in[1,q/2]\cap\mathbb{N}$, $j=q-i$,
and $\beta>\beta_{s,\,i}$, when we consider ${\bf u}_{i}\in\mathcal{U}_{i}$,
let
\[
a=a({\bf u}_{i}):=-1+\frac{1}{\beta u_{i}}\ ,\ \ b=b({\bf u}_{i}):=-1+\frac{1}{\beta(1-ju_{i})}\ ,
\]
and when we consider ${\bf v}_{i}\in\mathcal{V}_{i}$, let
\[
a=a({\bf v}_{i}):=-1+\frac{1}{\beta v_{i}}\ ,\ \ b=b({\bf v}_{i}):=-1+\frac{1}{\beta(1-jv_{i})}\ .
\]

\begin{lem}
\label{l: U12V1} Let $q\ge4$. If $\beta>\beta_{s,\,1}$ , $\mathcal{U}_{1}$
is a set of local minima. If $\beta>\beta_{s,\,2}$ , $\mathcal{U}_{2}$
is a set of saddle points. If $\beta_{s,\,1}<\beta<q$, $\mathcal{V}_{1}$
is a set of saddle points else if $\beta>q$, each point in $\mathcal{V}_{1}$
has at least two negative eigenvalues.
\end{lem}

\begin{proof}
Consider ${\bf u}_{1}\in\mathcal{U}_{1}$. Eigenvalues of $\nabla^{2}F_{\beta}({\bf u}_{1})$
are $a$ with multiplicative $q-2$ and $a+(q-1)b$ with multiplicative
1. By Lemma \ref{l:sign ab}, if $\beta>\beta_{s,\,1}$, then since
$u_{1}<m_{1}<1/q$, we obtain $a,\,a+(q-1)b>0$; hence, ${\bf u}_{1}$
is a local minimum.

Next, consider ${\bf v}_{1}\in\mathcal{V}_{1}$. Eigenvalues of $\nabla^{2}F_{\beta}({\bf v}_{1})$
are $a$ with multiplicative $q-2$ and $a+(q-1)b$ with multiplicative
1. By Lemma \ref{l:sign ab}, if $\beta_{s,\,1}<\beta<q$, then since
$m_{1}<v_{1}<1/q$, we obtain $a>0$ and $a+(q-1)b<0$; hence, it
is a saddle point. If $\beta>q$, then since $v_{1}>1/q$, we obtain
$a<0$ and $a+(q-1)b>0$ so that ${\bf v}_{1}$ has more than two
negative eigenvalues.

Finally, let $i\ge2$, $j=q-i$, and $\beta>\beta_{s,\,i}$. In this
case, ${\bf u}_{i}$ has eigenvalues $a,\,b$ with multiplicative
$j-1,\,i-2$ and the roots of $\lambda^{2}-(a+qb)\lambda+b(ia+jb)$.
Since $u_{i}<m_{i}\le1/q$ for all $i$ and $\beta>\beta_{s,\,i}$,
by Lemma \ref{l:sign ab}, $a>0$, $b<0$, and $b(ia+jb)<0$ so that
it has $j$ positive eigenvalues and $i-1$ negative eigenvalues.
Hence, ${\bf u}_{2}$ is a saddle point.
\end{proof}
\begin{rem}
For $q=3$, by the same argument, $\nabla^{2}F_{\beta}({\bf v}_{1})$
has only one negative eigenvalue and two positive eigenvalues for
$\beta\in(\beta_{s,\,1},\infty)\setminus\{q\}$.
\end{rem}

\subsection{Irrelevant Critical Points of $F_{\beta}$}

In this subsection, we eliminate unneeded critical points.
\begin{lem}
\label{l: U3V2}Let $q\ge5$. For $i\in[3,q/2]\cap\mathbb{N}$ and
$\beta>\beta_{s,\,i}$, each point in $\mathcal{U}_{i}$ has at least
two negative eigenvalues. And for $i\in[2,q/2]\cap\mathbb{N}$ and
$\beta\in(\beta_{s,\,i},\infty)\setminus\{q\}$, each point in $\mathcal{V}_{i}$
has at least two negative eigenvalues.
\end{lem}

\begin{proof}
By the proof of Lemma \ref{l: U12V1}, ${\bf u}_{i}$ for $i\ge3$
has at least two negative eigenvalues. Now, let $i\ge2$, $j=q-i$,
and $\beta\in(\beta_{s,\,i},\infty)\setminus\{q\}$. In this case,
each points in $\mathcal{V}_{i}$ has eigenvalues $a,\,b$ with multiplicative
$j-1,\,i-2$, and the roots of $\lambda^{2}-(a+qb)\lambda+b(ia+jb)$.
If $\beta_{s,\,i}<\beta<q$, then $v_{i}<1/q$ so that $a>0$, $b<0$,
and $b(ia+jb)>0$. In this case,
\[
a+qb\,=\,ia+jb+(1-i)a+(q-j)b\,<\,ia+jb\,<\,0\ ,
\]
so that the two roots of $\lambda^{2}-(a+qb)\lambda+b(ia+jb)$ are
negative. Hence, it has $j-1$ positive eigenvalues and $i$ negative
eigenvalues. If $\beta>q$, then $v_{i}>1/q$ so that $a<0$, and
points in $\mathcal{V}_{i}$ have at least $j-1$ negative eigenvalues,
where $j-1\ge2$ since $q\ge5$.
\end{proof}
\begin{lem}
\label{l: V2 when q=00003D4}Let $q=4$ and $\beta\ge q$. Then, we
have $\mathcal{V}_{2}=\mathcal{U}_{2}$.
\end{lem}

\begin{proof}
Observe that $\beta_{s,\,2}=q$. If $\beta=q$, $\mathcal{V}_{2}=\mathcal{U}_{2}$
since there is only one solution $m_{2}$ to $q=g_{2}(t)$. Suppose
$\beta>q$. By elementary computation, we obtain
\[
g_{2}\Big(\frac{1}{4}-t\Big)=g_{2}\Big(\frac{1}{4}+t\Big)\ \ \ \text{for}\ t\in\Big[0,\frac{1}{4}\Big)\ ,
\]
so that $v_{2}=(1/2)-u_{2}$. Hence, ${\bf v}_{2}=(u_{2},u_{2},v_{2},v_{2})$
is a permutation of ${\bf u}_{2}$, that is, each element of $\mathcal{V}_{2}$
is one of the elements of $\mathcal{U}_{2}$ so that $\mathcal{V}_{2}=\mathcal{U}_{2}$.
\end{proof}
By lemmas in this subsection, $\mathcal{U}_{i}$, $i\ge3$, and $\mathcal{V}_{i}$,
$i\ge2$, are not of interest.

\subsection{\label{subsec: At cri temp}At Critical Temperature}

In this subsection, we investigate the critical points at the critical
temperatures, that is, at $\beta=\beta_{s,\,i}$ or $\beta=q$. The
point ${\bf u}_{i}={\bf v}_{i}$ is degenerate when $\beta=\beta_{s,\,i}$
and the point ${\bf p}={\bf v}_{i}$ is degenerate when $\beta=q$
by Lemma \ref{l:ev} and \ref{l:sign ab}.
\begin{lem}
\label{l: not minima}If $i\le q/2$ and $\beta=\beta_{s,\,i}$, the
point ${\bf u}_{i}={\bf v}_{i}$ is not a local minimum. If $\beta=q$,
the point ${\bf p}={\bf v}_{i}$ is not a local minimum.
\end{lem}

\begin{proof}
Fix $1\le i\le j\le q-1$ such that $i+j=q$ and define $\bm{\ell}_{i}:[0,1/j]\to\Xi$
as
\[
\bm{\ell}_{i}(s)\,=\,\Big(s,\dots,s,\frac{1-js}{i},\dots,\frac{1-js}{i}\Big)\ .
\]
We therefore obtain 
\begin{align*}
F_{\beta}(\bm{\ell}_{i}(s))\, & =\,-\frac{1}{2}\bigg[js^{2}+i\Big(\frac{1-js}{i}\Big)\bigg]+\frac{1}{\beta}\bigg[js\log s+(1-js)\log\Big(\frac{1-js}{i}\Big)\bigg]\\
 & =\,-\frac{1}{2i}(jqs^{2}-2js+1)+\frac{1}{\beta}\bigg[\frac{(1-js)(1-qs)}{i}g_{i}(s)+\log s\bigg]\ .
\end{align*}
By \eqref{e: Def of h_i} and \eqref{e: deriv of g,h}, we have
\[
\frac{d}{ds}F_{\beta}(\bm{\ell}_{i}(s))\,=\,\frac{j}{i}(1-qs)+\frac{j}{\beta i}(qs-1)g_{i}(s)\,=\,\frac{j}{\beta i}(1-qs)(\beta-g_{i}(s))\ .
\]
We claim that $F_{\beta_{s,\,i}}(\bm{\ell}_{i}(m_{i}))$ and $F_{q}(\bm{\ell}_{i}(1/q))$
are not the local minima of $F_{\beta_{s,\,i}}(\bm{\ell}_{i}(s))$
and $F_{q}(\bm{\ell}_{i}(s))$, respectively, and this completes the
proof.

For the first claim, assume $i<j$, and note that $m_{i}<1/q$. Then,
$1-qs>0$ and $\beta_{s,\,i}-g_{i}(s)<0$ if $s$ is in a neighborhood
of $m_{i}$ and $s\ne m_{i}$. In this case, $\frac{d}{ds}F_{\beta_{s,\,i}}(\bm{\ell}_{i}(s))<0$
near $m_{i}$ so that ${\bf u}_{i}={\bf v}_{i}$ is not a local minimum.
If $i=j$, $\beta_{s,\,i}=q$ so that it suffices to show the second
assertion.

Next, note that $v_{i}(q)=1/q$ so that we have $g_{i}(s)<\beta=q$,
$1-qs>0$ if $s<1/q$ and $g_{i}(s)>q$, $1-qs<0$ if $s>1/q$. Therefore,
$\frac{d}{ds}F_{q}(\bm{\ell}_{i}(s))>0$ near $1/q$ so that ${\bf p}={\bf v}_{i}$
is not a local minimum.
\end{proof}
Even though ${\bf u}_{i}$, $i\ge3$, is not a saddle point if $\beta>\beta_{s,\,i}$,
we cannot exclude the possibility that ${\bf u}_{i}$ is a saddle
point when $\beta=\beta_{s,\,i}$; however, by the next two lemmas,
$\mathcal{U}_{i}(\beta_{s,\,i})$, $i\ge3$, are irrelevant to the
landscape of $F_{\beta}$.
\begin{lem}
\label{l: degenerate saddle1}Let $q\ge8$ and $i\ge4$. Then, if
$\beta=\beta_{s,\,i}$, ${\bf u}_{i}={\bf v}_{i}$ is not a saddle
point.
\end{lem}

\begin{proof}
By Lemma \ref{l:ev}, $-1+1\,/\,[\beta_{s,\,i}\{-ju_{i}(\beta_{s,\,i})\}]$
is an eigenvalue of $\nabla F_{\beta_{s,\,i}}$ at ${\bf u}_{i}$
with a multiple of at least two. Hence, by Lemma \ref{l:sign ab},
it has at least two negative eigenvalues.
\end{proof}
\begin{lem}
\label{l: degenerate saddle2}Let $q\ge6$. We have $F_{\beta_{s,\,3}}({\bf u}_{3})>F_{\beta_{s,\,3}}({\bf u}_{2})$.
Furthermore, if $q\ge7$, we have $F_{\beta_{s,\,3}}({\bf u}_{3})>F_{\beta_{s,\,3}}({\bf v}_{1})$.
Hence, ${\bf u}_{3}$ cannot be a saddle point lower than ${\bf u}_{2}$
or ${\bf v}_{1}$. 
\end{lem}

The proof is presented in Section \ref{subsec: Pf  l: M-U2-V1}. We
remark that if $q=6$, we have $\beta_{s,\,3}=q$ so that ${\bf v}_{1}(\beta_{s,\,3})={\bf p}$
and the second assertion is not needed.

\section{\label{sec: Pf beta}Analysis of Energy Landscape}

In this section, we prove lemmas introduced in Section \ref{subsec: pre beta}
and Lemma \ref{l: degenerate saddle2}. To prove these lemmas, we
need numerical computation given in Appendix \ref{sec: Numerical}.

\subsection{\label{subsec: Pf  l:beta_s<beta_c2}Proof of Lemma \ref{l: beta_s1<beta_c}}
\begin{lem}
\label{l: v1 > 1/2(q-1)}If $q\ge4$, we have $v_{1}(\beta_{s,\,2})>\frac{1}{2(q-1)}$.
\end{lem}

\begin{proof}
Fix $\beta=\beta_{s,\,2}$ and write $v_{1}=v_{1}(\beta_{s,\,2})$
for convenience. Since $\beta_{s,\,2}=g_{2}(m_{2})=g_{1}(v_{1})$,
we have
\begin{equation}
\frac{2}{1-qm_{2}}\log\frac{1-(q-2)m_{2}}{2m_{2}}\,=\,\frac{1}{1-qv_{1}}\log\frac{1-(q-1)v_{1}}{v_{1}}\label{e: g2m2=00003Dg1v1}
\end{equation}
Let 
\begin{equation}
v_{1}^{*}\,=\,\frac{1}{2q}+\frac{m_{2}}{2}\ ,\ \text{so that}\ \frac{1}{1-qv_{1}^{*}}\,=\,\frac{2}{1-qm_{2}}\ .\label{e: def of v^star}
\end{equation}
We claim that $g_{1}(v_{1}^{*})\,\le\,g_{1}(v_{1})$, that is, by
\eqref{e: g2m2=00003Dg1v1},
\[
\frac{1}{1-qv_{1}^{*}}\log\frac{1-(q-1)v_{1}^{*}}{v_{1}^{*}}\,\le\,\frac{2}{1-qm_{2}}\log\frac{1-(q-2)m_{2}}{2m_{2}}\ .
\]
By \eqref{e: def of v^star}, the above inequality is equivalent to
\[
\frac{1-(q-1)v_{1}^{*}}{v_{1}^{*}}\le\frac{1-(q-2)m_{2}}{2m_{2}}\ .
\]
By plugging $v_{1}^{*}$ given in \eqref{e: def of v^star} into this
inequality, it becomes $q^{2}m_{2}-2qm_{2}+1\ge0$. Hence, since $g_{1}$
is increasing at $v_{1}$, we obtain $v_{1}^{*}\le v_{1}$.

Finally, we claim that 
\[
v_{1}^{*}=\frac{1+qm_{2}}{2q}>\frac{1}{2(q-1)}\ ,\ \text{i.e.,}\ m_{2}>\frac{1}{q(q-1)}\ .
\]
According to Figure \ref{fig: g2h2h2'}, we can show this by
\[
h_{2}\Big(\frac{1}{q(q-1)}\Big)=\log\frac{q^{2}-2q+2}{2}-\frac{q(q-1)(q-2)}{q^{2}-2q+2}<0\ .
\]
This holds if $q=4$ or $q=5$ by elementary computation. Now, assume
$q\ge6$. Therefore, we obtain
\[
\log\frac{q^{2}-2q+2}{2}<\log q^{2}=2\log q<q-2<\frac{q(q-1)(q-2)}{q^{2}-2q+2}\ ,
\]
which completes the proof.
\end{proof}
We can prove Lemma \ref{l: beta_s1<beta_c} by the aforementioned
lemma.
\begin{proof}[Proof of Lemma \ref{l: beta_s1<beta_c}]
Since $\beta_{c}=g_{1}(\frac{1}{q(q-1)})=g_{1}(\frac{1}{2(q-1)})$,
we have $\beta_{s,\,1}<\beta_{c}$. By Lemma \ref{l: v1 > 1/2(q-1)},
since $g_{1}(t)$ is increasing on $(m_{1},1/(q-1)\,)$ and $m_{1}<1/(2q-2)$,
we obtain
\[
\beta_{s,\,2}=g_{1}(v_{1})>g_{1}(\frac{1}{2(q-1)})=\beta_{c}\ .
\]
\end{proof}

\subsection{\label{subsec: Pf l: beta_m}Proof of Lemma \ref{l: beta_m}}

We first introduce two lemmas.
\begin{lem}
\label{l: M-U2-V1}Let $q\ge5$. When $\beta=\beta_{s,\,2}$, we have
$F_{\beta_{s,\,2}}({\bf v}_{1})<F_{\beta_{s,\,2}}({\bf u}_{2})$ and
when $\beta=q$, we have $F_{q}({\bf v}_{1})=F_{q}({\bf p})>F_{q}({\bf u}_{2})$.
\end{lem}

The proof of the above lemma is given in Section \ref{subsec: Pf  l: M-U2-V1}.
\begin{lem}
\label{l: U2-V1 decre}Let $q\ge5$. $\beta^{2}\frac{d}{d\beta}[F_{\beta}({\bf u}_{2})-F_{\beta}({\bf v}_{1})]$
decreases as $\beta$ increases in $(\beta_{s,\,2},q)$.
\end{lem}

\begin{proof}
For $t=t(\beta)$, which satisfies $\beta=g_{i}(t)$, let 
\begin{equation}
{\bf c}_{i}\,=\,{\bf c}_{i}(\beta)\,=\,\Big(t,\dots,t,\frac{1-jt}{it},\dots,\frac{1-jt}{it}\Big)\ .\label{e: e.g. cri}
\end{equation}
Since ${\bf c}_{i}$ is a critical point, by the proof of Corollary
\ref{c: 2nd PT}, we have
\[
\frac{d}{d\beta}F_{\beta}({\bf c}_{i})=-\frac{1}{\beta^{2}}S({\bf c}_{i})\ .
\]
Define a function $k_{i}:(0,1)\to\mathbb{R}$ as 
\begin{equation}
k_{i}(t):=(1-jt)\log\frac{1-jt}{it}+\log t\ .\label{e: Def of k_i}
\end{equation}
By elementary computations, we obtain $S({\bf c}_{i})=k_{i}(t)$ so
that we have
\begin{equation}
\frac{d}{d\beta}F_{\beta}({\bf c}_{i})\,=\,-\frac{1}{\beta^{2}}k_{i}(t)\ .\label{e: F_deriv(c(t))}
\end{equation}

Now, by \eqref{e: F_deriv(c(t))}, we obtain
\begin{equation}
\beta^{2}\frac{d}{d\beta}\Big[F_{\beta}({\bf u}_{2})-F_{\beta}({\bf v}_{1})\Big]\,=\,k_{1}(v_{1}(\beta))-k_{2}(u_{2}(\beta))\ .\label{e: U2-V1 last}
\end{equation}
Observe that the value $u_{2}(\beta)$ decreases and the value $v_{1}(\beta)$
increases as $\beta$ increases. By elementary computation, for $t\in(0,1/q)$,
we obtain
\begin{multline}
k_{i}'(t)=-j\log\frac{1-jt}{it}+(1-jt)\Big(\frac{-j}{1-jt}-\frac{1}{t}\Big)+\frac{1}{t}=-j\log\frac{1-jt}{it}<0\ ,\label{e: k decrea}
\end{multline}
so that $k_{i}(t)$ decreasing on $(0,1/q)$. Hence, \eqref{e: U2-V1 last}
decreases as $\beta$ increases in $(\beta_{s,\,2},q)$.
\end{proof}
We can now prove Lemma \ref{l: beta_m}.
\begin{proof}[Proof of Lemma \ref{l: beta_m}]
By Lemma \ref{l: M-U2-V1}, there is $\beta_{0}\in(\beta_{s,\,2},q)$,
such that $\frac{d}{d\beta}[F_{\beta}({\bf u}_{2})-F_{\beta}({\bf v_{1}})]<0$.
Hence, by Lemma \ref{l: U2-V1 decre}, we can deduce that there is
only one critical value $\beta_{m}\in(\beta_{s,\,2},q)$, such that
\begin{equation}
F_{\beta_{m}}({\bf u}_{2})=F_{\beta_{m}}({\bf v}_{1})\ .\label{e: Def of beta_m}
\end{equation}
\end{proof}

\subsection{\label{subsec: Pf  l: M-U2-V1}Proofs of Lemmas \ref{l: degenerate saddle2}
and \ref{l: M-U2-V1}}

Before we go further, we conduct some computations. Recall the definition
of $m_{2}$ from Lemma \ref{l: g_i}. Since $\beta_{s,\,i}=g_{i}(m_{i})=\frac{i}{1-qm_{i}}\log\frac{1-jm_{i}}{im_{i}}$
and $m_{i}$ is the minimum of $g_{i}$, we have
\begin{align*}
0=h_{i}(m_{i})\, & =\,\log\frac{1-jm_{i}}{im_{i}}+\frac{qm_{i}-1}{qm_{i}(1-jm_{i})}\\
 & =\,\frac{1-qm_{i}}{i}\beta_{s,\,i}-\frac{1-qm_{i}}{qm_{i}(1-jm_{i})}\ ,
\end{align*}
so that
\begin{equation}
qjm_{i}^{2}-qm_{i}\,=qm_{i}(jm_{i}-1)\,=\,-\frac{i}{\beta_{s,\,i}}\ .\label{e: m_i-beta_si}
\end{equation}
For ${\bf c}_{i}$ defined in \eqref{e: e.g. cri}, since $S({\bf c}_{i})=k_{i}(t)$
and $\beta=g_{i}(t)$, we can write
\begin{equation}
F_{\beta}({\bf c}_{i})=\,\frac{1}{2i}\Big[qjt^{2}-2qt+1\Big]+\frac{1}{\beta}\log t\ .\label{e: F(c(t))}
\end{equation}
Hence, by \eqref{e: m_i-beta_si} and $\beta_{s,\,i}=g_{i}(m_{i})$,
we have
\begin{align}
F_{\beta_{s,\,i}}({\bf u}_{i}) & =\,\frac{1-qm_{i}}{2i}+\frac{1}{\beta_{s,\,i}}\Big(\log m_{i}-\frac{1}{2}\Big)\label{e: F_beta_si (ui)1}\\
 & =\,\frac{1}{2\beta_{s,\,i}}\log\frac{1-jm_{i}}{im_{i}}+\frac{1}{\beta_{s,\,i}}\log m_{i}-\frac{1}{2\beta_{s,\,i}}\ .\nonumber 
\end{align}
By \eqref{e: m_i-beta_si} again, we obtain
\begin{equation}
F_{\beta_{s,\,i}}({\bf u}_{i})=-\frac{1}{2\beta_{s,\,i}}\log(qe\beta_{s,\,i})\ .\label{e: F_beta_si (ui)2}
\end{equation}

Now, we introduce two technical lemmas required in the proof of Lemmas
\ref{l: degenerate saddle2} and \ref{l: M-U2-V1}. 
\begin{lem}
\label{l: Last claim1}For $q\ge6500$, we have
\[
\frac{1}{\beta_{s,\,2}}\Big(\log qm_{2}-\frac{1}{2}\Big)>\frac{(q-1)}{8q}(qm_{2})^{2}-\frac{1}{4}m_{2}+\frac{-q^{2}+4q+1}{8q(q-1)}\ .
\]
\end{lem}

The proof is given in Section \ref{sec: Pf of Last}.
\begin{lem}
\label{l: Last claim2}Let $q\ge5$. Define $f_{c}(\beta)=-\frac{1}{2\beta}\log(qe\beta)$
and 
\[
\Phi(\beta)=\frac{d}{d\beta}[f_{c}(\beta)-F_{\beta}({\bf u}_{2})]\ .
\]
Then, we have $\Phi(\beta)>0$ for $\beta>\beta_{s,\,2}$.
\end{lem}

\begin{proof}
We have
\begin{align*}
\frac{d}{d\beta}f_{c}(\beta) & =\,\frac{1}{2\beta^{2}}\log qe\beta-\frac{1}{2\beta}\frac{1}{\beta}\,=\,\frac{1}{2\beta^{2}}\log q\beta\ .
\end{align*}
By \eqref{e: F_deriv(c(t))}, we obtain
\[
\beta^{2}\frac{d}{d\beta}[f_{c}(\beta)-F_{\beta}({\bf u}_{2})]\,=\,\frac{1}{2}[\log q\beta+2k_{2}(u_{2})]\ .
\]
By \eqref{e: k decrea}, the above expression is increasing function
of $\beta$ since $u_{2}$ decreases as $\beta$ increases. Hence,
it is sufficient to show $\Phi(\beta_{s,\,2})>0$. First, let $q\ge55>e^{4}$.
By \eqref{e: k decrea}
\begin{align*}
\log q\beta+2k_{2}(u_{2})\, & >\,\log q\beta_{s,\,2}+2k_{2}(\frac{1}{2j})\\
 & =\,\log q\beta_{s,\,2}+\log\frac{2j-j}{i}+2\log\frac{1}{2j}\,=\,\log\frac{q\beta_{s,\,2}}{4ij}\ ,
\end{align*}
where we use $u_{2}<1/(2j)$ for the inequality. Since $\beta_{s,\,2}>\beta_{c}>2\log q$,
we obtain
\[
\frac{q\beta_{s,\,2}}{4ij}\,>\,\frac{2q\log q}{8(q-2)}\,>\,\frac{q}{q-2}\ .
\]
 Finally, for $5\le q\le54$, by Proposition \ref{p: numerical},
we have $\Phi(\beta_{s,\,2})>0$.
\end{proof}
By the above lemmas, Lemma \ref{l: M-U2-V1} can be proven.
\begin{proof}[Proof of Lemma \ref{l: M-U2-V1}]
By Proposition \ref{p: numerical} given in appendix, we can check
that $F_{\beta_{s,\,2}}({\bf u}_{2})>F_{\beta_{s,\,2}}({\bf v}_{1})$
holds for $5\le q\le6500$. Now, suppose that $q>6500$. By \eqref{e: F(c(t))}
and \eqref{e: F_beta_si (ui)1}, we can write
\begin{align*}
F_{\beta_{s,\,2}}({\bf u}_{2}) & =\,-\frac{1}{4}qm_{2}+\frac{1}{4}+\frac{1}{\beta_{s,\,2}}\Big(\log m_{2}-\frac{1}{2}\Big)\ ,\\
F_{\beta_{s,\,2}}({\bf v}_{1}) & =\,\frac{1}{2}\bigg[q(q-1)\Big(v_{1}-\frac{1}{q-1}\Big)^{2}-\frac{1}{q-1}\bigg]+\frac{1}{\beta_{s,\,2}}\log v_{1}\ .
\end{align*}
By the proof of Lemma \ref{l: v1 > 1/2(q-1)}, we have
\[
\frac{qm_{2}+1}{2q}=v_{1}^{*}\le v_{1}<\frac{1}{q}\ ,
\]
so that
\[
F_{\beta_{s,\,2}}({\bf v}_{1})<\frac{1}{2}\left[q(q-1)\Big(\frac{qm_{2}+1}{2q}-\frac{1}{q-1}\Big)^{2}-\frac{1}{q-1}\right]-\frac{1}{\beta_{s,\,2}}\log q\ .
\]
Hence, the lemma can be proven if we can prove
\begin{align*}
 & -\frac{1}{4}qm_{2}+\frac{1}{4}+\frac{1}{\beta_{s,\,2}}(\log m_{2}-\frac{1}{2})\\
 & >\,\frac{1}{2}\bigg[q(q-1)\Big(\frac{qm_{2}+1}{2q}-\frac{1}{q-1}\Big)^{2}-\frac{1}{q-1}\bigg]-\frac{1}{\beta_{s,\,2}}\log q\\
 & =\,\frac{1}{8}q(q-1)(m_{2})^{2}-\frac{1}{4}(q+1)m_{2}+\frac{(q+1)^{2}}{8q(q-1)}-\frac{1}{\beta_{s,\,2}}\log q\ .
\end{align*}
This is the content of Lemma \ref{l: Last claim1}. Finally, by Lemma
\ref{l: Last claim2}, we obtain $F_{q}({\bf p})-F_{q}({\bf u}_{2})=f_{c}(q)-F_{q}({\bf u}_{2})>0$
since $f_{c}(\beta_{s,\,2})=F_{\beta_{x,\,2}}({\bf u}_{2})$.
\end{proof}
Now, we prove Lemma \ref{l: degenerate saddle2}.
\begin{proof}[Proof of Lemma \ref{l: degenerate saddle2}]
Since the proof for $F_{\beta_{s,\,3}}({\bf u}_{3})>F_{\beta_{s,\,3}}({\bf v}_{1})$
is exactly the same as the proof of Lemma \ref{l: M-U2-V1} including
numerical verification, we omit it. By \eqref{e: F_beta_si (ui)2},
we can write
\[
F_{\beta_{s,\,3}}({\bf u}_{3})\,=\,f_{c}(\beta_{s,\,3})\ .
\]
 Hence, by Lemma \ref{l: Last claim2} and by Proposition \ref{p: numerical},
we have
\[
F_{\beta_{s,\,3}}({\bf u}_{3})\,=\,f_{c}(\beta_{s,\,3})\,>\,F_{\beta_{s,\,3}}({\bf u}_{2})\ .
\]
\end{proof}

\section{\label{sec: Pf metastates}Characterization of Metastable Sets}

In this section, we prove Theorems \ref{t: metasets all q}-\ref{t: metasets}.
First, we prove Theorem \ref{t: metasets all q}.
\begin{proof}[Proof of Theorem \ref{t: metasets all q}]
The first assertion is immediate from Lemmas \ref{l: cand cri} and
\ref{l: beta_s1<beta_s2}. The third assertion is proven by Proposition
\ref{p: cri} and Lemma \ref{l: not minima}. The fourth assertion
is Lemma \ref{l: beta_c}.

Now, it remains to show the second assertion. For $\beta\in(\beta_{1},\beta_{2}]$,
since ${\bf p}$ is the global minimum and ${\bf v}_{1}$ is a saddle
point, we have $F_{\beta}({\bf p})<F_{\beta}({\bf v}_{1})$ so that
$\mathcal{W}_{\frak{o}}\ne\emptyset$. By the same argument in the
proof of Lemma \ref{l: U2-V1 decre}, we have 
\[
\frac{d}{d\beta}[F_{\beta}({\bf v}_{1})-F_{\beta}({\bf p})]=-\frac{1}{\beta^{2}}[k_{1}(v_{1}(\beta))+\log q]\ .
\]
By \ref{e: k decrea}, $k_{1}(\cdot)$ is decreasing on $(0,\,1/q)$
and increasing on $(1/q,\,1/(q-1)\,)$. Since $k_{1}(1/q)=-\log q$,
we have $k_{1}(v_{1}(\beta))+\log q>0$ for $\beta\in(\beta_{1},q)$
so that
\[
\frac{d}{d\beta}[F_{\beta}({\bf v}_{1})-F_{\beta}({\bf p})]<0\ .
\]
Since ${\bf v}_{1}={\bf p}$ when $\beta=q$, we have $F_{\beta}({\bf v}_{1})>F_{\beta}({\bf p})$
for $\beta<q$ and $F_{\beta}({\bf v}_{1})<F_{\beta}({\bf p})$ for
$\beta>q$.
\end{proof}

\subsection{\label{subsec: Pf p:metasets1}Proof of Theorem \ref{t: metasets}}

Before we go further, we recall the height between two points. Let
$\bm{a},\,\bm{b}\in\text{int}\,\Xi$, and let $\Gamma_{\bm{a},\,\bm{b}}$
be a set of all $C^{1}$-path $\gamma:[0,1]\to\text{int}\,\Xi$, such
that $\gamma(0)=\bm{a}$ and $\gamma(1)=\bm{b}$. Then, we can define
the height $\frak{H}(\bm{a},\bm{b})$ between $\bm{a}$ and $\bm{b}$
as $\frak{H}(\bm{a},\bm{b})=\inf_{\gamma\in\Gamma_{\bm{a},\,\bm{b}}}\,\sup_{0\le t\le1}\,F_{\beta}(\gamma(t))$.
We prove Theorem \ref{t: metasets} in several steps.
\begin{lem}
\label{l: metasets1}Let $q\ge4$. If $\beta>\beta_{m}$, the sets
$\mathcal{W}_{i}(\beta)$, $i\in S$, are different. In particular,
they are nonempty.
\end{lem}

\begin{proof}
Since the elements of $\mathcal{U}_{1}$ are the lowest minima, we
have $F_{\beta}({\bf u}_{1})<H_{\beta}$ so that $\mathcal{W}_{i}$'s
are nonempty. Without loss of generality, suppose $\mathcal{W}_{1}=\mathcal{W}_{2}$.
Since ${\bf u}_{1}^{1},\,{\bf u}_{1}^{2}\in\mathcal{W}_{1}$ and $\mathcal{W}_{1}$
is connected, there is a $C^{1}$-path $\gamma:[0,1]\to\mathcal{W}_{1}$,
such that $\gamma(0)={\bf u}_{1}^{1}$, $\gamma(1)={\bf u}_{1}^{2}$.
Therefore, we have $F_{\beta}(\gamma(t))<H_{\beta}$ for $0\le t\le1$,
so that
\[
F_{\beta}({\bf u}_{1}^{1})<\frak{H}({\bf u}_{1}^{1},{\bf u}_{1}^{2})<H_{\beta}\ .
\]
Then, there is a saddle point $\bm{\sigma}({\bf u}_{1}^{1},{\bf u}_{1}^{2})$,
such that $F_{\beta}(\bm{\sigma}({\bf u}_{1}^{1},{\bf u}_{1}^{2}))=\frak{H}({\bf u}_{1}^{1},{\bf u}_{1}^{2})$.
However, by Proposition \ref{p: cri}, the values of saddle points
are greater than or equal to $H_{\beta}$. This is contradiction.
Hence, $\mathcal{W}_{i}$'s are different.
\end{proof}
\begin{lem}
\label{l: metasets2}Let $q\ge4$. If $\beta>q$, the set $\Sigma_{i,\,j}$
is singleton for all $i,j\in S$.
\end{lem}

\begin{proof}
First, we claim that $\Sigma_{i,\,j}$'s are not empty. Suppose one
of $\Sigma_{i,\,j}$'s is empty. Then, by symmetry, all of them are
empty. Let us fix $1\le k<l\le q$. Since ${\bf u}_{2}^{k,\,l}$ is
a saddle point, there is a unit eigenvector $\bm{w}$ that corresponds
to the unique negative eigenvalue of $\nabla^{2}F_{\beta}({\bf u}_{2}^{k,\,l})$.
There exists $\eta>0$, such that $F_{\beta}({\bf u}_{2}^{k,\,l}+t\bm{w})<H_{\beta}$
for all $0<|t|<\eta$. Now, consider the path $\bm{y}(t)$ described
by the ordinary differential equation
\begin{equation}
\dot{\bm{y}}(t)\,=\,-\nabla F_{\beta}(\bm{y}(t)),\ \ \ \bm{y}(0)\,=\,{\bf u}_{2}^{k,\,l}+\eta\bm{w}\ .\label{e: decre path}
\end{equation}
Then, $\bm{y}(t)$ converges to a critical point whose height is less
than $H_{\beta}$ as $t\to\infty$. If this convergent point is not
a local minimum, we can find an eigenvector $\bm{w}_{1}$ corresponding
to a negative eigenvalue of the Hessian of $F_{\beta}$ at that point.
Then, by the same argument defining the path \eqref{e: decre path},
the next path converges to another critical point whose height is
lower than that of the previous critical point. Finally, this path
converges to a local minimum. Since there is no local minimum other
than $\mathcal{U}_{1}$, $\bm{y}(t)$ converges to some elements of
$\mathcal{U}_{1}$, say ${\bf u}_{1}^{1}$ without loss of generality.
Since $\mathcal{W}_{i}$'s are different, $\bm{y}(\cdot)$ converges
to only one minimum.

By the same argument, the similar path starting at ${\bf u}_{2}^{k,\,l}-\epsilon\bm{w}$
converges to some ${\bf u}_{1}$, say ${\bf u}_{1}^{m}$. If $m\ne1$,
${\bf u}_{2}^{k,\,l}\in\Sigma_{1,\,m}$ so that $\Sigma_{1,\,m}$
is not empty. So, we obtain $m=1$. In this case, we obtain ${\bf u}_{2}^{k,\,l}\in\overline{\mathcal{W}_{1}}$
and ${\bf u}_{2}^{k,\,l}\notin\overline{\mathcal{W}}_{m}$ for $m\ne1$.
By symmetry, since $\mathcal{U}_{2}$ has $q(q-1)/2$ elements and
the number of $\mathcal{W}_{i}$ is $q$, there are $(q-1)/2$ elements
in $\mathcal{U}_{2}$ corresponding to each $\mathcal{W}_{i}$, that
is, $|\overline{\mathcal{W}_{1}}\cap\mathcal{U}_{2}|=(q-1)/2$, where
$|A|$ is the number of elements of set $A$. If ${\bf u}_{2}^{1,\,a}\in\overline{\mathcal{W}_{1}}$,
for some $2\le a\le q$, we obtain ${\bf u}_{2}^{1,\,a}\in\overline{\mathcal{W}_{a}}$
by symmetry, and therefore $\Sigma_{1,\,a}=\overline{\mathcal{W}_{1}}\cap\overline{\mathcal{W}_{a}}\neq\emptyset$.
Hence, we have ${\bf u}_{2}^{1,\,a}\notin\overline{\mathcal{W}_{1}}$.
If ${\bf u}_{2}^{a,\,b}\in\overline{\mathcal{W}_{1}}$ for some $1<a,\,b$,
since $q\ge4$ and by symmetry, ${\bf u}_{2}^{a,\,b}\in\overline{\mathcal{W}_{m}}$
for some $m\ne2,\,a,\,b$, and this contradicts the assumption that
$\Sigma_{1,\,m}=\overline{\mathcal{W}_{1}}\cap\overline{\mathcal{W}_{m}}=\emptyset$.
Hence, $\Sigma_{i,\,j}$'s are nonempty. 

Observe that the elements of $\Sigma_{i,\,j}$ are saddle points and
$F_{\beta}(\bm{x})=H_{\beta}$ for all $\bm{x}\in\Sigma_{i,\,j}$.
Hence, by Proposition \ref{p: cri}, $\Sigma_{i,\,j}\subset\mathcal{U}_{2}$.
Since $\nabla^{2}F_{\beta}({\bf u}_{2})$'s have only one negative
eigenvalue, each element of $\mathcal{U}_{2}$ connects only two wells,
i.e., $\Sigma_{i,\,j}\cap\Sigma_{k,\,l}=\emptyset$ if $\{i,j\}\ne\{k,l\}$.
Therefore, $\Sigma_{i,\,j}$ has at most one point and from the above
two paragraphs, we obtain $|\Sigma_{i,\,j}|=1$.
\end{proof}
We can now prove Theorem \ref{t: metasets}.
\begin{proof}[Proof of Theorem \ref{t: metasets}]
The first assertion follows from the definition of critical temperatures
\eqref{e: Def of cri temp} and Lemma \ref{l: beta_m}.

By Lemma \ref{l: metasets2}, to prove $\Sigma_{i,\,j}=\{{\bf u}_{2}^{i,\,j}\}$
when $\beta>q$, without loss of generality, it is sufficient to show
that $\Sigma_{1,\,2}\ne\{{\bf u}_{2}^{1,\,4}\}$ and $\Sigma_{1,\,2}\ne\{{\bf u}_{2}^{3,\,4}\}$.
First, suppose $\Sigma_{1,\,2}=\{{\bf u}_{2}^{1,\,4}\}$. Then, by
symmetry, we obtain ${\bf u}_{2}^{1,\,4}\in\Sigma_{1,\,3}$, which
contradicts to $\Sigma_{1,\,2}\cap\Sigma_{1,\,3}=\emptyset$. Second,
suppose $\Sigma_{1,\,2}=\{{\bf u}_{2}^{3,\,4}\}$ so that by symmetry,
we have $\Sigma_{1,\,5}=\{{\bf u}_{2}^{3,\,4}\}$ which is also contradiction.
Hence, we obtain $\Sigma_{1,\,2}=\{{\bf u}_{2}^{1,\,2}\}$.

Since $F_{\beta}$ is continuous in $\beta$, the values $H_{\beta}$
and $\frak{H}({\bf u}_{1}^{i}(\beta),{\bf u}_{1}^{j}(\beta))$, $i,j\in S$,
are continuous in $\beta$. Note that $\frak{H}({\bf u}_{1}^{i}(\beta),{\bf u}_{1}^{j}(\beta))=F_{\beta}({\bf u}_{2})=H_{\beta}$
for $\beta\ge q$ since there is no saddle point other than $\mathcal{U}_{2}$.
Since $F_{\beta}({\bf v}_{1})>H_{\beta}$ if $\beta>\beta_{m}=\beta_{3}$
and there is no saddle point other than the elements of $\mathcal{U}_{2}\cup\mathcal{V}_{1}$,
by continuity, we obtain
\[
\frak{H}({\bf u}_{1}^{i}(\beta),{\bf u}_{1}^{j}(\beta))\,=\,H_{\beta}\ \text{if}\ \beta\ge\beta_{3}\ .
\]
Hence, ${\bf u}_{2}^{i,\,j}$ is a saddle point between ${\bf u}_{1}^{i}$
and ${\bf u}_{1}^{j}$ and $\Sigma_{i,\,j}=\{{\bf u}_{2}^{i,\,j}\}$
if $\beta>\beta_{3}$. Coupled with Lemma \ref{l: metasets1}, the
fourth assertion holds except that $\Sigma_{0,\,i}=\emptyset$.

Without loss of generality, suppose that $\Sigma_{0,\,1}=\overline{\mathcal{W}_{0}}\cap\overline{\mathcal{W}_{1}}\neq\emptyset$.
We, therefore, obtain $\frak{H}({\bf p},{\bf u}_{1}^{1}(\beta))<F_{\beta}({\bf v}_{1})$
so that $\frak{H}({\bf p},{\bf u}_{1}^{1}(\beta))=H_{\beta}$. By
continuity, we get
\[
\frak{H}({\bf p},{\bf u}_{1}^{1}(\beta))=H_{\beta}\ \text{for}\ \beta_{3}\le\beta<q\ ,
\]
so that $F_{\beta}({\bf p})\le H_{\beta}$ for $\beta_{3}\le\beta<q$.
However, it is in contradiction to $F_{q}({\bf p})=F_{q}({\bf v}_{1})>H_{q}$.
Hence, we obtain $\Sigma_{0,\,i}=\emptyset$ for $i\in S$.

By the same argument and symmetry, the second assertion can be proven
for $\beta\in(\beta_{s,\,1},\beta_{s,\,2})=(\beta_{1},\beta_{s,\,2})$.
By continuity argument, we can extend these to $\beta\in(\beta_{1},\beta_{3})$.
The third assertion holds because of the first and fourth assertions,
symmetry, and continuity. Finally, the fifth assertion can be proven
by the same argument.
\end{proof}

\subsection{\label{subsec: Pf p:metasets2}Proof of Theorem \ref{t: metasets q4}}

If $q=4$, $\Sigma_{1,\,2}\ne\{{\bf u}_{2}^{3,\,4}\}$ cannot be proven
by symmetry argument.  Hence, we directly prove the Theorem \ref{t: metasets q4}.
\begin{proof}[Proof of Theorem \ref{t: metasets q4}]
By Lemma \ref{l: beta_m} and \eqref{e: Def of cri temp}, we obtain
the first assertion.

Consider $\mathcal{K}_{i,\,j}=\{\,\bm{x}\in\Xi\,:\,x_{i}=x_{j}=\max\{x_{1,}\dots,x_{4}\}\,\}$.
It can be observed that these six planes divide $\Xi$ into four pieces,
and each plain contains one element of $\mathcal{U}_{2}$ and ${\bf u}_{2}^{i,\,j}\in\mathcal{K}_{i,\,j}$.
We claim that $H_{\beta}=F_{\beta}({\bf u}_{2}^{i,\,j})<F_{\beta}(\bm{x})$
for all $\bm{x}\in\mathcal{K}_{i,\,j}$ if $\beta>q$. Note that ${\bf p}$
is not local minimum if $\beta\ge q$.

Let $\widetilde{F}_{\beta}(\bm{x})$ be a restriction of $F_{\beta}$
to $\mathcal{K}_{3,\,4}$ and let $\mathcal{K}_{3,\,4}^{o}=\{\bm{x}\in\mathcal{K}_{3,\,4}\,:\,x_{3}=x_{4}>x_{1},x_{2}\}$.
Since $x_{3}=x_{4}=\frac{1}{2}(1-x_{1}-x_{2})$, 
\[
\frac{\partial}{\partial x_{i}}\widetilde{F}_{\beta}(\bm{x})=-x_{i}+\frac{1}{\beta}\log x_{i}+x_{3}-\frac{1}{\beta}\log x_{3}\ ,
\]
so that if $\bm{x}\in\mathcal{K}_{3,\,4}$ is a critical point, we
have
\[
-x_{1}+\frac{1}{\beta}\log x_{1}=-x_{2}+\frac{1}{\beta}\log x_{2}=-x_{3}+\frac{1}{\beta}\log x_{3}\ .
\]
Since $x_{3}=x_{4}>x_{1},\,x_{2}$, if $\beta\ge q$, the critical
points in $\mathcal{K}_{3,\,4}^{o}$ are ${\bf u}_{2}^{3,\,4}$, ${\bf v}_{2}^{1,\,2}$.
From the proof Lemma \ref{l: V2 when q=00003D4}, we obtain ${\bf u}_{2}^{3,\:4}={\bf v}_{2}^{1,\,2}=(u_{2},u_{2},v_{2},v_{2})$.

Let $a=-1+\frac{1}{\beta u_{2}}$ and $b=-1+\frac{1}{\beta v_{2}}$.
We therefore obtain
\[
\nabla^{2}\widetilde{F}_{\beta}({\bf u}_{2}^{3,\,4})\ =\ \left(\begin{array}{cc}
a+\frac{1}{2}b & \frac{1}{2}b\\
\frac{1}{2}b & a+\frac{1}{2}b
\end{array}\right)\ .
\]
The eigenvalues of $\nabla^{2}\widetilde{F}_{\beta}({\bf u}_{2}^{3,\,4})$
are $a$ and $a+b$. By Lemma \ref{l:sign ab}, $a,\,b>0$ so that
${\bf u}_{2}^{3,\,4}$ is a local minimum in $\mathcal{K}_{3,\,4}^{o}$.
Since this is the unique critical point, ${\bf u}_{2}^{3,\,4}$ is
the unique minimum in $\mathcal{K}_{3,\,4}^{o}$. Since $\mathcal{K}_{3,\,4}$
is a closure of $\mathcal{K}_{3,\,4}^{o}$ and there is no critical
point in $\mathcal{K}_{3,\,4}^{o}\setminus\{{\bf u}_{2}^{3,\,4}\}$,
${\bf u}_{2}^{3,\,4}$ is the unique minimum in $\mathcal{K}_{3,\,4}$.
Hence, $\mathcal{W}_{i}$'s are different if $\beta>q$.

Let $\beta>q$. By the definition of $\mathcal{K}_{i,\,j}$, we obtain
$\overline{\mathcal{W}}_{k}\cap\mathcal{K}_{i,\,j}=\emptyset$ if
$k\ne i,j$ so that $\Sigma_{i,\,j}\subset\mathcal{K}_{i,\,j}$. By
Lemma \ref{l: metasets2}, $\Sigma_{i,\,j}$ are not empty. It can
be observed $F_{\beta}(\bm{x})=H_{\beta}$ and $\nabla F_{\beta}(\bm{x})=0$
if $\bm{x}\in\Sigma_{i,\,j}$. Since $\Sigma_{i,\,j}\subset\mathcal{K}_{i,\,j}$,
we have $\Sigma_{i,\,j}=\{{\bf u}_{2}^{i,\,j}\}$, thus the fourth
assertion is proved.

For the third assertion, note that $F_{q}(\bm{x})=H_{q}$ for all
$\bm{x}\in\Sigma_{i,\,j}$ and ${\bf p}$ is the only point in $\mathcal{K}_{i,\,j}$,
such that $F_{q}(\bm{x})=H_{q}$. Moreover, we obtain $F_{q}(\bm{x})>H_{q}=F_{q}({\bf p})$
if $\bm{x}\in\mathcal{K}_{i,\,j}^{o}$, and finally we can deduce
$F_{q}(\bm{x})>H_{q}=F_{q}({\bf p})$ if $\bm{x}\in\mathcal{K}_{i,\,j}\setminus\{{\bf p}\}$
using elementary calculus. Hence, $\mathcal{W}_{i}$'s are different
if $\beta=q$.

For the second assertion, we can use the symmetry argument and the
proofs are the same as the proof of Theorem \ref{t: metasets}.
\end{proof}

\section{\label{sec: Pf of Last}Proof of Lemma \ref{l: Last claim1}}

This section is devoted to the proof of Lemma \ref{l: Last claim1}.
In Section \ref{subsec: Last1}, we provide an auxiliary lemma to
prove Lemma \ref{l: deriv q}. In section \ref{subsec: Last2}, we
prove this auxiliary lemma. So far, we have fixed an integer $q\ge3$;
however, in this section, we consider $q$ as a real number and several
variables as functions of $q$. For example, $m_{2}=m_{2}(q)$, $j(q)=q-2$,
and $\beta_{s,\,2}=\beta_{s,\,2}(q)$.

\subsection{\label{subsec: Last1}Proof of Lemma \ref{l: Last claim1}}
\begin{lem}
\label{l: deriv q}The function $f_{\star}$ of $q$ is defined as
\begin{equation}
f_{\star}(q)=\frac{1}{\beta_{s,\,2}}\Big(\log qm_{2}-\frac{1}{2}\Big)-\frac{1}{8}(qm_{2})^{2}+\frac{1}{4}m_{2}+\frac{251}{2002}\ .\label{e: Def of f_star}
\end{equation}
Then, if $q>e^{8}$, $f_{\star}'(q)=\frac{d}{dq}f_{\star}(q)>0$.
\end{lem}

\begin{proof}[Proof of Lemma \ref{l: Last claim1}]
By Proposition \ref{p: numerical}, we obtain $f_{\star}(6500)>0$.
We observe that $\frac{(q-1)}{8q}<\frac{1}{8}$ and $\frac{-q^{2}+4q+1}{8q(q-1)}<-\frac{251}{2002}$
if $q>1000$. Hence, Lemma \ref{l: deriv q} proves Lemma \ref{l: Last claim1}.
\end{proof}

\subsection{\label{subsec: Last2}Proof of Lemma \ref{l: deriv q}}

Let $s_{2}=s_{2}(q)=qm_{2}$. In the first lemma, we compute $m_{2}'=(d/dq)m_{2}$,
$s_{2}'=(d/dq)s_{2}$, and $\beta_{s,\,2}'=(d/dq)\beta_{s,\,2}$.
\begin{lem}
\label{l: deriv q-1}We have
\begin{align*}
m_{2}'=\frac{d}{dq}m_{2} & =\,-\frac{m_{2}(1-jm_{2}-qjm_{2}^{2})}{q(1-2jm_{2})}\ ,\\
s_{2}'=\frac{d}{dq}s_{2} & =\,=-\frac{js_{2}^{2}(1-s_{2})}{q(q-2js_{2})}\ ,\\
\beta_{s,\,2}'=\frac{d}{dq}\beta_{s,\,2} & =\,\frac{1}{1-s_{2}}\bigg(\beta_{s,\,2}s_{2}'-2\frac{-s_{2}+s_{2}^{2}+qs_{2}'}{(q-js_{2})s_{2}}\bigg)\ .
\end{align*}
\end{lem}

\begin{proof}
We observe that
\[
\beta_{s,\,2}=g_{2}(m_{2})=\frac{2}{1-qm_{2}}\log\frac{1-jm_{2}}{2m_{2}}=\frac{2}{qm_{2}(1-jm_{2})}\ ,
\]
so that
\[
\log(1-jm_{2})-\log2m_{2}=\frac{2}{q}\Big(\frac{1}{2m_{2}}-\frac{1}{1-jm_{2}}\Big)\ .
\]
By differentiating this equation in $q$, we get
\[
\frac{-m_{2}-jm_{2}'}{1-jm_{2}}-\frac{m_{2}'}{m_{2}}=-\frac{2}{q^{2}}\Big(\frac{1}{2m_{2}}-\frac{1}{1-jm_{2}}\Big)+\frac{2}{q}\Big(-\frac{m_{2}'}{2m_{2}^{2}}+\frac{-m_{2}-jm_{2}'}{(1-jm_{2})^{2}}\Big)\ .
\]
By elementary computation, we can write
\begin{equation}
m_{2}'=-\frac{m_{2}(1-jm_{2}-qjm_{2}^{2})}{q(1-2jm_{2})}\ .\label{e: deriv of m_2}
\end{equation}
Let $s_{2}=qm_{2}$. Then,
\begin{equation}
s_{2}'=m_{2}+qm_{2}'=-\frac{js_{2}^{2}(1-s_{2})}{q(q-2js_{2})}\ .\label{e: deriv of s_2}
\end{equation}

Next, we compute $\beta_{s,\,2}'$. Note that
\[
\beta_{s,\,2}=\frac{2}{1-s_{2}}\log\frac{q-js_{2}}{2s_{2}}\ ,
\]
so that
\begin{align}
\beta_{s,\,2}'\, & =\,-\frac{2s_{2}'}{(1-s_{2})^{2}}\log\frac{q-js_{2}}{2s_{2}}+\frac{2}{1-s_{2}}(\frac{1-s_{2}-js_{2}'}{q-js_{2}}-\frac{s_{2}'}{s_{2}})\nonumber \\
 & =\,\frac{1}{1-s_{2}}\bigg(s_{2}'\frac{2}{1-s_{2}}\log\frac{q-js_{2}}{2s_{2}}+2\frac{s_{2}-s_{2}^{2}-js_{2}s_{2}'-qs_{2}'+js_{2}s_{2}'}{(q-js_{2})s_{2}}\bigg)\label{e: deriv of beta_s2}\\
 & =\,\frac{1}{1-s_{2}}\bigg(\beta_{s,\,2}s_{2}'-2\frac{-s_{2}+s_{2}^{2}+qs_{2}'}{(q-js_{2})s_{2}}\bigg)\ .\nonumber 
\end{align}
\end{proof}
The next lemma provides the bound of $m_{2}(q)$.
\begin{lem}
\label{l: deriv q-2}Let $q>e^{8}$. We have
\[
\frac{1}{2q\log q}<m_{2}(q)<\frac{1}{q\log q}\ .
\]
\end{lem}

\begin{proof}
It can be observed that $h_{2}(m_{2})=0$ and $h_{2}(t)>0$ if $m_{2}<t<1/q$.
We claim that
\[
h_{2}(a)\,=\,\log\frac{1-ja}{2a}+\frac{qa-1}{qa(1-ja)}\,>0\ ,
\]
where $a=1/q\log q$. The above inequality can be written as
\[
\log\frac{q\log q-j}{2}>\Big(\frac{q\log q-q}{q\log q-j}\Big)\log q\ .
\]
Since the right-hand side is smaller than $\log q$, it suffices to
show that
\[
\log q+\log\frac{\log q-1+2/q}{2}>\log q\ ,
\]
which is true if $q>e^{3}$. Hence, $m_{2}<1/q\log q$. Next, we have
$m_{2}>(1/2q)\log q$ since
\[
\log\Big(q\log q-\frac{j}{2}\Big)-2\Big\{\frac{q\log q-q/2}{(q\log q-j/2)}\Big\}\log q<0\ ,
\]
which is true if $q>e^{8}$. 
\end{proof}
In the next two lemmas, we prove that some quantities are positive.
\begin{lem}
\label{l: deriv q-3}Let $q>e^{8}$. We have
\[
m_{2}'-s_{2}s_{2}'>0\ .
\]
\end{lem}

\begin{proof}
We have
\begin{align*}
m_{2}'-s_{2}s_{2}' & \,=\,-\frac{m_{2}(1-jm_{2}-jqm_{2}^{2})}{q(1-2jm_{2})}+\frac{js_{2}^{3}(1-s_{2})}{q(q-2js_{2})}\\
 & \,=\,\frac{s_{2}(-1+jm_{2}+jq(q+1)m_{2}^{2}-jq^{3}m_{2}^{3})}{q(q-2js_{2})}\ .
\end{align*}
It suffices to show that
\[
jq(q+1)m_{2}^{2}-jq^{3}m_{2}^{3}-1>0\ .
\]
Since $\frac{1}{2q\log q}<m_{2}<\frac{1}{q\log q}$, we obtain
\begin{align*}
jq(q+1)m_{2}^{2}-jq^{3}m_{2}^{3}-1 & >\frac{jq(q+1)}{4q^{2}(\log q)^{2}}-\frac{jq^{3}}{q^{3}(\log q)^{3}}-1\\
 & =\frac{1}{q(\log q)^{3}}\Big[\frac{(q+1)(q-2)}{4}\log q-q(q-2)-q(\log q)^{3}\Big]\\
 & >\frac{1}{q(\log q)^{3}}[2(q+1)(q-2)-q(q-2)-q(\log q)^{3}]\\
 & =\frac{1}{q(\log q)^{3}}[q^{2}-q(\log q)^{3}-4]>0\ .
\end{align*}
In the second and third inequalities, we use $q>e^{8}$. Hence, $m_{2}'-s_{2}s_{2}'>0$. 
\end{proof}
\begin{lem}
\label{l: deriv q-4}Let $q>e^{8}$. We have
\begin{equation}
\Big(\frac{1}{2}-\log s_{2}\Big)\beta_{s,\,2}'+\beta_{s,\,2}\frac{s_{2}'}{s_{2}}>0\ .\label{e: last ineq2}
\end{equation}
\end{lem}

\begin{proof}
Let $A(q)=\frac{1}{2}-\log s_{2}$. From Lemma \ref{l: deriv q-2},
we obtain
\[
\frac{5}{2}<\frac{1}{2}+\log8<\frac{1}{2}+\log\log q<A(q)<\frac{1}{2}+\log(2\log q)\ ,
\]
and
\begin{align*}
A(q)\beta_{s,\,2}'+\beta_{s,\,2}\frac{s_{2}'}{s_{2}}\, & =\,\frac{s_{2}'}{1-s_{2}}\Big[A(q)\beta_{s,\,2}-\frac{2q}{q-2}\frac{A(q)}{s_{2}^{2}}\Big]+\frac{s_{2}'}{1-s_{2}}\Big[\frac{1-s_{2}}{s_{2}}\beta_{s,\,2}\Big]\\
 & =\,\frac{s_{2}'}{1-s_{2}}\Big[\beta_{s,\,2}\Big(\frac{1}{s_{2}}+A(q)-1\Big)-\frac{2q}{q-2}\frac{A(q)}{s_{2}^{2}}\Big]\ .
\end{align*}
Hence, since $s_{2}'<0$, it suffices to show that
\[
\Big(\frac{2q}{q-2}\Big)\frac{A(q)}{s_{2}^{2}}>\beta_{s,\,2}\Big(\frac{1}{s_{2}}+A(q)-1\Big)=\frac{\beta_{s,\,2}}{s_{2}}[1+(A(q)-1)s_{2}]\ ,
\]
i.e.,
\[
\beta_{s,\,2}<\frac{1}{1+(A(q)-1)s_{2}}\cdot\frac{2q}{q-2}\cdot\frac{A(q)}{s_{2}}\ .
\]

Since, $s_{2}<1/\log q$, the right-hand side is greater than
\begin{align*}
\frac{1}{1+(A(q)-1)s_{2}}\cdot\frac{2qA(q)}{q-2}\log q & >\frac{1}{1+(A(q)-1)s_{2}}\Big(\frac{5q}{q-2}\Big)\log q\\
 & >\frac{5}{1+(A(q)-1)s_{2}}\log q\ ,
\end{align*}
and
\[
\beta_{s,\,2}<g_{2}(1/q\log q)\,=\,\frac{2\log q}{\log q-1}\log\frac{q\log q-(q-2)}{2}<\,\frac{5}{2}\log(q\log q)<\,\frac{15}{4}\log q\ ,
\]
where the last inequality is equivalent to $1/2>\log(\log q)/\log q$
which is true for $q>e^{8}$.

Hence, it is enough to show that
\[
\frac{1}{1+(A(q)-1)s_{2}}>\frac{3}{4}\ ,\ \text{i.e.,}\ \frac{1}{3}>(A(q)-1)s_{2}\ .
\]
Since $0<A(q)<1/2+\log(2\log q)$ and $s_{2}<1/\log q$, we obtain,
for $q>e^{8}$, 
\[
(A(q)-1)s_{2}<\frac{\log(2\log q)-1/2}{\log q}<\frac{1}{3}\ .
\]
\end{proof}
Now, we derive the proof of Lemma \ref{l: deriv q}.
\begin{proof}[Proof of Lemma \ref{l: deriv q}]
 By Lemma \ref{l: deriv q-1},
\[
-s_{2}+s_{2}^{2}+qs_{2}'\,=\,-\frac{js_{2}^{2}(1-s_{2})}{q-2js_{2}}-s_{2}(1-s_{2})\,=\,(q-js_{2})\frac{q}{js_{2}}s_{2}'\ ,
\]
so that
\[
\beta_{s,\,2}'=\frac{1}{1-s_{2}}\Big(\beta_{s,\,2}s_{2}'-2\frac{q}{js_{2}^{2}}s_{2}'\Big)=\frac{1}{1-s_{2}}\Big(\beta_{s,\,2}-\frac{2q}{js_{2}^{2}}\Big)s_{2}'\ .
\]

Now, we return to $f_{\star}(q)$. We have
\begin{align*}
f_{\star}(q)\, & =\,\frac{1}{\beta_{s,\,2}}(\log qm_{2}-\frac{1}{2})-\frac{1}{8}(qm_{2})^{2}+\frac{1}{4}m_{2}+\frac{251}{2002}\\
 & =\,\frac{1}{\beta_{s,\,2}}(\log s_{2}-\frac{1}{2})-\frac{1}{8}(s_{2})^{2}+\frac{1}{4}m_{2}+\frac{251}{2002}\ ,
\end{align*}
so that
\begin{align*}
f_{\star}'(q) & =-\frac{\beta_{s,\,2}'}{\beta_{s,\,2}^{2}}\Big(\log s_{2}-\frac{1}{2}\Big)+\frac{1}{\beta_{s,\,2}}(\frac{s_{2}'}{s_{2}})+\frac{1}{4}(m_{2}'-s_{2}s_{2}')\\
 & =\frac{1}{\beta_{s,\,2}^{2}}\Big[\beta_{s,\,2}'(\frac{1}{2}-\log s_{2})+\beta_{s,\,2}(\frac{s_{2}'}{s_{2}})\Big]+\frac{1}{4}(m_{2}'-s_{2}s_{2}')\ .
\end{align*}

Finally, Lemmas \ref{l: deriv q-3} and \ref{l: deriv q-4} prove
Lemma \ref{l: deriv q}.
\end{proof}

\appendix

\section{\label{sec: Numerical}Some Numerical Computations}

Recall the definition \eqref{e: Def of f_star} of $f_{\star}(\cdot)$.
In this section, we verify several inequalities numerically. Our purpose
is the following proposition. The proof is presented at the end of
this section.
\begin{prop}
\label{p: numerical}The following hold.
\begin{enumerate}
\item For $5\le q\le6500$, we have $F_{\beta_{s,\,2}}({\bf u}_{2})>F_{\beta_{s,\,2}}({\bf v}_{1})$.
\item For $6\le q\le54$, we have $\frac{d}{d\beta}[f_{c}(\beta)-F_{\beta}({\bf u}_{2})]\Big|_{\beta=\beta_{s,\,2}}>0$.
\item $f_{\star}(6500)>0$.
\end{enumerate}
\end{prop}

\subsection*{Bounds of $\beta_{s,\,2}$ , $m_{2}$ and $v_{1}$.}

We will obtain the bounds of $\beta_{s,\,2}$, $m_{2}$, and $v_{1}$.
Fix $q\ge5$ and let $j=q-2$. By gradient descent method, we obtain
the following.
\begin{lyxalgorithm}
\label{alg: beta}We define $\beta_{s,\,2}^{u}$ and $\beta_{s,\,2}^{l}$
in the following way.
\begin{enumerate}
\item $t_{0}\leftarrow1\,/\,(2q-4)$.
\item While $g_{2}'(t_{i})>10^{-6}$ , let $t_{i+1}\,\leftarrow\,t_{i}-g_{2}'(t_{i})/(300q^{2})$
.
\item If $g_{2}'(t_{i})\le10^{-6}$ , let $m_{2}^{*}\leftarrow t_{i}$.
\end{enumerate}
Let $\beta_{s,\,2}^{u}\coloneqq g_{2}(m_{2}^{*})+(36/q)|g_{2}'(m_{2}^{*})|$
and $\beta_{s,\,2}^{l}\coloneqq g_{2}(m_{2}^{*})-(36/q)|g_{2}'(m_{2}^{*})|$
.
\end{lyxalgorithm}

We record $m_{2}^{*}$ in the above algorithm and let
\[
\rho_{m}\,\coloneqq\,g_{2}'(m_{2}^{*})/q\ .
\]

\begin{lyxalgorithm}
\label{alg: m2}We define $m_{2}^{u}$ and $m_{2}^{l}$ in the following
way.
\begin{enumerate}
\item If $h_{2}(m_{2}^{*})\ge0$, let $m_{2}^{u}\coloneqq m_{2}^{*}+\rho_{m}$.
\begin{enumerate}
\item $t_{0}\leftarrow m_{2}^{*}$.
\item While $h_{2}(t_{i})\ge0$, let $t_{i+1}\leftarrow t_{i}-\rho_{m}$.
\item If $h_{2}(t_{i})<0$, let $m_{2}^{l}\coloneqq t_{i}-\rho_{m}$.
\end{enumerate}
\item If $h_{2}(m_{2}^{*})<0$, let $m_{2}^{l}\coloneqq m_{2}^{*}-\rho_{m}$.
\begin{enumerate}
\item $t_{0}\leftarrow m_{2}^{*}$.
\item While $h_{2}(t_{i})\le0$, let $t_{i+1}\leftarrow t_{i}+\rho_{m}$.
\item If $h_{2}(t_{i})>0$, let $m_{2}^{u}\coloneqq t_{i}+\rho_{m}$.
\end{enumerate}
\end{enumerate}
\end{lyxalgorithm}

By Newton method, we approximate $v_{1}$ which satisfies $g_{1}(v_{1})=\beta_{s,\,2}$.
\begin{lyxalgorithm}
\label{alg: v1}We define $v_{1}^{u}$ and $v_{1}^{l}$ in the following
way.
\begin{enumerate}
\item Let $t_{0}\,=\,0.8/q$ and $t_{-1}=0$.
\begin{enumerate}
\item While $|t_{i}-t_{i-1}|>10^{-5}/q$, let $t_{i+1}\leftarrow t_{i}-(g_{1}(t_{i})-\beta_{s,\,2}^{u})/g_{1}'(t_{i})$.
\item If $|t_{i}-t_{i-1}|\le10^{-5}/q$, let $v_{1}^{*}\coloneqq t_{i}$
and $\rho_{v}\coloneqq|t_{i}-t_{i-1}|$.
\end{enumerate}
\item If $g_{1}(v_{1}^{*})>\beta_{s,\,2}^{u}$, let $v_{1}^{u}\coloneqq v_{1}^{*}+\rho_{v}$.
\item If $g_{1}(v_{1}^{*})\le\beta_{s,\,2}^{u}$, let
\begin{enumerate}
\item $a_{0}\leftarrow v_{1}^{*}$.
\item While $g_{1}(a_{i})\le\beta_{s,\,2}^{u}$, let $a_{i+1}\leftarrow a_{i}+\rho_{v}$.
\item If $g_{1}(a_{i})>\beta_{s,\,2}^{u}$, let $v_{1}^{u}\coloneqq a_{i}+\rho_{v}$.
\end{enumerate}
\item If $g_{1}(v_{1}^{*})<\beta_{s,\,2}^{l}$, let $v_{1}^{l}\coloneqq v_{1}^{*}-\rho_{v}$.
\item If $g_{1}(v_{1}^{*})\ge\beta_{s,\,2}^{l}$, let
\begin{enumerate}
\item $b_{0}\leftarrow v_{1}^{*}$.
\item While $g_{1}(b_{i})\ge\beta_{s,\,2}^{l}$, let $b_{i+1}\leftarrow b_{i}-\rho_{v}$.
\item If $g_{1}(b_{i})<\beta_{s,\,2}^{l}$, let $v_{1}^{l}\coloneqq b_{i}-\rho_{v}$.
\end{enumerate}
\end{enumerate}
\end{lyxalgorithm}

\begin{lem}
\label{l: bounds}We have $\beta_{s,\,2}^{l}<\beta_{s,\,2}<\beta_{s,\,2}^{u}$
, $m_{2}^{l}<m_{2}<m_{2}^{u}$ , and $v_{1}^{l}<v_{1}<v_{1}^{u}$
.
\end{lem}

\begin{proof}
From the Taylor's theorem, we obtain
\[
g_{2}(m_{2}+t)=g_{2}(m_{2})+g_{2}'(m_{2}+t^{*})t
\]
for some $t^{*}\in(0,t)$ if $t>0$ or $t^{*}\in(t,0)$ if $t<0$.
Since $h_{2}$ is increasing in the neighborhood of $m_{2}$, we obtain
\begin{align*}
|g_{2}'(m_{2}+t^{*})|\, & =\,\bigg|\frac{2q}{[1-q(m_{2}+t^{*})]^{2}}h_{2}(m_{2}+t^{*})\bigg|\\
 & \le\,\bigg|\frac{2q}{[1-q(m_{2}+t^{*})]^{2}}h_{2}(m_{2}+|t|)\bigg|\,=\,\bigg(\frac{1-q(m_{2}+|t|)}{1-q(m_{2}+t^{*})}\bigg)^{2}|g_{2}'(m_{2}+|t|)|\ .
\end{align*}
Since $m_{2}+t^{*},\,m_{2}<1/(2j)$, we obtain
\[
\frac{1-q(m_{2}+|t|)}{1-q(m_{2}+t^{*})}\,\le\,\frac{1}{1-q(m_{2}+t^{*})}\,\le\,\frac{1}{1-q/(2j)}\,=\,\frac{2q-4}{q-4}\le6\ ,
\]
where the last inequality is from $q\ge5$. Hence, we have
\[
|g_{2}'(m_{2}+t^{*})|\le36|g_{2}'(m_{2}+|t|)|\ ,
\]
so that we have
\[
|\beta_{s,\,2}-g_{2}(m_{2}+t)|\,=\,|g_{2}(m_{2})-g_{2}(m_{2}+t)|\,\le\,|g_{2}'(m_{2}+t^{*})||t|\,\,\le\,\frac{36}{q}|g_{2}'(m_{2}+|t|)|\ ,
\]
which proves the first claim. In the last inequality, we use the fact
that $|t|<1/q$.

Since $h_{2}(t)>0$ if $t>m_{2}$ and $h_{2}(t)<0$ if $t<m_{2}$,
the second claim is true. Finally, since $g_{1}$ is increasing in
the neighborhood of $v_{1}$, the third claim holds.
\end{proof}
We finally prove Proposition \ref{p: numerical}.
\begin{proof}[Proof of Proposition \ref{p: numerical}]
From Lemma \ref{l: bounds}, we obtain
\[
\beta_{s,\,2}^{l}<\beta_{s,\,2}<\beta_{s,\,2}^{u}\ ,\ m_{2}^{l}<m_{2}<m_{2}^{u}\ ,\ \text{and}\ v_{1}^{l}<v_{1}<v_{1}^{u}\ .
\]
 By elementary computation, we have
\begin{align*}
F_{\beta_{s,\,2}}({\bf u}_{2})-F_{\beta_{s,\,2}}({\bf v}_{1}) & \ge\,\frac{1}{4}[q(q-2)\Big(m_{2}^{u}-\frac{1}{q-2}\Big)^{2}-\frac{2}{q-2}]+\frac{1}{\beta_{s,\,2}^{l}}\log m_{2}^{l}\\
 & -\,\frac{1}{2}\Big[q(q-1)\Big(v_{1}^{l}-\frac{1}{q-1}\Big)^{2}-\frac{1}{q-1}\Big]-\frac{1}{\beta_{s,\,2}^{u}}\log v_{1}^{u}\ ,\\
\log q\beta_{s,\,2}+2k_{2}(m_{2}) & \ge\,\log(q\beta_{s,\,2}^{l})+2k_{2}(m_{2}^{u})\ ,\\
f_{\star}(6500) & \ge\,\frac{1}{\beta_{s,\,2}^{l}}\Big(\log qm_{2}^{l}-\frac{1}{2}\Big)-\frac{1}{8}(qm_{2}^{u})^{2}+\frac{1}{4}m_{2}^{l}+\frac{251}{2002}\ .
\end{align*}
The second inequality holds since $k_{2}(\cdot)$ is decreasing according
to \eqref{e: k decrea}. From the numerical computations, we find
that the right-hand sides of the displayed equations are positive
for $5\le q\le6500$, and this completes the proof.
\end{proof}

\section{Proof of \eqref{e_mfe}}
\begin{proof}[Proof of \eqref{e_mfe}]
Since we have
\[
Z_{N}(\beta)=\sum_{\bm{x}\in\Xi}\frac{N!}{(Nx_{1})!\cdots(Nx_{q})!}\exp\{-\beta NH(\bm{x})\}\ ,
\]
we can use the elementary bound
\[
k\log k-k\le\log k!\le(k+1)\log(k+1)-k\ ,
\]
to obtain
\begin{align*}
 & \sum_{\bm{x}\in\Xi}\exp\Big\{-\beta N\Big[-\frac{1}{2}(x_{1}^{2}+\cdots+x_{q}^{2})+\frac{1}{\beta}\sum_{i=1}^{q}\big(x_{i}+\frac{1}{N}\big)\log\big(x_{i}+\frac{1}{N}\big)\Big]-q\log N\Big\}\\
\le & Z_{N}(\beta)\le\sum_{\bm{x}\in\Xi}\exp\Big\{-\beta NF_{\beta}(\bm{x})+\log(N+1)+N\log\big(1+\frac{1}{N}\big)\Big\}\ .
\end{align*}
Hence, by the definition of $F_{\beta}$ \eqref{e: def of F_beta},
we can obtain
\[
\sup_{\bm{x}\in\Xi}\{-F_{\beta}(\bm{x})\}+O\Big(\frac{\log N}{N}\Big)\le\frac{1}{\beta N}\log Z_{N}(\beta)\le\sup_{\bm{x}\in\Xi}\{-F_{\beta}(\bm{x})\}+O\Big(\frac{\log N}{N}\Big)
\]
and the proof is completed. 
\end{proof}
\begin{acknowledgement*}
This work was supported by Samsung Science and Technology Foundation
(Project Number SSTF-BA1901-03).
\end{acknowledgement*}

\end{document}